\documentclass[reqno]{amsart}

\usepackage{amsmath,amsthm,amsfonts,amscd,latexsym,amssymb,amsbsy,dsfont,mathrsfs,color}
\newtheorem{theorem}{Theorem}[section]
\newtheorem{theorem*}{Theorem}
\newtheorem{lemma}[theorem]{Lemma}
\newtheorem{corollary}[theorem]{Corollary}
\newtheorem{proposition}[theorem]{Proposition}
\newtheorem{assumption}[theorem]{ASSUMPTION} %{Assumption}

\theoremstyle{definition}
\newtheorem{definition}[theorem]{Definition}
\newtheorem{example}[theorem]{Example}

\theoremstyle{remark}
\newtheorem{remark}[theorem]{Remark}

\numberwithin{equation}{section}

\newcommand{\C}{\mbb{C}}

\newcommand{\N}{\mbb{N}}

\newcommand{\R}{\mbb{R}}
\newcommand{\OO}{\Omega}

\newcommand{\LL}{\mc{L}}

\newcommand{\K}{\mc{K}}
\newcommand{\mbb}{\mathbb}
\newcommand{\mc}{\mathcal}
\newcommand{\mi}{\mathit}
\newcommand{\mk}{\mathfrak}
\newcommand{\mr}{\mathrm}

\newcommand{\mscr}{\mathscr}
\newcommand{\lra}{\longrightarrow}

\newcommand{\sss}{\scriptscriptstyle}

\newcommand{\vep}{\varepsilon}

\newcommand{\bs}{{\tiny $\blacksquare$}}

\newcommand{\sL}{\mathsf{L}}

\newcommand\ssp{\sigma_{S}}	%spherical spectrum
\newcommand\srho{\rho_{S}}	%spherical resolvent set

\newcommand{\cS}{\mbb{S}}

\newcommand{\rr}{|}
\newcommand{\1}{\mr{I}}
\newcommand{\old}{\oldstylenums{1}}
\newcommand{\cP}{\mc{P}}

\def\sR{{\mathsf R}}
\def\sS{{\mathsf S}}

\def\sH{\mathsf{H}}

\def\bC{{\mathbb C}}           %%%  complex numbers

\def\bH{{\mathbb H}}
\def\bN{{\mathbb N}}

\def\bR{{\mathbb R}}

\def\bS{{\mathbb S}}

\def\gB{{\mathfrak B}}

\def\beq{\begin{equation}}
\def\eeq{\end{equation}}

\def\b{\langle}
\def\k{\rangle}

\def\supp{\mr{supp}}

\DeclareMathOperator{\diff}{d \! \hspace{0.13ex}}
\DeclareMathAlphabet{\mathpzc}{OT1}{pzc}{m}{it}
\newcommand{\Chi}{\mbox{{\Large $\chi$}}}
\newcommand{\Ran}{\mr{Ran}}
\newcommand{\Ker}{\mr{Ker}}

\begin{document}

\title[Spectral representations of normal operators via iqPVMs]{Spectral representations of normal operators \\ in quaternionic Hilbert spaces \\ via intertwining quaternionic PVMs}

%\title[iqPVMs and spectral representations of quaternionic normal operators]{Rigged quaternionic PVMs and spectral representations of normal operators in quaternionic Hilbert spaces}

%\title[iqPVMs and spectral representations of quaternionic normal operators]{Rigged quaternionic PVMs and \\ spectral representations of \\ quaternionic normal operators}

%%%%%%%

\author{Riccardo Ghiloni}
\address{Department of Mathematics, University of Trento, I--38123, Povo-Trento, Italy}
%\curraddr{}
\email{ghiloni@science.unitn.it \\ moretti@science.unitn.it \\ perotti@science.unitn.it}
%\thanks{Work partially supported by GNSAGA and  GNFM of INdAM}

\author{Valter Moretti}
%\address{}
%\curraddr{}
%\email{moretti@science.unitn.it}
%\thanks{}

\author{Alessandro Perotti}
%\address{}
%\curraddr{}
%\email{perotti@science.unitn.it}
%\thanks{}

%%%%%%%


\date{}

%%%%%%%

\begin{abstract}
The possibility of formulating quantum mechanics over quaternionic Hilbert spaces can be traced back to von Neumann's foundational works in 
the thirties. The absence of a suitable quaternionic version of spectrum prevented the full development of the theory.
The first rigorous quaternionic formulation has been started only in 2007 
with the definition of the spherical spectrum of a quaternionic operator based on a quadratic version of resolvent operator. The relevance of this notion is proved by the existence of a quaternionic continuous functional calculus and a theory of quaternionic semigroups relying upon it.
A problem of quaternionic formulation is the description of composite quantum systems in absence of a natural tensor product due to non-commutativity of quaternions. A promising tool towards a solution is a quaternionic projection-valued measure (PVM), making possible a tensor product of quaternionic operators with physical relevance. A notion with this property, called \textit{intert\-wining quaternionic PVM}, is presented here.
This foundational paper aims to investigate the interplay of this new mathematical object and the spherical spectral features of quaternionic generally unbounded normal operators. We discover in particular the existence of other spectral notions  equivalent to the spherical ones, but based on a standard non-quadratic notion of resolvent operator.
\end{abstract}

\maketitle

%%%%%%%

\setcounter{tocdepth}{1}

\tableofcontents

%%%%%%%

\section{Introduction}

\subsection{Physical motivations for intertwining quaternionic PVMs}
Quantum theories can be formulated either in \textit{real}, \textit{complex} or \textit{quaternionic} Hilbert spaces. This fundamental result has a long history just initiated with von Neumann's mathe\-matical formulation of quantum mechanics \cite{vN} in 1932 and ended more than sixty years later with the celebrated theorem by Sol\`er \cite{soler}, after a crucial result by Piron \cite{piron}. As a matter of fact, the whole physical phenomenology of the testable elementary YES-NO propositions of a given quantum system is encoded \cite{BC,libroQL} in the theore\-tical properties of an abstract orthocomplemented lattice, whose ordering relation is the logical implication, which is {\em complete}, {\em orthomodular}, {\em atomic} and {\em irreducible} (in the presence of superselection rules the last requirement may be relaxed). Although the orthomodularity imposes a weak distributivity, the lattice is non-distributive just in view of non-commutativity of quantum observables constructed out of the lattice -- roughly speaking the Heisenberg principle -- and therefore it does not admit an interpretation in terms of standard Boolean logic though it was the starting point of the so-called quantum logics \cite{BC,libroQL}, after the first formulation by Birkhoff and von Neumann \cite{BvN}. Piron \cite{piron} established that if the lattice includes at least four orthogonal minimal elements called atoms (this requirement can be formulated into essentially equivalent different ways), the lattice is isomorphic to the one of certain subspaces of a generalized  vector space $\sH$ over a division ring $\mathbb B$ equipped with a natural notion of involutory anti-isomorphism and an associated generalized non-singular Hermitian scalar product. Maeda-Maeda \cite{maeda} established that orthomodularity fixed these subspaces as the ones which are orthogonally closed with respect to the said scalar product (i.e., those satisfying $M=(M^\perp)^\perp$). The final step, due to Sol\`er \cite{soler} and successively improved by Holland \cite{holland} and by Aerts and van Steirteghem \cite{AvS}, concluded this investigation proving that if the vector space $\sH$ includes an infinite orthonormal system or it satisfies some similar requirement, then it must be a true Hilbert space, with $\mathbb B$ equal to either $\bR$, $\bC$ or $\bH$ (the ring of quaternions), no further possibility is permitted. In this case, the lattice of elementary propositions is faithfully represented by the lattice of closed subspaces of $\sH$, the order relation being the standard inclusion relation. Equivalently, the lattice is the one of orthogonal projectors. The Hilbert space $\sH$ turns out to be separable if the initial quantum lattice is. There is no chance to get rid of the cases $\bR$ and $\bH$ as Varadarajan \cite{varadarajan} demonstrated that all fundamental results on the basic theoretical formulation of quantum mechanics (like Gleason's and Wigner's theorems and the fact that observables are self-adjoint operators) survive the replacement of $\bC$ for $\bR$ and $\bH$. 
The real alternative seems not to be very interesting with respect to the quaternionic one. Many efforts have been indeed spent to investigate the differences of the quaternionic version with respect to the standard formulation in complex Hilbert spaces \cite{adler,Emch,finkelstein,hb}. However, this investigation has been performed in absence of a clearly defined notion of spectrum of a quaternionic (unbounded) operator, and very often making use of a Dirac bra-ket type formalism, even when that approach is evidently untenable. There are~no doubts that the obtained results are physically meaningful, but their rigorous mathemati\-cal formulation still does not exist. A mathematically sound notion of spectrum of a quaternionic operator, the \textit{spherical spectrum}, was proposed quite recently by Colombo, Gentili, Sabadini and Struppa in 2007 \cite{CGSS2007} (see also the book \cite{libroverde}). The continuous functional calculus for bounded normal operators in a quaternionic Hilbert space, which is the first step towards a full-fledged spectral theory, was founded even more recently by the authors of the present paper in \cite{GhMoPe}. The development of the related theory of quaternionic groups and semigroups has been started recently too \cite{ACS,CS,GR,GR2}.

Many theoretical issues remain open for the  non-standard formulations of quantum mechanics
either in quaternionic and real Hilbert spaces. One is the failure of the quantum version of Noether theorem. Essentially due to Wigner's theorem, continuous symmetries are pictured in terms of strongly continuous one-parameter groups of unitary operators. Stone theorem says that these groups admit a unique anti self-adjoint generator. Adopting the complex formulation, any anti self-adjoint operator $S$ is one-to-one related with a self-adjoint operator $iS$, thus representing an observable. This observable is the {\em conserved quantity} associated with the continuous symmetry (in Noetherian sense) as soon as the group of symmetries commutes with the time evolution. Even if in real and quaternionic formulations observables are described by (generally unbounded) self-adjoint operators, the relation between anti self-adjoint operators and self-adjoint operators is by no means obvious since, differently from the complex case, there is no standard imaginary unit. In the real case, there is no imaginary units at all in the given mathematical structure and this suggests that physics somehow provides a complex structure restoring the standard formulation. This motivates the lack of interest in the real formulation of quantum mechanics. In the quaternionic case a plethora of imaginary units exists given by unitary anti self-adjoint operators commuting with the said operator $S$, each leading to different physics~\cite{adler}. The solution of this intriguing problem is again expected to have a physical nature: physical laws seem to pick out a complex structure, possibly in common with all physically meaningful anti self-adjoint operators. Nevertheless this issue proves that the standard machinery of spectral theorem for (generally unbounded) self-adjoint operators is not enough in the quaternionic setting. A more sophisticated tool is necessary to encompass the self-adjoint and anti self-adjoint cases, but also the unitary case. In~other words, one needs a spectral theorem for {\em unbounded normal operators}, which makes use of the recently constructed mathematical tools related to the notion of {\em spherical spectrum}. The validity of such a version of the spectral theorem was established by Alpay, Colombo and Kimsey~in~\cite{ACK}.

Another very important question shows that such a spectral theorem is again not enough to handle the physics of the quaternionic formulation. It is a long standing problem the search for a physically sound description of {\em composite quantum systems}. If a system is made of two parts respectively described in two Hilbert spaces, the composite system is described in a third Hilbert space which contains the two others. The inclusion must be such to permit the description of {\em quantum entanglement} phenomena. A natural structure allowing it is the Hilbert tensor product of the two factors. Whereas the tensor product is possible as a standard notion in the real and complex cases, the noncommutativity of $\bH$ turns out to be an insurmountable obstruction for the quaternionic formulation. Barring apparently awkward solutions which are not based on the tensor product structure, a possible and  natural way out is to equip the vector space with another multiplication law on the opposite side (the scalar multiplication of ours quaternionic Hilbert spaces is a right multiplication). Indeed, quaternionic two-sided modules permit the notion of tensor product. The problem is that Sol\`er's theorem provides just one scalar multiplication rule: the right-one. The left-one should be somehow imposed by physics in such a way that the tensor product between all operators of the composite quantum system representing observables is compatible with the tensor product between the Hilbert spaces. In other words, such operators have to be linear, not only with respect to the initial right scalar multiplication, but also with respect to the new left scalar multiplication.

Our \textit{idea} is that, given a physical quantum system described in a quaternionic Hilbert space $\sH$, at least the self-adjoint operators representing observables should admit an operatorial representation of $\bH$ commuting with them and with their spectral measures. In other words, there must exist a left scalar multiplication $\bH \ni q \mapsto L_q$ of $\sH$ (maybe a family of such multiplications), which commutes with all physically relevant spectral measures. In particular, an element of this left scalar multiplication could be used to tackle the problem of the complex structure above mentioned in relation with the quantum version of Noether theorem. The new left scalar multiplication, together with the original right-one, provides $\sH$ with a structure of a two-sided quaternionic Hilbert module, permitting tensor pro\-ducts. As we said above, within this approach, the natural notion of \textit{quaternionic projection-valued measure} (\textit{qPVM} for short), necessary to formulate a possibly physically useful version of the spectral theorem, has to include a compatible left scalar multiplication, commuting with all the orthogonal projectors of the measure. This commutativity implies the linearity of such orthogonal projectors also with respect to the left scalar multiplication.
 
The pairs formed by qPVMs and by left scalar multiplications commuting with them are new mathematical objects we introduce and study here for the first time. We call such pairs {\em intertwining quaternionic projection-valued measures} or simply \textit{iqPVMs} for short. The physics involved in this new notion, if any, deserves further investigation. This foundational paper is instead devoted to investigate just the mathematical features of the introduced structures which appear to be remarkable in their own mathematical right, like Theorem E stated below.

%%%

\subsection{iqPVMs and main results}

\textit{Fix a quaternionic Hilbert space $\sH$} with right scalar multiplication $\sH \times \bH \ni (u,q) \mapsto uq \in \sH$. Denote by $\gB(\sH)$ the set of all bounded right $\bH$-linear operators $T:\sH \to \sH$% and with Hermitian quaternionic scalar product $\sH \times \sH \ni (u,v) \mapsto \b u|v \k \in \bH$. Denote by $\gB(\sH)$ the set of all bounded right $\bH$-linear operators $T:\sH \to \sH$.
, equipped with its natural structure of real Banach $C^*$-algebra whose product is the composition and whose $^*$-involution is the adjunction $T \mapsto T^*$.

The concept of qPVM can be modelled on that of complex PVM. Consider a second-countable Haudorff space $X$ and its Borel $\sigma$-algebra $\mscr{B}(X)$. We call orthogonal projector of $\sH$ an operator $T \in \gB(\sH)$ which is idempotent $TT=T$ and self-adjoint $T^*=T$. A \textit{qPVM over $X$ in $\sH$} is a map $P:\mscr{B}(X) \to \gB(\sH)$, which assigns to each Borel subset $E$ of $X$ an orthogonal projector $P(E)$ of $\sH$ in such a way that $P(X)$ is the identity map $\1$ on $\sH$, $P(E \cap E')=P(E)P(E')$ for every $E,E' \in \mscr{B}(X)$ and, given any sequence $\{E_n\}_{n \in \N}$ of pairwise disjoint Borel subsets of $X$, $P(\bigcup_{n \in \N}E_n)u$ is equal to the sum of the series $\sum_{n \in \N}P(E_n)u$ for every $u \in \sH$. The support of $P$ is the largest closed subset $\Gamma$ of $X$ such that $P(X \setminus \Gamma)=0$.  

We remind the reader that a left scalar multiplication of $\sH$ is a highly non-intrinsic operation depending on the choice of a preferred Hilbert basis $N$ of $\sH$. Given any $q \in \bH$, one defines the operator $L_q \in \gB(\sH)$ by requiring that $L_qz=zq$ for every $z \in N$. %The family of the $L_q$'s
The map $\LL:\bH \to \gB(\sH)$, sending $q$ into $L_q$, determinates a left scalar multiplication $\bH \times \sH \ni (q,u) \mapsto qu \in \sH$ by setting $qu:=L_qu$. For convenience, we call $\LL$ itself \textit{left scalar multiplication of $\sH$}. Such a multiplication induces a structure of quaternionic two-sided Banach $C^*$-algebra on $\gB(\sH)$ by considering the left and the right scalar multiplications $(q,T) \mapsto qT:=L_qT$ and $(T,q) \mapsto Tq:=TL_q$.

We introduce the new concept of \textit{iqPVM over $X$ in $\sH$} defining it as a pair $\cP=(P,\LL)$, where $P$ is a qPVM over $X$ in $\sH$ and $\LL$ is a left scalar multiplication $\bH \ni q \mapsto L_q$ of $\sH$ commuting with $P$ in the sense that $P(E)L_q=L_qP(E)$ for every $E \in \mscr{B}(X)$ and for every $q \in \bH$.

A first striking property of iqPVMs is described by the possibility of defining the operatorial integral of all \textit{$\bH$-valued} bounded Borel measurable functions on~$X$, consistently with the quaternionic two-sided Banach $C^*$-algebra structure of $\gB(\sH)$. If~$s=\sum_{\ell=1}^nS_\ell\Chi_{E_n}$ is a simple function with $S_\ell \in \bH$ and $E_\ell \in \mscr{B}(X)$, then $\int_Xs\diff\cP$ is defined as the operator $\sum_{\ell=1}^nL_{S_\ell}P(E_n)$ in $\gB(\sH)$. By a standard density procedure, this integral extends to all elements of $M_b(X,\bH)$, the set of $\bH$-valued bounded Borel measurable functions on $X$. Equip the latter set with the supremum norm and with its natural pointwise defined structure of quaternionic two-sided Banach $C^*$-algebra. The commutativity between $P$ and $\LL$ implies that the map $M_b(X,\bH) \ni \varphi \mapsto \int_X\varphi\diff\cP \in \gB(\sH)$ is a continuous quaternionic two-sided $C^*$-algebra homomorphism. It is worth noting that the integral $\int_X\varphi\diff\cP$ makes sense also when $\varphi$ is an arbitrary $\bH$-valued function whose restriction to the support of $P$ is bounded and measurable.   

Our first main result is a spectral theorem for bounded normal operators.

\textit{Fix~an imaginary unit $\imath$ in $\bH$.} Indicate by $\C^+_\imath$ the slice half-plane determined by~$\imath$; that is, $\C^+_\imath:=\{\alpha+\imath\beta \in \bH \, | \, \alpha,\beta \in \R, \beta \geq 0\}$. %, equipped with the Euclidean topology inherited by $\bH$.
For every given normal operator $T \in \gB(\sH)$, we prove the existence of an iqPVM $\cP=(P,\LL)$ over $\C^+_\imath$ in $\sH$ such that $P$ has bounded support and $T$ admits the integral representation $T=\int_{\C^+_\imath}\mi{id}\diff\cP$, where $\mi{id}:\C^+_\imath \hookrightarrow \bH$ denotes the inclusion map. The qPVM $P$ is unique, while $\LL$ is not. We stress that this non-uniqueness might have a physical meaning in the context of composite quaternionic quantum systems, as described above. With regard to this lack of uniqueness, we are able to give a complete characterization of all left scalar multiplications $\LL'$ of $\sH$ for which the pair $\cP'=(P,\LL')$ is an iqPVM and $T=\int_{\C^+_\imath}\mi{id}\diff\cP'$. Furthermore, we describe explicitly the relation between the support of $P$ and the spherical spectrum $\ssp(T)$ of $T$.

Our result can be stated as follows.

\vspace{1em}

\noindent \textbf{Theorem A (Bounded iq Spectral Theorem).} \textit{Given any normal operator $T$ in $\gB(\sH)$, there exists an iqPVM $\cP=(P,\LL)$ over $\C^+_\imath$ in $\sH$ such that}
\beq \label{eq:bounded-iq-sp-teo}
T=\int_{\C^+_\imath}\mi{id}\diff\cP.
\eeq

\textit{The qPVM $P$ is uniquely determined by $T$ on the whole $\mscr{B}(\C^+_\imath)$, while $L_q:=\LL(q)$ is uniquely determined by $T$ on  $\Ker(T-T^*)^\perp$ if $q \in \C_\imath$.}

\textit{The following additional facts hold.}
\begin{itemize}
 \item[$(1)$] \textit{A left scalar multiplication $\LL$ defines an iqPVM $\cP=(P,\LL)$ satisfying equa\-lity (\ref{eq:bounded-iq-sp-teo}) if and only if %$L_\imath'T=TL_\imath'$, $L_\jmath'T=T^*L_\jmath'$ for some imaginary unit $\jmath$ with $\imath\jmath=-\jmath\,\imath$ and $-L_\imath'(T-T^*)$ is a positive operator.} 
 it verifies the next three conditions}
 \begin{itemize}
  \item[$\bullet$] \textit{$L_\imath T=TL_\imath$.}
  \item[$\bullet$] \textit{$L_\jmath T=T^*L_\jmath$ for some imaginary unit $\jmath$ of $\bH$ with $\imath\jmath=-\jmath\,\imath$.}
  \item[$\bullet$] \textit{$-L_\imath(T-T^*)$ is a positive operator.}
 \end{itemize}
 \item[$(2)$] \textit{The support of $P$ is equal to $\ssp(T) \cap \C^+_\imath$. A point $q$ of $\C^+_\imath$ belongs to the spherical point spectrum $\sigma_{pS}(T)$ of $T$ if and only if $P(\{q\}) \neq 0$. The point $q \in \C^+_\imath$ is an element of the continuous spherical spectrum $\sigma_{cS}(T)$ of $T$ if and only if $P(\{q\})=0$ and $P(U) \neq 0$ for every open neighborhood $U$ of $q$ in $\C^+_\imath$. Finally, the residual spherical spectrum $\sigma_{rS}(T)$ of $T$ is empty.}
\end{itemize}

\vspace{1em}

The proof of this theorem is based on the properties of the above-mentioned operatorial integral and on the machinery developed in our previous paper on this topic \cite{GhMoPe}.

In \cite[Thm 4.7]{ACK} the authors prove two versions of formula \eqref{eq:bounded-iq-sp-teo}, including the uniqueness~of~$P$. The qPVM $P$ and the corresponding integral are constructed in a completely different way taking advantage of Riesz representation theorem. Such an integral is complex in nature, in contrast with our whose nature is fully quaternionic as we have seen above. Actually, an easy verification shows that the integral defined in \cite[Sect.~IV]{ACK} coincides with our integral performed only over \textit{$\C_\imath$-valued} bounded Borel measura\-ble functions. We underline that, in \cite{ACK}, no reference is made to notions similar to our iqPVMs, and hence to point $(1)$. Also questions regarding the relation between $P$ and the spherical spectrum of $T$, as in our point $(2)$, are not treated in~\cite{ACK}.

The case of unitary operators and the one of compact normal operators are considered in \cite{SU} and \cite{GhMoPe2}, respectively. Some versions of the spectral theorem obtained without a clearly defined notion of spectrum can be found in \cite{finkelstein,sharma,viswanath}. In the finite-dimensional case, where the spherical spectrum is the set of right eigenvalues, the spectral theorem for normal operators is well known \cite{Jacobson39} (see \cite{FarenickPidkowich} for an ample exposition and further references).

Our second main result is a spectral quaternionic $L^2$-representation theorem in the spirit of Dunford and Schwartz \cite[Thm X.5.2]{DS}. %Given a positive measure $\nu$ on $X$, we denote by $L^2(X,\bH;\nu)$ the 

\vspace{1em}

\noindent \textbf{Theorem B (iq Spectral $\boldsymbol{L^2}$-representation).} \textit{Let $T \in \gB(\sH)$ be a normal operator and let $\cP=(P,\LL)$ be an iqPVM over $\C^+_\imath$ in $\sH$ satisfying (\ref{eq:bounded-iq-sp-teo}). Then there exist a family $\{\nu_\alpha\}_{\alpha \in \mc{A}}$ of positive $\sigma$-additive finite Borel measures over the support $\Gamma$ of $P$ and an isometry $V:\sH \to \bigoplus_{\alpha \in \mc{A}} L^2(\Gamma,\bH;\nu_\alpha)$ such that, for every $\varphi \in M_b(\Gamma,\bH)$, the operator $\varphi(T):=\int_{\C^+_\imath}\varphi\diff\cP$ is multiplicatively represented by means of $V$ in the sense that}
\[
(V \varphi(T) V^{-1})(\oplus_{\alpha \in \mc{A}} f_\alpha)=\oplus_{\alpha \in \mc{A}} \,\, \varphi f_\alpha \quad \text{\textit{$\nu_\alpha$-a.e. in $\Gamma$,}}
\]
\textit{where $\varphi f_\alpha$ denotes the pointwise product between $\varphi$ and $f_\alpha$. In particular, if $L_q:=\LL(q)$ for $q \in \bH$, then}
\[
(V L_q V^{-1})(\oplus_{\alpha \in \mc{A}} f_\alpha)=\oplus_{\alpha \in \mc{A}} \,\, qf_\alpha \quad \text{\textit{$\nu_\alpha$-a.e. in $\Gamma$.}}
\]

\vspace{1em}

This theorem depends on an algebraic characterization of left scalar multiplications, which is interesting in its own right. Such a characterization is known to physicists, but it is used without any proof, see the foundational paper \cite{finkelstein}. Here we give a proof. Equip $\bH$ with its natural structure of real $^*$-algebra induced by the quaternionic conjugation $q \mapsto \overline{q}$. The mentioned characterization reads as follows. 

\vspace{1em}

\noindent \textbf{Theorem C (algebraic character of left s.\ m.).}
\textit{A map $\LL:\bH \to \gB(\sH)$ is a left scalar multiplication of $\sH$ if and only if it is a real $^*$-algebra homomorphism.}

\vspace{1em}

The integral $\int_X \varphi\diff\cP$ of a function $\varphi \in M_b(X,\bH)$ with respect to a given iqPVM $\cP=(P,\LL)$ can be extended to all  %the elements of $M(X,\bH)$, the set of possibly unbounded measurable
possibly unbounded $\bH$-valued Borel measurable functions on $X$. This can be done by means of an accurate review of the complex construction. The extended integral permits to generalize the bounded iq Spectral Theorem to the unbounded case.

\vspace{1em}

\noindent \textbf{Theorem D (iq Spectral Theorem).} \textit{Let $T:D(T) \to \sH$ be a closed normal operator with dense domain. Then there exists an iqPVM $\mc{Q}=(Q,\LL)$ over $\C^+_\imath$ in $\sH$ such that} 
\beq \label{eq:unbounded-iq-sp-teo}
T=\int_{\C^+_\imath}\mi{id}\diff\mc{Q}.
\eeq

\textit{The qPVM $Q$ is uniquely determined by $T$ on the whole $\mscr{B}(\C^+_\imath)$, while $L_q:=\LL(q)$ is uniquely determined by $T$ on  $\Ker(\,\overline{T-T^*}\,)^\perp$ if $q \in \C_\imath$. Points (1) and (2) of Theorem A continue to hold without changes, except for replacing $\cP$ for $\mc{Q}$.}

\vspace{1em}

%In order to proof this result we
The proof of this result consists of reducing to the bounded case by means of a transformation sending unbounded normal operators to bounded normal ones. The~same strategy was used in \cite[Thm 6.2]{ACK} to prove two versions of formula \eqref{eq:unbounded-iq-sp-teo} and the uniqueness of $Q$, under the additional assumption $D(T^*)=D(T)$. Instead, we will find that equality as a by-product of our preceding iq Spectral Theorem. Also in this situation, the operatorial integral exploited in \cite{ACK} corresponds to our integral computed only over possibly unbounded \textit{$\C_\imath$-valued} Borel measurable functions. 

The iq Spectral Theorem D reveals a completely new, and surprisingly deep, connection between the spherical spectrum $\ssp(T)$ of $T$, which is defined in terms of the quadratic operator $\Delta_q(T):= T^2 -T(2\mr{Re}(q)) + \1|q|^2$, and the classical-like notion of spectrum of $T$ defined by means of any fixed left scalar multiplication associated with $T$ itself via the mentioned theorem. Given an operator $T:D(T) \to \sH$ and a left scalar multiplication $\LL(q)=L_q$ of $\sH$, the left resolvent set $\rho_\LL(T)$ of $T$ w.r.t.\ $\LL$ is the subset of $\bH$ formed by all quaternions $q$ such that $T-L_q:D(T) \to \sH$ is bijective and the resolvent operator $\sR_q(T):=(T-L_q)^{-1}$ is bounded. The \textit{left spectrum $\sigma_\LL(T)$ of $T$ w.r.t.\ $\LL$} is defined by $\sigma_\LL(T):=\bH \setminus \rho_\LL(T)$.

\vspace{1em}

\noindent \textbf{Theorem E (spherical and left spectrum).} \textit{Let $T:D(T) \to \sH$ be a closed normal operator with dense domain and let $\mc{Q}=(Q,\LL)$ be an iqPVM over $\C^+_\imath$ in $\sH$ satisfying (\ref{eq:unbounded-iq-sp-teo}). Then it holds}
\[
\sigma_\LL(T)=\ssp(T) \cap \C^+_\imath
\]
\textit{and}% $\sR_q(T)=\int_{\C^+_\imath}(z-q)^{-1}\diff\mc{Q}(z)$ for every $q \in \rho_\LL(T)$.}
\[
\sR_q(T)=\int_{\C^+_\imath}(z-q)^{-1}\diff\mc{Q}(z)\quad \text{ \textit{if} $q \in \rho_\LL(T)$}.
\]

\textit{Moreover, if $q \in \C^+_\imath$, the following additional facts hold.}
\begin{itemize}
 \item \textit{$q$ is a left eigenvalue of $T$ w.r.t.\ $\LL$ (i.e. $\Ker(T-L_q) \neq \{0\}$) if and only if it is a right eigenvalue of $T$ (i.e. $Tu=uq$ for some $u \in D(T) \setminus \{0\}$). If~$u$~is a right eigenvector of $T$ associated with $q$ (i.e. $Tu=uq$, $u \in D(T) \setminus \{0\}$), then it is also a left eigenvector of $T$ w.r.t.\ $\LL$ associated with $q$ (i.e. $Tu=L_qu$).}
 \item \textit{If $\Ker(T-L_q)=\{0\}$, then $\Ran(T-L_q)$ is dense in $\sH$. In other words, the left residual spectrum of $T$ w.r.t.\ $\LL$ is empty, as the spherical residual one.}
 \item \textit{$q$ belongs to the left continuous spectrum of $T$ w.r.t.\ $\LL$ (i.e. $\Ker(T-L_q)=\{0\}$ and $\sR_q(T)$ is not bounded) if and only if it belongs to the spherical continuous spectrum of $T$.}
\end{itemize}

\vspace{1em}

As a final application of the iq Spectral Theorem D, we establish the existence of a slice-type decomposition for unbounded closed normal operators. The bounded case was treated in \cite{GhMoPe} (though a  version of this result for the bounded case already appeard in the original paper by Teichm\"uller  \cite{T}), see Theorem J and the physical significance of this kind of decomposition in the introduction of that paper.  

\vspace{1em}

\noindent \textbf{Theorem J (the unbounded case).} \textit{Given any closed normal operator $T:D(T) \to \sH$ with dense domain, there exist three operators $A:D(A) \to \sH$, $B:D(B) \to \sH$ and $J \in \gB(\sH)$ such that}
\begin{itemize}
 \item \textit{$T=A+JB$,}
 \item \textit{$A$ is self-adjoint and $B$ is positive and self-adjoint,}
 \item \textit{$J$ is anti self-adjoint and unitary,}
 \item \textit{$\overline{AB}=\overline{BA}$, $AJ=JA$ and $BJ=JB$.}
\end{itemize}
\textit{The operators $A$ and $B$ are uniquely determined by $T$ via the formulas $A=\overline{(T+T^*)\frac{1}{2}}$ and $B=\big|\,\overline{(T-T^*)\frac{1}{2}}\,\big|$. The operator $J$ is determined by $T$ on $\Ker(\,\overline{T-T^*}\,)^\perp$.} 

%\vspace{1em}

%%%

\subsection{Structure of the paper} The paper is organized as follows.

In Section \ref{sec:preliminaries}, we recall some basic notions and results necessary throughout the paper. We briefly deals with quaternionic Hilbert spaces, their operators and their complex subspaces, quaternionic two-sided Banach $C^*$-algebras, spherical spectra, and slice-type decomposition of bounded normal operators, including an extended version of the mentioned Theorem J of \cite{GhMoPe}.

In Section \ref{sec:iqPVMs}, we introduce and study the concepts of iqPVM and of operatorial integral of $\bH$-valued bounded Borel measurable functions with respect to an iqPVM. For short and without loss of generality, we deal with these notions on the second-countable Hausdorff spaces $\C^+_\imath$ only.

Section \ref{sec:bounded-sp-teo} is devoted to the proof of above Theorems A, B and C in their complete form. The section includes some elementary, but significant, examples regarding the bounded iq Spectral Theorem A. Particularly enlightening are the matricial~ones.  

In Section \ref{sec:unb-int}, we introduce and accurately study the concept of operatorial integral of possibly unbounded $\bH$-valued Borel measurable functions on $\C^+_\imath$ with respect to an iqPVM.

In Section \ref{sec:spectral-unb} we present the proof of the complete version of the iq Spectral Theorem D (see Theorem \ref{spectraltheoremU} and Proposition \ref{prop:unb-sp-teo}). We give also some consequences of such a proof, which are of general interest. For example, we show that, if $T:D(T) \to \sH$ is a closed normal operator with dense domain, then $T^*=U^*TU$ for some unitary operator $U:\sH \to \sH$ (observe that this is false in the complex case), $D(T^*)=D(T)$ and $\Ker(T)=\Ker(T^*)=\Ran(T)^\perp$.

Section \ref{sec:applications} contains the proofs of Theorems E and J in their complete form, and it includes an explicit matricial example, which shows Theorem E at work.

\vspace{1em}

\textbf{Acknowledgements.} We would like to thank our colleague Igor Khavkine for his precious advice concerning the relations between composite quantum systems and tensor product in the quaternionic setting. This work is supported by INFN-TIFPA, by GNSAGA and GNFM of INdAM, by the grants FIRB ``Differential Geometry and Geometric Function Theory'' and PRIN ``Variet\`a reali e complesse: geometria, topologia e analisi armonica'' of the Italian Ministry of Education.

%%%%%%%

\section{Preliminaries} \label{sec:preliminaries}

In this section, we recall some basic concepts and results concerning functional analysis over quaternions. For details we refer the reader to \cite{GhMoPe}.

\subsection{Quaternionic Hilbert spaces}%, operators and left multiplications}

Let $\bH$ be the noncommutative real division algebra of quaternions. Given $q=q_0+q_1i+q_2j+q_3k \in \bH$ with $q_0,q_1,q_2,q_3 \in \R$, we denote by $\mr{Re}(q)=q_0$ the real part of $q$, by $\mr{Im}(q)=q_1i+q_2j+q_3k$ the imaginary part of $q$, by $\overline{q}=q_0-q_1i-q_2j-q_3k$ the conjugate of $q$ and by $|q|=\sqrt{q_0^2+q_1^2+q_2^2+q_3^2}$ the norm of $q$. 
The quaternion $q$ is called imaginary unit if $q^2=-1$ or, equivalently, if $\mr{Re}(q)=0$ and $|q|=1$.
Indicate by $\bS$ the set of all imaginary units of $\bH$ and by $\R^+$ the set of all non-negative real numbers. For every $\imath \in \bS$, we define the complex plane $\C_\imath:=\{\alpha+\imath\beta \in \bH \, | \, \alpha,\beta \in \bR\}$.

A \textit{quaternionic (right) Hilbert space} is an Abelian group $(\sH,+)$ equipped with a right scalar multiplication $\sH \times \bH \ni (u, q) \mapsto uq \in \sH$ and a function $\sH \times \sH \ni (u,v) \mapsto \b u|v \k \in \bH$, %having the following three properties: 
called Hermitian quaternionic scalar product of $\sH$, having the following pro\-perties:% hold:
\begin{itemize}
 \item $(u+v)q=uq+vq$, $v(p+q)=vp+vq$, $v(pq)=(vp)q$, and $u1=u$,
 \item $\b u| vp+wq \k = \b u|v\k p+\b u|w\k q$, $\b u|v \k = \overline{\b v | u \k}$ and $\b u| u \k \in \R^+$% and $\b u | u \k=0 \Longrightarrow u=0$,%$u=0$ if $\b u | u \k=0$,
\end{itemize}
for every $p,q \in \bH$ and for every $u,v,w \in \sH$, and
\begin{itemize}
 \item %viewing $\sH$ as a real vector space,
 the function $\|\cdot\|:\sH \to \R^+$, defined by setting $\|u\|:=\sqrt{\b u| u \k}$, is a norm making $\sH$ a real Banach space. 
\end{itemize}

%We make an assumption that will hold throughout the remaining part of the paper. 

\begin{assumption} \label{assumpt:1}
In what follows, all the quaternionic Hilbert spaces $\sH$ we consider will be non-zero; namely, $\sH \neq \{0\}$. 
\end{assumption}

Given a subset $N$ of $\sH$, we denote by $N^\perp$ the (right $\bH$-linear) subspace of $\sH$ consisting of all vectors $u$ such that $\b u|z \k=0$ for every $z \in N$. Every quaternionic Hilbert space $\sH$ has a Hilbert basis; that is, a subset $N$ such that $\|z\|=1$ if $z \in N$, $\b z | z' \k=0$ if $z,z' \in N$ with $z \neq z'$, and $N^\perp=\{0\}$. The latter condition is equivalent to say that, for every $u \in \sH$, the series $\sum_{z \in N}z\b z | u \k$ converges absolutely to $u$ in $\sH$. 

A (right $\bH$-linear) operator $T:D(T) \to \sH$ of $\sH$ is a map such that $T(u+v)=T(u)+T(v)$ and $T(up)=T(u)p$ for every $u,v \in D(T)$ and for every $p \in \bH$, where the domain $D(T)$ of $T$ is a subspace of $\sH$. Denote by $\Ran(T)$ the ima\-ge of $T$ and by $\Ker(T)$ its kernel $T^{-1}(0)$. %$T(up+vq)=T(u)p+T(v)q$ for every $u,v \in \sH$ and for every $p,q \in \bH$. %
 Define $\|T\|:=\sup_{u \in D(T) \setminus \{0\}}\|Tu\|/\|u\|$. The operator $T$ is continuous if and only if it is bounded; that is, $\|T\|<+\infty$. %Denote by $\gB(\sH)$ the set of all continuous operators $T:\sH \to \sH$.
The operator $T$ is said to be closed if its graph $D(T) \oplus \Ran(T)$ is closed in $\sH \times \sH$, equipped with the product topology. If $T$ has a closed operator extension, then $T$ is called closable and the closure $\overline{T}$ of $T$ is its smallest closed extension.

Let $\gB(\sH)$ be the set of all continuous operators $T:\sH \to \sH$. Such a set becomes a real Banach algebra with unity if we define its sum as the pointwise sum, its product as the composition, its real scalar multiplication $\gB(\sH) \times \R \ni (T,r) \mapsto Tr \in \gB(\sH)$ as $(Tr)(u):=T(u)r$ if $u \in \sH$ and its unity as the identity map $\1:\sH \to \sH$ of $\sH$.

Let $T: D(T) \to \sH$ and $S: D(S) \to \sH$ be operators. We write $T \subset S$ if $D(T) \subset D(S)$ and $S \rr_{D(T)}=T$. As usual, $T=S$ if and only if $D(T)=D(S)$ and $T(u)=S(u)$ if $u \in D(T)$. The sum $T+S:D(T+S) \to \sH$ and the composition $TS:D(TS) \to \sH$ are defined in the standard pointwise way:
\begin{itemize}
 \item $D(T+S):= D(T) \cap D(S)$, $(T+S)(u):=T(u)+S(u)$ if $u \in D(T+S)$,
 \item $D(TS):=S^{-1}(D(T))$, $(TS)(u):=T(S(u))$ if $u \in D(TS)$.
\end{itemize}
It is easy to verify that such operations of sum and of composition are associative; namely, if $R:D(R) \to \sH$ is another operator, then $(T+S)+R=T+(S+R)$ and $(TS)R=T(SR)$. Thanks to this fact, as usual, we can write the sum $T+S+R$ and the composition $TSR$ without the parentheses. In what follows, for simplicity, we will use the symbols $Tu$ in place of $T(u)$.

Suppose the operator $T:D(T) \to \sH$ has dense domain. Define the adjoint $T^*:D(T^*) \to \sH$ of $T$ by setting
\[
D(T^*):= \{u \in \sH \, | \, \exists w_u \in \sH \mbox{ with } \b w_u|v \k=\b u | T v \k \ \forall v \in D(T)\}
\]
and
\[%\beq \label{defagg}
\b T^* u | v \k=\b u | T v \k \quad \forall v \in D(T), \forall u \in D(T^*).
\]%\eeq

The operator $T$ is called \textit{normal} if $TT^*=T^*T$. Among normal operators, the ones described by the following definitions are very important: the operator $T$ is called  \textit{self-adjoint} if $T^*=T$, \textit{anti self-adjoint} if $T^*=-T$ and \textit{unitary} if $D(T)=\sH$ and $TT^*=T^*T=\mr{I}$. The operator $T$ is said to be \textit{positive}, and we write $T \geq 0$, if $\b u | Tu \k \in \R^+$ for every $u \in D(T)$. A positive operator is also self-adjoint if $D(T)=\sH$.  As in the complex case, if $T \in \gB(\sH)$ with $T \geq 0$, then there exists a unique operator $\sqrt{T} \in \gB(\sH)$ with $\sqrt{T} \geq 0$ such that $\sqrt{T}\sqrt{T}=T$. We recall that, given any $T \in \gB(\sH)$, the operator $T^*$ belongs to $\gB(\sH)$ and $T^*T$ is a positive. In this way, we can define the operator $|T| \in \gB(\sH)$ by setting $|T|:=\sqrt{T^*T}$. If $T \in \gB(\sH)$ is normal, then it holds:
\beq \label{eq:Ker-Ran}
\Ker(T)=\Ker(T^*)=\Ker(|T|)
\; \; \text{and} \; \;
\overline{\Ran(T)}=\overline{\Ran(T^*)}=\overline{\Ran(|T|) }.
\eeq

Let us recall the concept of left scalar multiplication of $\sH$. Given a Hilbert basis $N$ of $\sH$, we define the \textit{left scalar multiplication of $\sH$ induced by $N$} as the map $\mc{L}_N:\bH \to \gB(\sH)$ given by 
\begin{center}
$(\mc{L}_N(q))(u):=\sum_{z \in N}zq\b z|u \k \;$ if $q \in \bH$ and $u \in \sH$.
\end{center}
%The map $\mc{L}_N$ turns out to be a norm-preserving real algebra homomorphism. If $\sH$ and $N$ are understood, then
For short, we will often write $L_q$ and $qu$ in place of $\mc{L}_N(q)$ and $L_qu$, respectively. The map $\mc{L}_N$ turns out to be a norm-preserving real algebra homomorphism. More precisely, it holds:
\begin{itemize}
 \item $\|qu\|=|q| \, \|u\|$ and $L_q \in \gB(\sH)$,
 \item $L_{p+q}=L_p+L_q$, $L_pL_q=L_{pq}$, $L_1=\1$ and $(L_q)^*=L_{\overline{q}}$,
 \item $L_ru=ur$ and $L_qz=zq$,
 \item $\b \overline{q} u |v \k=\b u | qv\k$
\end{itemize}
for every $u,v \in \sH$, $p,q \in \bH$, $r \in \R$ and $z \in N$. Given an operator $T:D(T) \to \sH$ and a quaternion $q \in \bH$, one can define the operator $qT:D(T) \to \sH$ by setting
\[
qT:=L_qT.
\]
If $L_q(D(T)) \subset D(T)$, one can also define the operator $Tq:D(T) \to \sH$ by setting
\[
Tq:=TL_q.
\]
In particular, if $T \in \gB(\sH)$, then $\|qT\|=\|Tq\|=|q| \, \|T\|$ and hence $qT,Tq \in \gB(\sH)$.

In what follows, by the symbol $\bH \ni q \stackrel{\LL}{\longmapsto} L_q$ or simply by $\bH \ni q \mapsto L_q$, we mean the left scalar multiplication of $\sH$ induced by some Hilbert basis of $\sH$ itself.

%%%

\subsection{Quaternionic two-sided Banach $C^*$-algebras}

A \textit{quaternionic two-sided vector space} is an Abelian group $(V,+)$ equipped with left and right scalar multiplications $\bH \times V \ni (q,u) \mapsto qu \in V$ and $V \times \bH \ni (u,q) \mapsto uq \in V$ such that:
\begin{itemize}
 \item $q(u+v)=qu+qv$, $(q+p)u=qu+pu$ and $q(pu)=(qp)u$,
 \item $(u+v)q=uq+vq$, $u(q+p)=uq+up$ and $(up)q=u(pq)$,
 \item $(qu)p=q(up)$, $u=1u=u1$ and $ru=ur$
\end{itemize}
for every $u,v \in V$, $q,p \in \bH$ and $r \in \R$. Such a quaternionic two-sided vector space $V$ becomes a \textit{quaternionic two-sided algebra} if, in addition, it is equipped with an associative and distributive product $V \times V \ni (u,v) \mapsto uv \in V$ such that
\begin{itemize}
% \item $u(vw)=(uv)w$,
% \item $u(v+w)=uv+uw$ and $(u+v)w=uw+vw$,
 \item $q(uv)=(qu)v$ and $(uv)q=u(vq)$
% \item $u=\1 u=u \1$ for some (unique) element $\1$ of $V$.
\end{itemize}
for every $u,v \in V$ and $q \in \bH$. If $uv=vu$ for every $u,v \in V$, then $V$ is called \textit{commutative}. Moreover, $V$ is said to be \textit{with unity}, or \textit{unital}, if there exists a (unique) element $\old$ of $V$, the \textit{unity} of $V$, such that $u=\old u=u\old$ for every $u \in V$.

The quaternionic two-sided algebra $V$ is a \textit{quaternionic two-sided Banach $C^*$-algebra} if, in addition, it is equipped with a norm $\|\cdot\|:V \to \R^+$ making it a real Banach space and with a $^*$-involution $^*:V \to V$ (that is, a map such that $(u^*)^*=u$, $(u+v)^*=u^*+v^*$, $(qu)^*=u^*\overline{q}$, $(uq)^*=\overline{q}u^*$ and $(uv)^*=v^*u^*$ if $u,v \in V$ and $q \in \bH$) having the following properties:
\begin{itemize}
 \item $\|qu\|=\|uq\|=|q| \, \|u\|$, $\|uv\| \leq \|u\| \, \|v\|$ and $\|u^*u\|=\|u\|^2$
\end{itemize} 
for every $u,v \in V$ and $q \in \bH$.

A \textit{$^*$-homomorphism} $\phi:V \to W$ between quaternionic two-sided Banach $C^*$-algebras is a map such that $\phi(u+v)=\phi(u)+\psi(v)$, $\phi(qu)=q\phi(u)$, $\phi(uq)=\phi(u)q$, $\phi(uv)=\phi(u)\phi(v)$ and $\phi(u^*)=\phi(u)^*$ for every $u,v \in V$ and $q \in \bH$. If $V$ and $W$ are unital, then we require also that  $\phi(\old)=\old$. %If $\phi$ is also homeomorphism\footnote{Basta ``bijective''?}, then it is called \textit{$^*$-isomorphism}\footnote{Serve?}.

\begin{remark} \label{rem:banach}
If in the preceding definition of quaternionic two-sided Banach unital $C^*$-algebra, one uses $\R$ (or $\C_\imath$ for some fixed $\imath \in \bS$) as set of scalars, then one obtains the concept of a real (respectively $\C_\imath$-complex) two-sided Banach unital $C^*$-algebra. Since $ru=ur$ for every $u \in V$ and for every $r \in \R$, the concept of real two-sided Banach unital $C^*$-algebra is equivalent to the usual one of real Banach unital $C^*$-algebra. Similarly, complex Banach unital $C^*$-algebras correspond to two-sided $\bC_\imath$-Banach unital $C^*$-algebras $V$ having the following additional property:
\beq \label{eq:cj-two-sided-banach-algebra}
\text{$cu=uc\;$ for every $u \in V$ and for every $c \in \C_\imath$.}
\eeq

Since a quaternionic two-sided Banach unital $C^*$-algebra $V$ is also a real (two-sided) Banach unital $C^*$-algebra, we can speak about real Banach unital $C^*$-subalgebras of $V$. If \eqref{eq:cj-two-sided-banach-algebra} is verified, then one can also speak about $\bC_\imath$-Banach unital $C^*$-subalgebras of $V$. \bs
\end{remark}

\begin{example} \label{example:banach}
$(1)$ Every left scalar multiplication $\bH \ni q \mapsto L_q$ of $\sH$ induces a structure of quaternionic two-sided Banach $C^*$-algebra with unity on the set $\gB(\sH)$. It suffices to equip $\gB(\sH)$ with the pointwise sum, with the left and right scalar multiplications $(q,T) \mapsto qT:=L_qT$ and $(T,q) \mapsto Tq:=TL_q$, with the composition as product, with the adjunction $T \mapsto T^*$ as $^*$-involution and with the usual norm (see the preceding subsection). Its unity is the identity map $\1:\sH \to \sH$ of $\sH$.

$(2)$ Given a non-empty set $E$, the set of bounded $\bH$-valued functions on $E$, equipped with the pointwise defined operations of sum, of left and right quaternionic scalar multiplications, of product, of conjugation $\overline{f}(q):=\overline{f(q)}$ as $^*$-involution and with the supremum norm $\|\cdot\|_\infty$, is a quaternionic two-sided Banach $C^*$-algebra, whose unity is the function constantly equal to $1$. \bs 
\end{example}

%%%

\subsection{Complex subspaces} \label{subsec:complex-subspaces} Let $\sH$ be a quaternionic Hilbert space, let $J \in \gB(\sH)$ be an anti self-adjoint and unitary operator, and let $\imath \in \bS$. Define the subsets $\sH^{J\imath}_+$ and $\sH^{J\imath}_-$ of $\sH$ by setting %$\sH^{J\imath}_+:=\{u \in \sH \, | \, Ju=u\imath\}$ and $\sH^{J\imath}_-:=\{u \in \sH \, | \, Ju=-u\imath\}$.
\begin{center}
$\sH^{J\imath}_+:=\{u \in \sH \, | \, Ju=u\imath\} \; \;$ and $\; \; \sH^{J\imath}_-:=\{u \in \sH \, | \, Ju=-u\imath\}$.
\end{center}

It is not difficult to verify that $\sH^{J\imath}_\pm \neq \{0\}$ and, for every $u,v  \in \sH^{J\imath}_\pm$ and for every $c \in \C_\imath$, it holds: $u+v,uc \in \sH^{J\imath}_\pm$ and $\b u |v \k \in \C_\imath$. In this way, $\sH^{J\imath}_+$ and $\sH^{J\imath}_-$ inherit a structure of non-zero complex Hilbert space from $\sH$: its sum is the sum of $\sH$, its complex scalar multiplication is the right scalar multiplication of $\sH$ restricted to $\C_\imath$ and its Hermitian product coincides with the one of $\sH$. These $\C_\imath$-complex Hilbert spaces are called \textit{complex subspaces $\sH^{J\imath}_+$ and $\sH^{J\imath}_-$ associated with $J$ and $\imath$}. Moreover, if $\jmath \in \bS$ with $\imath\jmath=-\jmath\,\imath$, then the map $\sH^{J\imath}_+ \ni u \mapsto u\jmath \in \sH^{J\imath}_-$ is a well-defined isometric $\C_\imath$-anti-linear isomorphism.

As a $\C_\imath$-complex (right) vector space, $\sH$ decomposes as $\sH^{J\imath}_+ \oplus \sH^{J\imath}_-$; that is, $\sH^{J\imath}_+ \cap \sH^{J\imath}_-=\{0\}$ and, given any $u \in \sH$, it holds: $u=u_++u_-$ with $u_\pm:=\frac{1}{2}(u \mp Ju\imath) \in \sH^{J\imath}_\pm$.

The set $\gB(\sH^{J\imath}_\pm)$ of all bounded (right) $\C_\imath$-linear operators on $\sH^{J\imath}_\pm$ is a $\C_\imath$-complex Banach space if we equip it with the pointwise sum, with the pointwise (right) $\C_\imath$-scalar multiplication $\gB(\sH^{J\imath}_\pm) \times \C_\imath \ni (T,c) \mapsto Tc \in \gB(\sH^{J\imath}_\pm)$ defined by
\[
\text{$(Tc)(u):=T(u)c\;$ for every $u \in \sH^{J\imath}_\pm$}
\]
and with the usual norm $\|T\|:=\sup_{u \in \sH^{J\imath}_\pm \setminus \{0\}}\|Tu\|/\|u\|$. It is worth observing that, if $T \in \gB(\sH^{J\imath}_\pm)$ and $c \in \C_\imath$, then $Tc$ is (right) $\C_\imath$-linear and hence it is a well-defined element of $\gB(\sH^{J\imath}_\pm)$, because the algebra $\C_\imath$ is commutative. Indeed, for every $u \in \gB(\sH^{J\imath}_\pm)$ and for every $d \in \C_\imath$, we have that
\[
(Tc)(ud)=T(ud)c=T(u)dc=T(u)cd=(Tc)(u)d.
\]

\begin{proposition}[{\cite[Prop.~3.11]{GhMoPe}}]
\label{prop:first-3.11}
For every operator $T \in \gB(\sH^{J\imath}_+)$, there exi\-sts a unique operator $\widetilde{T} \in \gB(\sH)$ extending $T$. Such an extension can be defined explicitly as follows. Choose an imaginary unit $\jmath$ in $\bS$ with $\imath\jmath=-\jmath\,\imath$ and define
\beq \label{eq:def-ext}
\widetilde{T}(u):=T(u_+)-T(u_-\jmath)\jmath
\eeq
for every $u \in \sH$, where $u_+:=\frac{1}{2}(u-Ju\imath) \in \sH^{J\imath}_+$ and $u_-:=\frac{1}{2}(u+Ju\imath) \in \sH^{J\imath}_-$.

The following additional facts hold:
\begin{itemize}
 \item[$(\mr{a})$] $\|\widetilde T\|=\|T\|$,
 \item[$(\mr{b})$] $J\widetilde{T}=\widetilde{T}J$,
 \item[$(\mr{c})$] $(\widetilde T)^*=\widetilde{T^*}$,
 \item[$(\mr{d})$]  $\widetilde{ST}=\widetilde{S}\widetilde{T}$ if $S \in \gB(\sH^{J\imath}_+)$,
 \item[$(\mr{e})$] $\widetilde{T} \geq 0$ if $T \geq 0$.
\end{itemize}
\end{proposition}

\begin{proof}
This proposition is a direct consequence of Proposition 3.11 of \cite{GhMoPe} and of its proof, except for point $(\mr{e})$ we are going to prove.

Suppose $T \geq 0$. Let $\jmath,u,u_+,u_-$ be as in the statement. We must show that $\langle u | \widetilde{T}u\rangle \in \R^+$. For simplicity, we set $a:=u_+$ and $b:=u_-$. By \eqref{eq:def-ext}, we have:
\begin{align*}
\langle u|\widetilde{T}u\rangle &= \langle a+b|Ta-T(b\jmath)\jmath\rangle\\
&= \langle a|Ta \rangle-\langle a|T(b\jmath)\rangle\jmath+\langle b|Ta\rangle-\langle b|T(b\jmath)\rangle\jmath\\
&=\langle a|Ta \rangle-\langle a|T(b\jmath)\rangle\jmath+\jmath\langle b\jmath|Ta\rangle-\jmath\langle b\jmath|T(b\jmath)\rangle\jmath\\
&=\langle a|Ta \rangle-\langle Ta|b\jmath\rangle\jmath+\jmath\langle b\jmath|Ta\rangle+\langle b\jmath|T(b\jmath)\rangle\\
&=\langle a|Ta \rangle-\overline{\langle b\jmath|Ta\rangle}\jmath+\jmath\langle b\jmath|Ta\rangle+\langle b\jmath|T(b\jmath)\rangle,
\end{align*}
where we used the fact that $T$ is self-adjoint and $\langle b\jmath|T(b\jmath)\rangle \in \R$. Since $\langle b\jmath|Ta\rangle \in \bC_\imath$ (see Lemma 3.9 of \cite{GhMoPe}) and $\imath\jmath=-\jmath \, \imath$, it follows at once that $\jmath\langle b\jmath|Ta\rangle=\overline{\langle b\jmath|Ta\rangle}\jmath$. In this way, we obtain that $\langle u|\widetilde{T}u\rangle=\langle a|Ta \rangle+\langle b\jmath|T(b\jmath)\rangle \in \R^+$, as desired.
\end{proof}

%%%

\subsection{Spherical spectrum} \label{subsec:ssp}

Before giving the definitions of spherical resolvent and of spherical spectrum of an operator, we recall some facts concerning the ``complex slice'' nature of quaternions. Such a nature is described by the fact that $\bH=\bigcup_{\imath \in \bS}\C_\imath$ and $\C_\imath \cap \C_\jmath =\R$ for every $\imath,\jmath \in \bS$ with $\imath \neq \pm \jmath$. This is equivalent to assert that
\begin{itemize}
 \item every quaternion $q \in \bH \setminus \R$ decomposes uniquely as $q=a+\jmath b$, where $a,b \in \R$, $b>0$ and $\jmath \in \bS$.
\end{itemize}
Evidently, we have also that $q=a+(-\jmath)(-b)$, where $-\jmath$ is again an element of~$\bS$. 
   %It is well-known that $\bH=\bigcup_{\jmath \in \bS}\bC_\jmath$ and $\bC_\jmath \cap \bC_\kappa=\bR$ for every $\jmath,\kappa \in \bS$ with $\jmath \neq \pm\kappa$.

Let $q \in \bH$. Write $q$ as follows, $q=\alpha+\jmath \beta$ for some $\alpha,\beta \in \R$ and $\jmath \in \bS$. Define the $2$-sphere $\bS_q$ of $\bH$ by setting $\bS_q:=\alpha+\beta\bS=\{\alpha+\imath\beta \in \bH \, | \, \imath \in \bS\}$. Observe that $\overline{q}=\alpha-\jmath\beta \in \bS_q$. One can show that
\[
\bS_q=\{x \in \bH \, | \, \mr{Re}(x)=\mr{Re}(q), |x|=|q|\}=\{sqs^{-1} \in \bH \, | \, s \in \bH \setminus \{0\}\}.
\]

Given a subset $E$ of $\C$, we define the circularization $\OO_E$ of $E$ in $\bH$ as the following subset of $\bH$:
\[
\OO_E=\{c+\kappa d \in \bH \, | \, c,d \in \R, c+i d \in E, \kappa \in \bS\},
\]
A subset $L$ of $\bH$ is said to be \textit{circular} if $L=\OO_E$ for some $E \subset \C$. Observe that $L$ is circular if and only $L \supset \bS_q$ for every $q \in L$.

Let $\sH$ be a quaternionic Hilbert space and let $T:D(T) \to  \sH$ be an operator. For every $q \in \bH$, 
following \cite{CGSS2007} and  \cite{libroverde},
define the operator $\Delta_q(T):D(T^2) \to \sH$ by setting
\[
\Delta_q(T):=T^2-T (2\mr{Re}(q))+\1 |q|^2. %T^2-T(q+\overline{q})+\1 |q|^2.
\]

According to the definition in \cite{GhMoPe}, which is equivalent to the one in 
\cite{CGSS2007} and  \cite{libroverde} for $T \in \gB(\sH)$ (see however Remark 4.2 in \cite{GhMoPe}),
the \textit{spherical resolvent set of $T$}, denoted by $\srho(T)$, is the circular subset of $\bH$ consisting of all quaternions $q$ having the following three properties:
\begin{itemize}
 \item $\Ker(\Delta_q(T))=\{0\}$,
 \item $\Ran(\Delta_q(T))$ is dense in $\sH$, 
 \item $\Delta_q(T)^{-1} : \Ran(\Delta_q(T)) \to  D(T^2)$ is bounded.
\end{itemize}

The \textit{spherical spectrum $\ssp(T)$ of $T$} is defined by $\ssp(T):=\bH \setminus \srho(T)$. The circular set $\ssp(T)$ is the union of three pairwise disjoint circular subsets $\sigma_{\mi{pS}}(T)$, $\sigma_{\mi{rS}}(T)$ and $\sigma_{\mi{cS}}(T)$, called respectively \textit{spherical pointwise spectrum of $T$}, \textit{spherical residual spectrum of $T$} and \textit{spherical continuous spectrum of $T$}. They are defined as follows:
\begin{align*}
\sigma_{\mi{pS}}(T)&:=\{q \in \bH \,|\, \Ker(\Delta_q(T)) \neq \{0\} \},\\
\sigma_{\mi{rS}}(T)&:=\big\{q \in \bH \big| \, \Ker(\Delta_q(T))=\{0\}, \, \overline{\Ran(\Delta_q(T))} \neq \sH \big\},\\
\sigma_{\mi{cS}}(T)&:=\big\{q \in \bH \big| \, \Ker(\Delta_q(T))= \{0\},\, \overline{\Ran(\Delta_q(T))}=\sH, \, \text{$\Delta_q(T)^{-1}$ is not bounded}\big\}.
\end{align*}

Given $q \in \bH$, a vector $u \in \sH \setminus \{0\}$ is called a \textit{eigenvector of $T$} associated with the \textit{eigenvalue~$q$} if $Tu=uq$. Observe that, if $q$ is an eigenvalue of $T$ associated with the eigenvector $u$ and $s \in \bH \setminus \{0\}$, then $sqs^{-1}$ is an eigenvalue of $T$ associated with the eigenvector $us$. In this case, $\bS_q$ is called \textit{eigensphere of $T$}. In spite of an apparently unrelated definition, it turns out that the spherical pointwise spectrum $\sigma_{\mi{pS}}(T)$ of $T$ coincides with the set of all eigenvalues of $T$.

If $T \in \gB(\sH)$, then $\ssp(T)$ is always a non-empty compact subset of $\bH$. 

%%%

\subsection{Slice-type decomposition of bounded normal operators} 

We recall a slice-type decomposition result for bounded normal operators on $\sH$.

\begin{theorem}[{\cite[Thm~5.9, Prop.~5.11, Thm~5.14]{GhMoPe}}]
\label{thm:first-5.9}
Let $\sH$ be a quaternionic Hilbert space and let $T \in \gB(\sH)$ be a normal operator. Then there exist $A,B,J \in \gB(\sH)$ such that
\begin{itemize}
 \item[$(\mr{i})$] $T=A+JB$,
 \item[$(\mr{ii})$] $A$ is self-adjoint and $B \geq 0$,
 \item[$(\mr{iii})$] $J$ is anti self-adjoint and unitary.
 \item[$(\mr{iv})$] $A$, $B$ and $J$ commute mutually.
\end{itemize}
The operators $A$ and $B$ are uniquely determined by $T$ by means of equalities $A=(T+T^*)\frac{1}{2}$ and $B=|T-T^*|\frac{1}{2}$. The operator $J$ is uniquely determined by $T$ on $\Ker(B)^\perp=\mr{Ker}(T-T^*)^\perp$; namely, if $A',B',J' \in \gB(\sH)$ are operators satisfying conditions $(\mr{i})$-$(\mr{iv})$ with $A'$, $B'$ and $J'$ in place of $A$, $B$ and $J$ respectively, then $A'=(T+T^*)\frac{1}{2}$, $B'=|T-T^*|\frac{1}{2}$ and $J'u=Ju$ for every $u \in \mr{Ker}(T-T^*)^\perp$.

Furthermore, fixed any $\imath \in \bS$, we have:
\begin{itemize}
 \item[$(\mr{a})$] $T(\sH^{J\imath}_+) \subset \sH^{J\imath}_+$ and $T^*(\sH^{J\imath}_+) \subset \sH^{J\imath}_+$.
 \item[$(\mr{b})$] $(T\rr_{\sH_+^{J\imath}})^*= T^*\rr_{\sH_+^{J\imath}}$, where $T\rr_{\sH_+^{J\imath}}$ and $T^*\rr_{\sH_+^{J\imath}}$ denote the complex operators in $\gB(\sH_+^{J\imath})$ obtained restricting respectively $T$ and $T^*$ from $\sH_+^{J\imath}$ into itself.
 \item[$(\mr{c})$] $\sigma(T\rr_{\sH_+^{J\imath}}) \cup \overline{\sigma(T\rr_{\sH_+^{J\imath}})}= \ssp(T)\cap \bC_i$. Here $\sigma(T\rr_{\sH_+^{J\imath}})$ is considered as a subset of $\bC_\imath$ via the natural identification of $\bC$ with $\bC_\imath$ induced by the real vector isomorphism $\bC \ni \alpha+i\beta \mapsto \alpha+\imath\beta \in \bC_\imath$.
 \item[$(\mr{d})$]
$\ssp(T)=\OO_\K$, where $\K:=\sigma(T\rr_{\sH_+^{J\imath}}) \subset \C$.
\end{itemize}

Finally, it holds:
\begin{itemize}
 \item[$(\mr{e})$] There exists a left scalar multiplication $\bH \ni q \mapsto L_q$ of $\sH$ such that $L_\imath=J$ and, for every $q \in \bH$, $L_qA=AL_q$ and $L_qB=BL_q$.
\end{itemize}
\end{theorem}

%%%%%%%

\section{Intertwining quaternionic projection-valued measures: iqPVMs} \label{sec:iqPVMs}

\subsection{Simple functions}
In the following, if $\imath \in \bS$ is a fixed imaginary unit, $\mscr{B}(\bC^+_\imath)$ will denote the $\sigma$-algebra of all the Borel subsets of $\bC_\imath^+:=\{\alpha+\imath\beta \in \bC_\imath \, | \, \beta \geq 0\}$, when $\bC_\imath^+$ is endowed with the topology induced by the Euclidean one of $\bH \simeq \R^4$. Given $E \in \mscr{B}(\bC^+_\imath)$, $\mscr{B}(E)$ denotes the analogous $\sigma$-algebra of the sets $F$ such that $F \in \mscr{B}(\bC^+_\imath)$ and $F \subset E$. As usual, a function $f:E \to \bH$ is (Borel) measurable if $f^{-1}(O) \in \mscr{B}(\bC^+_\imath)$ for every open (equivalently, Borel) subset $O$ of $\bH$. We denote by $M_b(E,\bH)$ the \textit{set of all bounded $\bH$-valued measurable functions on $E$}.

\begin{remark}\label{remvari}
$(1)$ $M_b(E,\bH)$ is a quaternionic two-sided Banach unital $C^*$-algebra with respect the pointwise defined operations and to the supremum norm $\|\cdot\|_\infty$. Its unity is the function $1_E:E \to \bH$ constantly equal to $1$. More precisely, $M_b(E,\bH)$ is a quaternionic two-sided Banach $C^*$-subalgebra of the quaternionic two-sided Banach unital $C^*$-algebra of bounded $\bH$-valued functions on $E$, described in Example \ref{example:banach}$(2)$.
 
$(2)$ If $E,F \in \mscr{B}(\bC^+_\imath)$ and $E \subset F$, then $f|_E \in M_b(E,\bH)$ if $f \in M_b(F,\bH)$. Similarly, if  $f \in M_b(E,\bH)$, the extension of $f$ to $F$ defined as the null function over $F \setminus E$, produces an element of $M_b(F,\bH)$. \bs
\end{remark}

As for standard measure theory, we have the following elementary definition.

\begin{definition}\label{defsf}
We will say that a function $s:\bC^+_\imath \to \bH$ is \emph{simple} if it can be represented as follows:
\beq \label{decsimps0}
s=\sum_{\ell=1}^n S_{\ell}\Chi_{E_{\ell}},
\eeq
where $S_1,\ldots,S_n$ are (not necessarily distinct) quaternions and $E_1,\ldots,E_n$ are pairwise disjoint Borel sets in $\mscr{B}(\bC^+_\imath)$. Here $\Chi_{E_\ell}$ denotes the characteristic function of $E_\ell$ in $\bC^+_\imath$. \bs
\end{definition}

\begin{remark}\label{remark1}
$(1)$ Evidently, every simple function and every continuous function is measurable.

$(2)$ We stress that, given any simple function on $\mscr{B}(\bC^+_\imath)$, there are many equiva\-lent representations as the right-hand side of equation (\ref{decsimps0}). However, all the results and definitions concerning such functions we are going to present are not affected by that ambiguity, as one can straightforward verify.

$(3)$ Taking into account that, in \eqref{decsimps0}, $E_\ell \cap E_{\ell'}= \emptyset$ if $\ell \neq \ell'$, one easily obtains:
\beq \label{sups}
|s(z)|^2=\sum_{\ell=1}^n |S_{\ell}|^2\Chi_{E_\ell}(z) \quad \mbox{for every $z \in \bC_\imath^+$. \bs} 
\eeq
\end{remark}

%%%

\subsection{Intertwining quaternionic projection-valued measures and their integral}

Let us introduce the basic notion of intertwining quaternionic projection-valued measure. We begin with the quaternionic projection-valued measures. $\N$ henceforth denotes the set of all non-negative integers.

\begin{definition}\label{defSPOVM}
Let $\sH$ be a quaternionic Hilbert space and let $\imath \in \bS$. We say that a map $P: \mscr{B}(\bC^+_\imath) \to \gB(\sH)$ is a \emph{quaternionic projection-valued measure over $\bC^+_\imath$ in $\sH$}, or \emph{qPVM over $\C_\imath$ in $\sH$} for short, if it has the following properties:
\begin{itemize}
\item[$(\mr{a})$] $P(\bC_\imath^+)=\1$, where $\1:\sH \to \sH$ is the identity map on $\sH$.
\item[$(\mr{b})$] $P(E)P(E')= P(E \cap E')$ for every $E,E' \in \mscr{B}(\bC^+_\imath)$.
\item[$(\mr{c})$] If $\{E_n\}_{n \in N}\subset \mscr{B}(\bC^+_\imath)$ with $N \subset \bN$ and $E_n \cap E_m =\emptyset$ for $n \neq m$, then:
\[
\sum_{n \in N} P(E_n)u = P\left(\,\bigcup_{n \in N} E_n\right)\!u \quad \mbox{for every $u \in \sH$}.
\]
\item[$(\mr{d})$] $P(E)\geq 0$ for every $E \in  \mscr{B}(\bC^+_\imath)$.
\end{itemize} 

The \emph{support $\supp(P)$ of $P$} is defined as the following closed subset of $\C^+_\imath$:
\beq
\supp(P) = \bC_\imath^+ \setminus \bigcup \{O \subset \bC_\imath^+ \, | \, O \mbox{ open in $\bC_\imath^+$, $P(O)=0$}\}.
\eeq
If $\supp(P)$ is bounded, then $P$ is said to be \emph{bounded}. \bs
\end{definition}

\begin{remark}\label{rempropP} 
$(1)$ As an immediate consequence of conditions $(\mr{a})$-$(\mr{d})$, we have that $P(\emptyset)=0$ and $P(E)P(E)=P(E)$ for every $E \in \mscr{B}(\bC^+_\imath)$. Moreover, bearing in mind that positivity implies self-adjointness, one easily deduce that $P(E)= P(E)^*$ for every $E \in \mscr{B}(\bC^+_\imath)$. It follows that every $P(E)$ is an orthogonal projector of $\sH$. In particular, $\|P(E)\|=1$ if $P(E) \neq 0$.

$(2)$ With the given definitions, one straightforwardly verify that the \emph{isotonous} and the \emph{sub-additivity properties} hold. In other words, respectively and for every $u \in \sH$, we have:
\[
\text{$\langle u|P(E) u\rangle \leq \langle u| P(F) u \rangle \;$ if $E,F \in \mscr{B}(\bC^+_\imath)$ with $E \subset F$}
\]
and
\[
\textstyle
\text{$\langle u| P(\bigcup_{n \in N} E_n)u \rangle \leq \sum_{n \in N}\langle u |P(E_n) u\rangle \;$ if $\{E_n\}_{n \in N} \subset \mscr{B}(\bC^+_\imath)$ with $N \subset \N$}.
\]
Similarly, the \emph{inner} and \emph{outer continuity} hold; namely, for every $\{E_n\}_{n \in \N} \subset \mscr{B}(\bC^+_\imath)$ and for every $u \in \sH$, we have:
\[
\textstyle
\text{$\lim_{n \to +\infty}\langle u| P(E_n) u\rangle=\langle u|P(\bigcup_{n \in \N} E_n) u \rangle \;$ if $E_n \subset  E_{n+1}$}
\]
and
\[
\textstyle
\text{$\lim_{n \to +\infty}\langle u| P(E_n) u \rangle=\langle u|P(\bigcap_{n \in \N} E_n) u \rangle \;$ if $E_{n} \supset  E_{n+1}$.}
\]

$(3)$ Since the topology of $\bC^+_\imath$ has a countable base, it holds $P(\bC_\imath^+ \setminus \supp(P))=0$ from $(2)$. In particular, we infer that $P(E) = P(E \cap \supp(P))$ for every $E \in \mscr{B}(\bC^+_\imath)$, and hence $P(E)=0$ if $E \cap \supp(P)=\emptyset$. \bs
\end{remark}

\begin{definition} \label{def:iqPVM}
By an \emph{intertwining quaternionic projection-valued measure over $\bC^+_\imath$ in $\sH$}, an \emph{iqPVM over $\bC^+_\imath$ in $\sH$} for short, we mean a pair $(P,\LL)$, where $P$ is a qPVM over $\C^+_\imath$ in $\sH$ and $\LL$ is a left scalar multiplication $\bH \ni q \mapsto L_q$ of $\sH$ which commutes with $P$ in the following sense:
\beq \label{eq:PL=LP}
P(E)L_q = L_q P(E) \; \mbox{ for every $E \in \mscr{B}(\bC^+_\imath)$ and for every $q \in \bH$.}
\eeq
%If $\cP=(P,\LL)$ is a iqPVM over $\C^+_\imath$ in $\sH$, then we define the \emph{support $\supp(\cP)$ of $\cP$} as the support of $P$; namely, we set $\supp(\cP):=\supp(P)$.\footnote{Credo che la nozione di $\supp(\cP)$ non serva. Propongo di cancellarla.}
If $P$ is bounded, then $\cP$ is said to be \emph{bounded}. \bs
%For convenience, we will write $(P,L_q)$ in place of $(P,\LL)$.
\end{definition}

We make an assumption that will hold throughout the remaining part of the present section. 

\begin{assumption} \label{assumpt:2}
Fix a quaternionic Hilbert space $\sH$, an imaginary unit $\imath \in \bS$ and an iqPVM $\cP=(P,\LL)$ over $\C^+_\imath$ in $\sH$. Indicate $\supp(P)$ by $\Gamma$ and the left scalar multiplication $\mc{L}$ of $\sH$ by $\bH \ni q \mapsto L_q$. Equip $\gB(\sH)$ with the structure of quaternionic two-sided Banach unital $C^*$-algebra induced by $\bH \ni q \mapsto L_q$ (see Example \ref{example:banach}(2)).~\bs 
\end{assumption}

The definition of integral of a simple function with respect to an iqPVM can be given in the standard way.

\begin{definition} \label{def:int-simple}
Let $s:\bC^+_\imath \to \bH$ be a simple function and let $s=\sum_{\ell=1}^nS_\ell\Chi_{E_\ell}$ be one of its representations. We define the \emph{integral $\int_{\bC_\imath^+} s \diff\cP$ of $s$ with respect to $\cP$} as the following operator in $\gB(\sH)$:
\beq \label{defints}
\int_{\bC_\imath^+} s \diff\cP:=\sum_{\ell=1}^n  L_{S_\ell} P(E_{\ell}). \; \text{ \bs}
\eeq
\end{definition}

In order to extend this definition of integral to a generic bounded measurable function, we state and prove two elementary, but fundamental, results.

First, we need a definition. Given $\varphi \in M_b(\bC^+_\imath,\bH)$, we set
\beq \label{normP}
\|\varphi\|^{\sss(P)}_{\infty} := \inf\big\{k \in \R  \, \big| \, P\big(\{q \in \bC_\imath^+ \, | \, |\varphi(q)|>k\}\big)=0\big\}.
\eeq

\begin{proposition} \label{prop:s}
The following assertions hold.
\begin{itemize}
\item[$(\mr{a})$] Given any $\varphi \in M_b(\bC_i^+,\bH)$, we have 
\beq \label{normP2}
\|\varphi\|^{\sss(P)}_\infty \leq \|\varphi\|_\infty
\eeq
and
\beq \label{firstestimate}
\|\varphi\|^{\sss(P)}_{\infty}=\left\| \int_{\bC_\imath^+} \varphi \diff\cP \right\| \; \mbox{ if $\varphi$ is simple.}
\eeq
\item[$(\mr{b})$] Let $u \in \sH$. Define the function $\mu^{\sss(P)}_u:\mscr{B}(\bC^+_\imath) \to \R^+$ by setting
\beq \label{measuremuPx}
\mu^{\sss(P)}_u(F):=\langle u|P(F)u \rangle=\|P(F)u\|^2 \; \mbox{ if $F \in \mscr{B}(\bC^+_\imath)$.}
\eeq
Then $\mu^{\sss(P)}_u$ is a finite regular positive Borel measure over $\mscr{B}(\bC^+_\imath)$ such that $\supp(\mu^{\sss(P)}_u) \subset \Gamma$ and
\beq \label{secondestimate}
\int_{\bC_\imath^+} |\varphi|^2 \diff\mu^{\sss(P)}_u=
\left\| \int_{\bC_\imath^+} \varphi \diff\cP u \right\|^2 \; \mbox{ if $\varphi$ is simple.}
\eeq
\end{itemize}
\end{proposition}
\begin{proof}
$(\mr{a})$ Let $s=\sum_\ell S_\ell\Chi_{E_\ell}$ be a simple function represented as in Definition \ref{defsf}.  Since $P(\emptyset)=0$, it is evident that

\beq \label{eq:norm-s}
\|s\|^{\sss(P)}_\infty=\max\{|S_{\ell}| \in \R^+ \, | \, P(E_{\ell}) \neq 0\}
\eeq
and hence inequality \eqref{normP2} is verified. Let $u \in \sH$. Bearing in mind commutativity property \eqref{eq:PL=LP} of $\cP$ and the fact that $P(E_\ell)^*=P(E_\ell)$ and $P(E_\ell)P(E_{\ell'})=P(E_{\ell} \cap E_{\ell'})=\delta_{\ell\ell'}P(E_\ell)$ for every $\ell,\ell'$, we obtain at once that
\begin{align} \label{eq}
\left\| \int_{\bC_\imath^+} s \diff\cP u\right\|^2 &
=\sum_{\ell} \langle u| P(E_\ell) L^*_{S_\ell}L_{S_\ell}P(E_\ell) u \rangle 
=\sum_{\ell} \langle u| L_{|S_\ell|^2}P(E_\ell)u \rangle \nonumber\\
&=\sum_{\ell} |S_\ell|^2\langle u| P(E_\ell)u \rangle=\sum_{\ell'} |S_\ell'|^2 \langle u| P(E_{\ell'})u \rangle, %=\sum_{\ell'} |S_{\ell'}|^2 \|P(E_{\ell'})u\|^2,
\end{align}
where $\ell'$ ranges in the subset of the values of $\ell$ such that $P(E_{\ell}) \neq 0$. Thanks to properties $(\mr{a})$ and $(\mr{c})$ of Definition \ref{defSPOVM}, to point $(2)$ of Remark \ref{rempropP} and to \eqref{eq:norm-s}, we have that
\begin{align*}
\left\| \int_{\bC_\imath^+} s \diff\cP u\right\|^2 &=\sum_{\ell'} |S_{\ell'}|^2 \langle u | P(E_{\ell'})u \rangle \\
&\leq (\|s\|^{\sss(P)}_\infty)^2 \left\langle u \left| P\left(\bigcup_{\ell'}E_{\ell'}\right) \! u \right. \right\rangle \leq (\|s\|^{\sss(P)}_\infty)^2\|u\|^2
\end{align*}
and hence $\|\int_{\bC_\imath^+} s \diff\cP \| \leq \|s\|^{\sss(P)}_\infty$. To prove (\ref{firstestimate}), it is enough to establish the existence of $v \in \sH \setminus \{0\}$ such that $\|\int_{\bC_\imath^+} s \diff\cP v\|^2=(\|s\|^{\sss(P)}_\infty)^2\|v\|^2$. If $\|s\|^{\sss(P)}_\infty=0$, then every $v \in \sH \setminus \{0\}$ has the desired property. Suppose $\|s\|^{\sss(P)}_\infty \neq 0$. Choose $\ell'_0$ in such way that $|S_{\ell'_0}|=\|s\|^{\sss(P)}_\infty$ and $P(E_{\ell'_0}) \neq 0$. In this way, the set $P(E_{\ell'_0})(\sH)$ contains at least an element $v \neq 0$. By definition of integral, one finds that
\[
\int_{\bC_\imath^+} s \diff\cP v = L_{S_{\ell'_0}} P(E_{\ell'_0})v=S_{\ell'_0}v
\]
and hence $\|\int_{\bC_\imath^+} s \diff\cP v\|^2=|S_{\ell'_0}|^2\|v\|^2=(\|s\|^{\sss(P)}_\infty)^2\|v\|^2$.

$(\mr{b})$ The fact that $\mu^{\sss(P)}_u$ is a finite positive Borel measure on $\mscr{B}(\bC^+_\imath)$ with $\supp(\mu^{\sss(P)}_u) \subset \Gamma$ is an immediate consequence of Definition \ref{defSPOVM}. Equality \eqref{secondestimate} follows directly from \eqref{eq}. Moreover $\mu_u^{\sss(P)}$ is regular because the measure space $\bC^+_\imath$ is Hausdorff and locally compact, and open sets are  countable union of compacts with finite measure (see Theorem 2.18 of \cite{Rudin}).
\end{proof}

\begin{proposition}\label{convergencesimple}
If $E \in \mscr{B}(\bC^+_\imath)$ and $\varphi \in M_b(E,\bH)$, then there exists a sequence of simple functions $\{s_n:\bC^+_\imath \to \bH\}_{n\in \bN}$ such that $||s_n -\widetilde{\varphi}||_\infty \to 0$ as $n\to +\infty$, where $\widetilde{\varphi}\in M_b(\bC_\imath^+,\bH)$ extends $\varphi$ to the null function outside $E$. Furthermore, the simple functions $s_n$ are $\bC_\imath$-valued, or real-valued, if $\varphi$ is.
\end{proposition}
\begin{proof}
Let $\{\widetilde{\varphi}_a:\bC^+_\imath \to \R\}_{a=0}^3$ be the components of $\widetilde{\varphi}$ with respect to the vector basis $e_0=1$, $e_1=i$, $e_2=j$, $e_3=k$ of $\bH$; namely, $\widetilde{\varphi}=\sum_{a=0}^3e_a\widetilde{\varphi}_a$. Since $|\widetilde{\varphi}|^2=\sum_{a=0}^3|\widetilde{\varphi}_a|^2$, each component $\widetilde{\varphi}_a$ of $\widetilde{\varphi}$ is a bounded measurable real-valued function. Thus, by a standard argument (see, for example, the proof of Theorem 1.17 of \cite{Rudin}), one can uniformly approximate $\widetilde{\varphi}_a$ by a sequence $\{s_{a,n}\}_{n \in \N}$ of real-valued simple functions on $\bC^+_\imath$. The sequence $\{s_n:=\sum_{a=0}^3e_as_{a,n}\}_{n \in \N}$ of quaternionic simple functions has the desired approximation property.
\end{proof}

We are now in a position to state the key concept of integral of a function in $M_b(E,\bH)$ with respect to an iqPVM.

\begin{definition}\label{defintslicefunctions}
Given $E \in \mscr{B}(\bC^+_\imath)$ and $\varphi \in M_b(E,\bH)$, the \emph{integral of $\varphi$ with respect to $\cP$}  is the operator in $\gB(\sH)$ defined as the following limit:
\beq \label{defintf}
\int_{E}\varphi \diff\cP:=\lim_{n \to +\infty}  \int_{\bC_\imath^+} s_n \diff\cP,
\eeq
where $\{s_n\}_{n \in \N}$ is any 
sequence of simple functions on $\bC^+_\imath$ such that $\|s_n-\widetilde{\varphi}\|_\infty \to 0$ if $n \to +\infty$ and $\widetilde{\varphi} \in M_b(\bC_\imath^+,\bH)$ extends $\varphi$ to the null function outside $E$.

Given an arbitrary subset $F$ of $\C^+_\imath$ containing $E$ and an arbitrary (possibly unbounded) function $\psi:F \to \bH$ whose restriction to $E \cap \Gamma$ belongs to $M_b(E \cap \Gamma,\bH)$, then we also define the operator $\int_E \psi \diff\cP$ in $\gB(\sH)$ by setting
\beq \label{defintfpsi}
\int_E \psi \diff\cP:=\int_{E \cap \Gamma} \psi|_{E \cap \Gamma} \diff\cP,
\eeq
where we understand that $\int_{E \cap \Gamma} \psi|_{E \cap \Gamma} \diff\cP:=0$ if $E \cap \Gamma=\emptyset$. \bs
\end{definition}

\begin{remark} \label{remarksupp}
$(1)$ Definition \eqref{defintf} is well-posed. In fact, \eqref{normP2} and \eqref{firstestimate} imply that $\|\int_{\bC_\imath^+} s_n \diff\cP -\int_{\bC_\imath^+} s_m \diff\cP\|=\|s_n-s_m\|^{\sss(P)}_\infty \leq \|s_n-s_m\|_\infty$ and hence the sequence of 
$\big\{\int_{\bC_\imath^+} s_n \diff\cP\big\}_{n \in \N}$ is a Cauchy sequence in $\gB(\sH)$. The independence of the limit from the choice of the sequence $\{s_n\}_{n \in \N}$ follows analogously.

$(2)$ If $E \in \mscr{B}(\bC^+_\imath)$ and $\varphi \in M_b(E,\bH)$, then
\[
\int_E \varphi \diff\cP=\int_{E \cap \Gamma} \varphi|_{E \cap \Gamma} \diff\cP,
\]
where the integrals are understood in the sense of \eqref{defintf}. Indeed, this equality trivially holds for simple functions and it survives the limit procedure. As a consequence, we obtain that definition \eqref{defintfpsi} is well-posed too.

$(3)$ Given $E \in \mscr{B}(\bC^+_\imath)$ and $\varphi \in M_b(E,\bH)$, sometimes it is useful to denote the integral $\int_E\varphi\diff\cP$ by the symbol $\int_E\varphi(q)\diff\cP(q)$. In fact, the latter notation $\int_E\varphi(q)\diff\cP(q)$ allows to write the integrand $\varphi$ via explicit expressions. We will use this notation in Section \ref{sec:spectral-unb}.

$(4)$ Let $E \in \mscr{B}(\bC^+_\imath)$, let $\varphi \in M_b(E,\bH)$ and let $\cP'=(P,\LL')$ be another iqPVM over $\C^+_\imath$ in $\sH$. If $\varphi$ is real-valued, then 
\[
\int_E\varphi\diff\cP=\int_E\varphi\diff\cP'.
\]
This follows immediately from the fact that $\LL(r)=\LL'(r)$ for every $r \in \R$. \bs
\end{remark}

\begin{theorem}\label{TEO1}
Consider $M_b(\Gamma,\bH)$ as a quaternionic two-sided unital Banach $C^*$-algebras (see Remark \ref{remvari}(1)). The following hold.
\begin{itemize}
 \item[$(\mr{a})$] The map
\[ 
M_b(\Gamma,\bH) \ni \varphi \longmapsto \int_\Gamma \varphi \diff\cP \in \gB(\sH) \]
is continuous because norm-decreasing and is a $^*$-homomorphism; namely, for every $\varphi,\psi \in M_b(\Gamma,\bH)$
and for every $q \in \bH$, it holds:
\begin{itemize}
 \item[$(\mr{i})$] $\int_\Gamma (\varphi+\psi) \diff\cP = \int_\Gamma \varphi \diff\cP + \int_\Gamma \psi \diff\cP$, $\int_\Gamma \varphi q \diff\cP = \big(\int_\Gamma \varphi \diff\cP\big)q$ and $\int_\Gamma q\varphi \diff\cP =q\big(\int_\Gamma \varphi \diff\cP\big)$.
 \item[$(\mr{ii})$] $\int_\Gamma \varphi\psi \diff\cP = \big(\int_\Gamma \varphi \diff\cP\big)\big(\int_\Gamma \psi \diff\cP\big)$.
 \item[$(\mr{iii})$] $\left(\int_\Gamma \varphi \diff\cP \right)^*=\int_\Gamma \overline{\varphi} \diff\cP$.
 \item[$(\mr{iv})$] $\int_\Gamma \Chi_\Gamma \diff\cP=\mr{I}$.
\end{itemize}

Furthermore, we have:
\begin{itemize}
\item[$(\mr{v})$] The operator $\int_\Gamma \varphi \diff\cP$ is normal for every $\varphi \in M_b(\Gamma,\bH)$.
 \item[$(\mr{vi})$] $\int_\Gamma q \diff\cP = L_q$ if we think the integrand $q$ as the function on $\Gamma$ constantly equal to $q$.
 \item[$(\mr{vii})$] $\Ker\big(\int_{\Gamma}\varphi\diff\cP\big)=P(\varphi^{-1}(0))(\sH)$ for every $\varphi \in M_b(\Gamma,\bH)$.
 \item[$(\mr{viii})$] $\big|\int_{\Gamma}\varphi\diff\cP\big|=\int_{\Gamma}|\varphi|\diff\cP$ for every $\varphi \in M_b(\Gamma,\bH)$.
\end{itemize}
 \item[$(\mr{b})$] If $\varphi \in M_b(\Gamma,\bH)$ and $u \in \sH$, then
\beq \label{firstestimatef}
\left\|\int_\Gamma \varphi \diff\cP \right\|= \|\varphi\|^{\sss(P)}_\infty,
\eeq
\beq \label{secondestimatef}
\left\|\int_\Gamma \varphi \diff\cP u \right\|^2 = \int_\Gamma |\varphi|^2 \diff\mu^{\sss(P)}_u
\eeq
and
\beq \label{intemediatestimatef}
\left \langle  u\left| \int_\Gamma \varphi \diff\cP u \right. \right\rangle=\int_\Gamma \varphi \, \diff\mu^{\sss(P)}_u \; \mbox{ if $\varphi$ is real-valued}.
\eeq

\item[$(\mr{c})$] Let $\jmath \in \bS$, let $\Psi :\bC_\imath^+ \to \bC^+_\jmath$ be a measurable map and let $Q:\C^+_\jmath \to \gB(\sH)$ be the map defined by setting $Q(E):=P(\Psi^{-1}(E))$ for every $E \in \mscr{B}(\C^+_\jmath)$. Then $\mc{Q}:=(Q,\mc{L})$ is an iqPVM over $\C^+_\jmath$ in $\sH$ such that
\beq \label{PQ}
\int_{\bC_\imath^+} \xi \circ \Psi \diff\cP= \int_{\bC_\jmath^+} \xi \diff\mc{Q} \; \mbox{ if $\xi \in M_b(\bC_\jmath^+,\bH)$}.
\eeq
\end{itemize}
\end{theorem}
\begin{proof} 
We begin by proving properties $(\mr{i})$-$(\mr{vi})$ of point $(\mr{a})$. Such properties $(\mr{i})$-$(\mr{iv})$ for simple functions can be easily verified by direct inspection. We underline that, in such verifications, the commutativity condition \eqref{eq:PL=LP} plays a crucial role. Thanks to Proposition \ref{convergencesimple}, these properties extend immediately to the whole $M_b(\Gamma,\bH)$ by density. Properties $(\mr{v})$ and $(\mr{vi})$ follow at once from the preceding ones $(\mr{i})$-$(\mr{iv})$. Bearing in mind \eqref{normP2}, equality \eqref{firstestimatef} follows from \eqref{firstestimate} by density. In turn, inequality \eqref{normP2} and equality  \eqref{firstestimatef} imply that the map $\varphi \mapsto \int_\Gamma \varphi \diff\cP$ is norm-decreasing and thus continuous.

Let us prove \eqref{secondestimatef}. Let $u \in \sH$, let $\varphi \in M_b(\Gamma,\bH)$ and let $\{s_n\}_{n \in \N}$ be a sequence of simple functions on $\bC^+_\imath$ uniformly converging to $\widetilde{\varphi}$ as in Proposition \ref{convergencesimple}. Since $\varphi$ is bounded and the convergence is uniform, there exists a positive real constant $C$ such that $\|\widetilde{\varphi}\|_\infty+\sup_{n \in \N}\|s_n\|_\infty \leq C$. In this way, we have:
\begin{align*}
\left|\int_\Gamma (|\varphi|^2 - |s_n|^2) \diff\mu^{\sss(P)}_u\right| &\leq \int_\Gamma (|\varphi|+ |s_n|)|\varphi-s_n| \diff\mu^{\sss(P)}_u\\
&\leq C \big\|\widetilde{\varphi}-s_n\big\|_\infty\mu^{\sss(P)}_u(\Gamma)=C \big\|\widetilde{\varphi}-s_n\big\|_\infty\|u\|^2,
\end{align*}
where in the last equality we used the fact that $\mu^{\sss(P)}_u(\Gamma)=\langle u |P(\Gamma) u \rangle =  \langle u |P(\bC^+_\imath) u\rangle = \langle u|u \rangle=\|u\|^2$. It follows that $\int_S |s_n|^2 \diff\mu^{\sss(P)}_u \to \int_\Gamma |\varphi|^2 \diff\mu^{\sss(P)}_u$ as $n \to +\infty$. On the other hand, $\int_\Gamma |s_n|^2 \diff\mu^{\sss(P)}_u=\|\int_{\bC^+_\imath} s_n \diff\cP u\|^2$ by \eqref{secondestimate}. Since $\|\int_{\bC^+_\imath} s_n \diff\cP u\| \to \|\int_{\bC^+_\imath} \widetilde{\varphi} \diff\cP u\|=\|\int_\Gamma \varphi \diff\cP u\|$, we obtain equality \eqref{secondestimatef}.

Consider now a real-valued simple function $s:\bC^+_\imath \to \bH$. Let $s=\sum_\ell S_\ell\Chi_{F_\ell}$ be one of its representation with $S_\ell \in \R$. Equality \eqref{intemediatestimatef} holds for $s$. Indeed, we have:
\begin{align*}
\left \langle u\left| \, \int_\Gamma s \diff\cP u \right. \right\rangle &=\left \langle u\left| \, \sum_\ell S_\ell P(F_\ell) u \right. \right\rangle=\sum_\ell \left \langle u\left| P(F_\ell) u \right. \right\rangle S_\ell\\
&=\sum_\ell S_\ell \, \mu^{\sss(P)}_u(F_\ell)=\int_\Gamma s \diff\mu^{\sss(P)}_u
\end{align*}
By density, \eqref{intemediatestimatef} extends to all real-valued functions $\varphi \in M_b(\bC^+_\imath,\bH)$.

Let us show properties $(\mr{vii})$ and $(\mr{viii})$ of point $(\mr{a})$. Let $\varphi$ be an arbitrary function in $M_b(\Gamma,\bH)$ and let $E:=\varphi^{-1}(0)$. By above point $(\mr{a})(\mr{ii})$, we know that
\[
\int_\Gamma \varphi \diff\cP \, P(E)=\int_\Gamma \varphi \diff\cP \int_\Gamma \Chi_E \diff\cP=\int_\Gamma \varphi \Chi_E \diff\cP=0
\]
and hence $P(E)(\sH) \subset \Ker\big(\int_\Gamma \varphi \diff\cP\big)$. Consider a vector $u$ in $\Ker\big(\int_\Gamma \varphi \diff\cP\big)$. By \eqref{secondestimatef}, we have that $\int_\Gamma |\varphi|^2\diff\mu_u^{\sss(P)}=0$. Since the measure $\mu_u^{\sss(P)}$ is positive, it follows that $\int_F|\varphi|^2\diff\mu_u^{\sss(P)}=0$ for every $F \in \mscr{B}(\C^+_\imath)$ with $F \subset \Gamma$. In particular, this is true when $F$ is equal to $F_n:=\{q \in \Gamma \, | \, |\varphi(q)| \geq \frac{1}{n}\}$ for any $n \in \N \setminus \{0\}$. We infer that $0=\int_{F_n}|\varphi|^2\diff\mu_u^{\sss(P)} \geq n^{-2}\mu_u^{\sss(P)}(F_n) \geq 0$ and hence $\mu_u^{\sss(P)}(F_n)=0$ and $\mu_u^{\sss(P)}(\Gamma \setminus E)=\lim_{n \to +\infty}\mu_u^{\sss(P)}(F_n)=0$. By \eqref{measuremuPx}, the latter equality is equivalent to the following one $P(\Gamma \setminus E)u=0$, which in turn is equivalent to $u=P(E)u$; that is, $u \in P(E)(\sH)$. This proves that $\Ker(\int_\Gamma \varphi \diff\cP)=P(E)(\sH)$.

By \eqref{intemediatestimatef}, we know that the operator $\int_\Gamma|\varphi|\diff\cP$ is positive. On the other hand, points $(\mr{a})(\mr{ii})$ and $(\mr{a})(\mr{iii})$ imply that $\left(\int_\Gamma\varphi\diff\cP\right)^*\left(\int_\Gamma\varphi\diff\cP\right)=\int_\Gamma|\varphi|^2\diff\cP=\left(\int_\Gamma|\varphi|\diff\cP\right)^2$. Now equality $(\mr{viii})$ follows immediately from the uniqueness of the square root (see \cite[Thm 2.18]{GhMoPe}). 

Concerning point $(\mr{c})$, it is easy to verify, by direct inspection, that $\mc{Q}$ is an iqPVM over $\C^+_\jmath$ and \eqref{PQ} holds for simple functions on $\bC^+_\jmath$. Finally, consider $\xi \in M_b(\bC^+_\jmath,\bH)$ and a sequence $\{\eta_n\}_{n \in \N}$ of simple functions on $\bC^+_\jmath$ converging uniformly to $\xi$. Since the sequence $\{\eta_n \circ \Psi\}_{n \in \N}$ of simple functions on $\bC^+_\imath$ converges uniformly to $\xi \circ \Psi$, by the definition of integral of a bounded measurable function with respect to an iqPVM, we infer at once \eqref{PQ}:
\[
\int_{\bC_\imath^+} \xi \circ \Psi \diff\cP=\lim_{n \to +\infty}\int_{\bC_\imath^+} \eta_n \circ \Psi \diff\cP=\lim_{n \to +\infty}\int_{\bC_\jmath^+} \eta_n \diff\mc{Q}=\int_{\bC_\jmath^+} \xi \diff\mc{Q}.
\]
The proof is complete.
\end{proof}

Fix $u,v \in \sH$ and $\jmath \in \bS$ in such a way that $\imath\jmath=-\jmath \, \imath$. The assigned qPVM $P$ over $\C^+_\imath$ in $\sH$ allows to define a quaternion-valued finite measure $\nu^{\sss(P)}_{u,v}:\mscr{B}(\bC^+_\imath) \to \bH$ by setting
\beq \label{eq:nu}
\nu^{\sss(P)}_{u,v}(E):= \langle u | P(E) v\rangle \; \mbox{ if $E \in \mscr{B}(\bC^+_\imath)$}.
\eeq 
By the quaternionic polarization identity (see Proposition 2.2 of \cite{GhMoPe}), we have that
\begin{align}
4\nu^{\sss(P)}_{u,v}(E)=&\,\mu^{\sss(P)}_{u+v}(E) -\mu^{\sss(P)}_{u-v}(E) +\big(\mu^{\sss(P)}_{u\imath+v}(E) -\mu^{\sss(P)}_{u\imath-v}(E)\big) \imath \nonumber\\
&+\big(\mu^{\sss(P)}_{u\jmath +v}(E) -\mu^{\sss(P)}_{u\jmath-v}(E)\big)\jmath+\big(\mu^{\sss(P)}_{u\kappa +v}(E)-\mu^{\sss(P)}_{u\kappa-v}(E)\big) \kappa, \label{dec}
\end{align}
where $\kappa$ is the imaginary unit of $\bH$ defined by $\kappa:=\imath\jmath$.

Given a Borel set $E \in \mscr{B}(\bC^+_\imath)$ and a \emph{real-valued} function $\varphi \in M_b(\bC^+_\imath,\bH)$, we define the quaternion $\int_E \varphi \diff\nu^{\sss(P)}_{u,v}$ in the natural way suggested by decomposition (\ref{dec}):
\begin{align*}
\textstyle
\int_E \varphi \diff\nu^{\sss(P)}_{u,v}:= & \, \textstyle\frac{1}{4}\big(\int_E \varphi \diff\mu^{\sss(P)}_{u+v}-\int_E \varphi \diff\mu^{\sss(P)}_{u-v} +\left(\int_E \varphi \diff\mu^{\sss(P)}_{u\imath+v} -\int_E \varphi \diff\mu^{\sss(P)}_{u\imath-v}\right) \imath\\
&\textstyle+\left(\int_E \varphi \diff\mu^{\sss(P)}_{u\jmath +v} -\int_E \varphi \diff\mu^{\sss(P)}_{u\jmath-v}\right)\jmath+\left(\int_E \varphi \diff\mu^{\sss(P)}_{u\kappa +v}-\int_E \varphi \diff\mu^{\sss(P)}_{u\kappa-v}\right) \kappa\big), %\label{dec-int}
\end{align*}

Equality \eqref{intemediatestimatef} generalizes as follows:

\begin{proposition} \label{nu}
For every $u,v \in \sH$ and for every real-valued function $\varphi \in M_b(\Gamma,\bH)$, it holds:
\beq \label{intemediatestimatef2}
\left\langle u \left| \int_\Gamma \varphi \diff\cP v\right.\right\rangle = \int_\Gamma \varphi \diff\nu^{\sss(P)}_{u,v}.
\eeq
\end{proposition}
\begin{proof}
By density, it suffices to show \eqref{intemediatestimatef2} for real-valued simple functions on $\Gamma$. Moreover, by linearity, one can reduce to consider simple functions of the form $s=S\Chi_E$ for some $S \in \R$ and $E \in \mscr{B}(\bC^+_\imath)$. Define $u':=P(E)u$ and $v':=P(E)v$. Equality \eqref{intemediatestimatef2} holds for such a function $s$. Indeed, we have:
\begin{align*}
\textstyle
4\left\langle u \left| \int_\Gamma s \diff\cP v\right.\right\rangle =& \, \textstyle 4\langle u | SP(E)v\rangle=4\langle u | P(E)v\rangle S=4\langle u' | v' \rangle S\\
=& \, \textstyle (\langle u'+v'|u'+v'\rangle - \langle u'-v'|u'-v'\rangle)S\\
&+(\langle u'\imath+v'|u'\imath+v'\rangle - \langle u'\imath-v'|u'\imath-v'\rangle)\imath S\\
&+(\langle u'\jmath+v'|u'\jmath+v'\rangle - \langle u'\jmath-v'|u'\jmath-v'\rangle)\jmath S\\
&+(\langle u'\kappa+v'|u'\kappa+v'\rangle - \langle u'\kappa-v'|u'\kappa-v'\rangle)\kappa S\\
=& \, \textstyle (\langle u+v|P(E)(u+v)\rangle - \langle u-v|P(E)(u-v)\rangle)S\\
&+(\langle u\imath+v|P(E)(u\imath+v)\rangle - \langle u\imath-v|P(E)(u\imath-v)\rangle)\imath S\\
&+(\langle u\jmath+v|P(E)(u\jmath+v)\rangle - \langle u\jmath-v|P(E)(u\jmath-v)\rangle)\jmath S\\
&+(\langle u\kappa+v|P(E)(u\kappa+v)\rangle - \langle u\kappa-v|P(E)(u\kappa-v)\rangle)\kappa S\\
=& \, \textstyle 4S\nu^{\sss(P)}_{u,v}(E)=4\int_\Gamma s \diff\nu^{\sss(P)}_{u,v},
\end{align*}
where we used the quaternionic polar identity, decomposition \eqref{dec}, and the fact that the operator $P(E)$ is right linear, self-adjoint and $P(E)P(E)=P(E)$. 
\end{proof}

\begin{remark}\label{remwekadef}
Given any real-valued function $\varphi \in M_b(\Gamma,\bH)$, it is clear that there is a unique operator $S: \sH \to \sH$ satisfying
\beq \label{intemediatestimatef2bis}
\left\langle u \left| S v\right.\right\rangle = \int_\Gamma \varphi \diff\nu^{\sss(P)}_{u,v} \;\; \text{ for every $u,v \in \sH$,} 
\eeq
and, obviously, $S=\int_\Gamma \varphi \diff\cP$. This observation can be exploited to define $\int_\Gamma \varphi \diff\cP$ from scratch as such unique operator.  This definition, though definitely elegant and adopted by some authors (see \cite{ACK}), is not so technically easy to handle differently from the complex case. Above all,  this definition  turns out to be hardly extendible to the case of a generic $\bH$-valued bounded measurable function $\varphi$ in the absence of a quaternionic version of Radon-Nikodym theorem. \bs
\end{remark}

We conclude this subsection with an extension lemma, which will prove its crucial importance later.

\begin{lemma} \label{lem:fundamental}
Let $\sH$ be a quaternionic Hilbert space, let $J \in \gB(\sH)$ be an anti-self-adjoint and unitary operator, and let $\imath \in \bS$. Then, for every complex PVM $p:\mscr{B}(\bC^+_\imath) \to \gB(\sH^{J\imath}_+)$ over $\bC^+_\imath$ in $\sH^{J\imath}_+$, there exists a unique qPVM $P:\mscr{B}(\bC^+_\imath) \to \gB(\sH)$ over $\bC^+_\imath$ in $\sH$ extending $p$; that is, 
\[
P(E)(\sH^{J\imath}_+) \subset \sH^{J\imath}_+
\; \text{ and } \;
P(E)|_{\sH^{J\imath}_+}=p(E) \; \mbox{ for every $E \in \mscr{B}(\bC^+_\imath)$}.
\]
Such an extension has the following properties:
\beq \label{eq:norm-comm}
\|P(E)\|=\|p(E)\| \; \mbox{ and } \; P(E)J=JP(E)
\eeq
for every $E \in \mscr{B}(\bC^+_\imath)$. In particular, $P$ and $p$ have the same support.
%In addition, there exists a left scalar multiplication $\bH \ni q \stackrel{\LL}{\longmapsto} L_q$ of $\sH$ such that the pair $\cP=(P,\LL)$ is a iqPVM on $\bC^+_\imath$ such that
%\begin{itemize}
% \item[$(\mr{i})$] 
%\end{itemize}
%\[
%\left.\int_{\bC^+_\imath} \varphi \diff\cP\right|_{\sH^{J\imath}_+}=\int_{\bC^+_\imath} \varphi \diff p
%\; \; \mbox{ if $\varphi \in M_b(\bC^+_\imath)$ is $\bC_\imath$-valued.}
%\]
%Here the second integral is understood in the usual sense: it is the integral of the bounded measurable $\bC_\imath$-complex valued function $\varphi$ with respect to the complex PVM $p$.
%Such a left scalar multiplication $\LL$ is never unique.

%There exists a (not unique) left scalar multiplication $\bH \ni q \stackrel{\LL}{\longmapsto} L_q$ of $\sH$ such that the pair $\cP=(P,\LL)$ is a iqPVM on $\bC^+_\imath$ with the following property:
%\begin{itemize}
% \item[$(\mr{i})$] 
%\end{itemize}
%\[
%\int_{\bC^+_\imath}\int_\imath \varphi \diff\cP=
%\]
\end{lemma} 
\begin{proof}
This result follows immediately from Proposition \ref{prop:first-3.11}. We have to be careful regarding one point only: condition $(\mr{c})$ of Definition \ref{defSPOVM}. Let us prove that $P$ has this property. Consider a family of Borel sets $\{E_n\}_{n \in N}$ in $\mscr{B}(\bC^+_\imath)$ with $N \subset \bN$ and $E_n \cap E_m =\emptyset$ for $n \neq m$. Let $x \in \sH$ and let $\jmath,a,b$ be as in the statement of Proposition \ref{prop:first-3.11}. Bearing in mind definition \eqref{eq:def-ext}, we obtain:
\begin{align*}
\textstyle\sum_{n \in N}P(E_n)x &\textstyle=\sum_{n \in N}\big(p(E_n)a-p(E_n)(b\jmath)\jmath\big)\\
&\textstyle=\sum_{n \in N}p(E_n)a-\sum_{n \in N}p(E_n)(b\jmath)\jmath\\
&\textstyle=p\left(\bigcup_{n \in N}E_n\right)a-p\left(\bigcup_{n \in N}E_n\right)(b\jmath)\jmath=P\left(\bigcup_{n \in N}E_n\right)x,
\end{align*}
as desired.
\end{proof}

%%%

%\subsection{Integration of real-valued functions with respect to quaternionic projection-valued measures} \label{subsec:real-int}

%To be written

%%%%%%%

%\section{The spectral theorem for bounded normal operators}
\section{Spectral representation of bounded normal operators \\ and iqPVMs} \label{sec:bounded-sp-teo}

%%%

\subsection{The spectral theorem for bounded normal operators}

We are now in a position to state and prove the spectral theorem for normal operators on quaternionic Hilbert spaces. %Among other results,  we intend to prove that, for $T \in \gB(\sH)$ normal and a preferred imaginary unit $\imath$, there is a iqPVM $\cP=(P,L)$ such that $T = \int_{\bC_\imath^+} \mi{id} \diff\cP$, moreover $\sigma_S(T) = \Omega_{\supp(P)}$ and $P$ is completely determined by $T$ and $\imath$, while $L$ is only partially determined. \footnote{Richiamo sopra la propriet\`a seguente: $\ssp(T)$ is the disjoint union of $\sigma_{\mi{pS}}(T)$ and $\sigma_{\mi{cS}}(T)$, because $\sigma_{\mi{rS}}(T)=\emptyset$.} 

\begin{theorem}\label{spectraltheorem}
Let $\sH$ be a quaternionic Hilbert space and let $\imath \in \bS$. Then, given a normal operator $T \in \gB(\sH)$, there exists a bounded iqPVM $\cP=(P,\LL)$ over $\bC^+_\imath$ in $\sH$ such that
\beq \label{decspet}
T=\int_{\bC^+_\imath} \mi{id} \diff\cP,
\eeq
where $\mi{id}:\bC^+_\imath \hookrightarrow \bH$ indicates the inclusion map.

The following additional facts hold:
\begin{itemize}
\item[$(\mr{a})$] $T$ determines uniquely $P$ and the restriction of $L_q=\LL(q)$ to $\Ker(T-T^*)^\perp$ for every $q \in \C_\imath$. More precisely, if $\cP'=(P',\LL')$ is any bounded iqPVM over $\C^+_\imath$ in $\sH$ such that $T=\int_{\bC^+_\imath} \mi{id} \diff\cP'$, then $P'=P$ and $L'_q(u)=L_q(u)$ for every $u \in \Ker(T-T^*)^\perp$ and for every $q \in \C_\imath$, where $L'_q$ denotes $\LL'(q)$. Furthermore, we have:
 \begin{itemize}
  \item[$(\mr{i})$] $L'_\imath T=TL'_\imath$ and $L'_\imath T^*=T^*L'_\imath$,
  \item[$(\mr{ii})$] $L'_\jmath T=T^*L'_\jmath$ if $\jmath \in \cS$ with $\imath\jmath=-\jmath\,\imath$,
  \item[$(\mr{iii})$] $-L'_\imath(T-T^*)=|T-T^*|=\int_{\C^+_\imath}2\beta \diff\cP'$, where the integrand $2\beta$ indicates the function $\C^+_\imath \ni \alpha+\imath\beta \mapsto 2\beta$,
 \end{itemize}

 \item[$(\mr{b})$] $\supp(P)=\ssp(T) \cap \C^+_\imath$.
 \item[$(\mr{c})$] $P(\R)(\sH)=\mr{Ker}(T-T^*)$  and hence $P(\bC^+_\imath \setminus \R)(\sH)=\mr{Ker}(T-T^*)^\perp$.
 \item[$(\mr{d})$] Concerning the spherical spectrum $\ssp(T)$ of $T$, we have:
 \begin{itemize}
  \item[$(\mr{i})$] $q \in \sigma_{\mi{pS}}(T)$ if and only if $P(\bS_q \cap \bC^+_\imath) \neq 0$. In particular, every isolated point of $\ssp(T) \cap \bC^+_\imath$ belongs to $\sigma_{\mi{pS}}(T)$.
  \item[$(\mr{ii})$] $\sigma_{\mi{rS}}(T)=\emptyset$.
  \item[$(\mr{iii})$] $q \in \sigma_{\mi{cS}}(T)$ if and only if $P(\bS_q \cap \bC^+_\imath)=0$ and $P(U) \neq 0$ for every open subset $U$ of $\bC^+_\imath$ containing $\bS_q \cap \bC^+_\imath$. Furthermore, if $q \in \sigma_{\mi{cS}}(T)$, then, for every positive real number $\epsilon$, there exists a vector $u_\epsilon \in \sH$ such that $\|u_\epsilon\|=1$ and $\|Tu_\epsilon - u_\epsilon q\|<\epsilon$.
 \end{itemize}
\end{itemize}
\end{theorem}

\begin{proof}
We divide the proof into two steps.

\textit{Step I.} Let us prove the existence of an iqPVM $\cP=(P,\LL)$ satisfying \eqref{decspet}. In view of Theorem \ref{thm:first-5.9}, there exists an anti self-adjoint and unitary operator $J \in \gB(\sH)$ such that $T=A+JB$, where $A=(T+T^*)\frac{1}{2}$, $B=|T-T^*|\frac{1}{2}$ and the operators $A$, $B$ and $J$ commute mutually. In particular, the operators $A$, $B$ and $T$ commute with $J$ and hence such operators leave fixed the complex subspace $\sH^{J\imath}_+$. Denote by $A_+,B_+,T_+ \in\gB(\sH^{J\imath}_+)$ the restrictions of $A,B,T$ from $\sH^{J\imath}_+$ into itself, respectively. By point $(\mr{b})$ of the above mentioned theorem, it follows that $A_+$ and $B_+$ are self-adjoint and $T_+$ is normal. Moreover, $B_+ \geq 0$ and $T_+=A_++B_+\imath$ in $\gB(\sH^{J\imath}_+)$. Apply the classical complex spectral theorem to $T_+$. We obtain a $\bC_\imath$-complex PVM $p:\mscr{B}(\bC^+_\imath) \to \gB(\sH^{J\imath}_+)$ such that $T_+=\int_{\bC^+_\imath}\mi{id} \diff p$. Here, by abuse of notation, $\mi{id}$ indicates the identity map on $\bC^+_\imath$. Since $B_+=\int_{\sigma(T_+)} \mr{im}(z) \diff p(z)$ and $B_+ \geq 0$, it must be 
$\mr{im}(\sigma(T_+)) \subset  [0,+\infty)$; that is,
$\sigma(T_+) \subset \bC^+_\imath$, where we identified $\bC$ with $\bC_\imath$ in the natural way.
We conclude that $\supp(p)$ is contained in $\bC_\imath^+$, because it coincides with $\sigma(T_+)$. The latter fact allows to apply Lemma \ref{lem:fundamental}, obtaining a qPVM $P:\mscr{B}(\bC^+_\imath) \to \gB(\sH^{J\imath}_+)$ extending $p$ and satisfying \eqref{eq:norm-comm}. In particular, $P(E)$ commutes with $J$ for every $E \in \mscr{B}(\bC^+_\imath)$.

Our next aim is to construct a left scalar multiplication $\bH \ni q \mapsto L_q$ of $\sH$ such that $L_\imath=J$ and $P(E)$ commutes with $L_q$ also for all $q \in \sH \setminus \{\imath\}$. In order to do this, it suffices to find a Hilbert basis $N$ of $\sH^{J\imath}_+$ such that
\beq \label{eq:zp(E)z'}
\langle z|p(E)z'\rangle \in \R \; \mbox{ if $z,z' \in N$}.
\eeq
Indeed, such a basis $N$ of $\sH^{J\imath}_+$ would be also a Hilbert basis of $\sH$ and, if $\bH \ni q \stackrel{\LL}{\longmapsto} L_q$ denotes the left scalar multiplication of $\sH$ induced by $N$, then $L_\imath=J$ (see Lemma 3.10$(\mr{b})$ and Proposition 3.8$(\mr{f})$ of \cite{GhMoPe} for a proof of this assertion). Furthermore, for every $E \in \mscr{B}(\bC^+_\imath)$ and for every $q \in \bH$, we would have:
\begin{align*}
(L_qP(E))(z')&=\sum_{z \in N}zq\langle z|p(E)z'\rangle=\sum_{z \in N}z\langle z|p(E)z'\rangle q\\
&=(p(E)(z'))q=P(E)(z'q)=(P(E)L_q)(z').
\end{align*}
In other words, it would hold that $L_qP(E)=P(E)L_q$; that is, the pair $(P,\LL)$ would be an iqPVM over $\bC^+_\imath$ in $\sH$.

To construct the desired Hilbert basis $N$ of $\sH^{J\imath}_+$, we adapt to the present situation the argument we used to prove Theorem 5.14 of \cite{GhMoPe} (see point $(\mr{e})$ of Theorem \ref{thm:first-5.9} for the statement of such a result). Since the operator $T_+$ is bounded and normal, by Theorem X.5.2 of \cite{DS}, there exist a family $\{\mu_\lambda\}_{\lambda \in \Lambda}$ of finite positive regular Borel measures on $\sigma(T_+)$ and an isometry $V:\sH^{J\imath}_+ \to \bigoplus_{\lambda \in \Lambda} L^2_\lambda$ with $L^2_\lambda:=L^2(\sigma(T_+),\bC_\imath;\mu_\lambda)$ such that, for every bounded measurable function $\psi:\sigma(T_+) \to \bC_\imath$, the operator $\psi(T_+):= \int_{\sigma(T_+)} \psi \diff p \in \gB(\sH^{J\imath}_+)$ is multiplicatively represented  in $\bigoplus_{\lambda \in \Lambda}L^2_\lambda$ in the sense that
\[
(V \psi(T_+) V^{-1})(\oplus_{\lambda \in \Lambda} \varphi_\lambda)=\oplus_{\lambda \in \Lambda} \psi \cdot \varphi_\lambda.
\]
In particular, we have that 
$(V p(E) V^{-1})(\oplus_{\lambda \in \Lambda} \varphi_\lambda)=\oplus_{\lambda \in \Lambda} \Chi_E \cdot \varphi_\lambda$ if $E \in \mscr{B}(\bC^+_\imath)$. By Zorn's lemma, it is easy to infer the existence of a Hilbert basis $M$ of $\bigoplus_{\lambda \in \Lambda}L^2_\lambda$ formed by real-valued functions defined on $\sigma(T_+)$ and belonging to some $L^2_\lambda$. The Hilbert basis $N:=V^{-1}(M)$ of $\gB(\sH^{J\imath}_+)$ has the required property \eqref{eq:zp(E)z'}. Indeed, if $z,z' \in N$, $f:=V(z) \in L^2_\lambda$, $f':=V(z') \in L^2_{\lambda'}$ and $E \in \mscr{B}(\bC^+_\imath)$, then
\[
\langle z|p(E)z'\rangle=0 \in \R \; \mbox{ if $\lambda \neq \lambda'$}
\]
and
\[
\langle z|p(E)z'\rangle =\int_{\sigma(T_+)}\overline{f}\Chi_Ef' \diff \mu_\lambda=\int_{\sigma(T_+)}f\Chi_Ef' \diff \mu_\lambda \in \R \; \mbox{ if $\lambda=\lambda'$},
\]
because $f$, $f'$ and $\Chi_E$ are real-valued.

We have just constructed a Hilbert basis $N$ of $\gB(\sH)$, which induces a left scalar multiplication $\bH \ni q \stackrel{\LL}{\longmapsto} L_q$ of $\sH$ such that $L_\imath=J$ and the pair $\cP:=(P,\LL)$ is an iqPVM over $\bC^+_\imath$ in $\sH$.

It remains to verify formula \eqref{decspet}. Let $s$ be a $\C_\imath$-valued simple function on $\C^+_\imath$. Suppose $s$ is of the form $s=S\Chi_E$ for some $S \in \C_\imath$ and $E \in \mscr{B}(\C^+_\imath)$. By definition, the operator $\int_{\bC^+_\imath} s \diff \cP$ is equal to $L_SP(E)$ and hence it commutes with $J=L_\imath$, because $L_\imath L_S=L_{\imath S}=L_{S\imath}=L_SL_\imath$ and $L_\imath P(E)=P(E)L_\imath$. It follows that
\[
\text{$\left(\int_{\bC^+_\imath} s \diff \cP\right)(\sH^{J\imath}_+) \subset \sH^{J\imath}_+ \; \;$ and $\; \; \left.\int_{\bC^+_\imath} s \diff \cP\right|_{\sH^{J\imath}_+}=\int_{\bC^+_\imath} s \diff p$.}
\]
By linearity, the latter two facts hold true for all $\C_\imath$-valued simple function on $\C^+_\imath$.

%If $s=\sum_{\ell=1}^n S_\ell \Chi_{E_\ell}$ is a simple function with $S_\ell \in \bC_\imath$, we have that
%\begin{align*}
%\left.\int_{\bC^+_\imath} s \diff \cP\right|_{\sH^{J\imath}_+}&=\sum_{\ell=1}^n\big(S_\ell P(E_\ell)\big)\big|_{\sH^{J\imath}_+}=\sum_{\ell=1}^n\big(P(E_\ell)S_\ell\big)\big|_{\sH^{J\imath}_+}\\%\nonumber\\
%\label{eq:p}
%&=\sum_{\ell=1}^n p(E_\ell)S_\ell=\int_{\bC^+_\imath} s \diff p.
%\end{align*}
%\footnote{
%Ho tolto la seguente parte che mi pare non serva: Moreover, by combining \eqref{firstestimate}, \eqref{eq:norm-s} and the uniqueness of the quaternionic extension $P$ of $p$, it follows immediately that   
%\begin{align}
%\left\|\int_{\bC^+_\imath} s \diff \cP\right\|&=\max\{|S_{\ell}| \in \R^+ \, | \, P(E_{\ell}) \neq 0\}\nonumber\\
%\label{eq:S}
%&=\max\{|S_{\ell}| \in \R^+ \, | \, p(E_{\ell}) \neq 0\}=\left\|\int_{\bC^+_\imath} s \diff p\right\|.
%\end{align}
%}
Making use of a sequence of $\bC_\imath$-valued simple functions converging uniformly to $\mi{id}$, we infer at once that the same is true also for $\mi{id}$:
\[
\text{$\left(\int_{\bC^+_\imath} \mi{id} \diff \cP\right)(\sH^{J\imath}_+) \subset \sH^{J\imath}_+ \; \;$ and $\; \; \left.\int_{\bC^+_\imath} \mi{id} \diff \cP\right|_{\sH^{J\imath}_+}=\int_{\bC^+_\imath} \mi{id} \diff p=T_+$.}
\]
%\[
%\left.\int_{\bC^+_\imath} \mi{id} \diff\cP\right|_{\sH^{J\imath}_+}=\int_{\bC^+_\imath} \mi{id} \diff p=T_+.
%\]

Since each operator in $\gB(\sH^{J\imath}_+)$ extends uniquely to an operator in $\gB(\sH)$ (see Proposition \ref{prop:first-3.11}), the operators $\int_{\bC^+_\imath} \mi{id} \diff\cP$ and $T$ coincide. Formula \eqref{decspet} is proved.  

\textit{Step II.} Let us prove points $(\mr{a})$-$(\mr{d})$.

$(\mr{a})$ Let $\cP'=(P',\LL')$ be another iqPVM over $\bC^+_\imath$ in $\sH$ such that $P'$ is bounded and $\int_{\bC_\imath^+} \mi{id} \diff \cP'=T=\int_{\bC_\imath^+} \mi{id} \diff \cP$. First observe that properties $(\mr{i})$-$(\mr{iii})$ follows immediately from points $(\mr{a})(\mr{i})$, $(\mr{a})(\mr{ii})$, $(\mr{a})(\mr{iii})$ and $(\mr{a})(\mr{viii})$ of Theorem \ref{TEO1} and from the fact that $\imath\,\mi{id}=\mi{id}\,\imath$ and $\jmath\,\mi{id}=\overline{\mi{id}}\,\jmath$ if $\jmath \in \bS$ with $\imath\jmath=-\jmath\,\imath$.

Using again points $(\mr{a})(\mr{i})$, $(\mr{a})(\mr{iii})$ and $(\mr{a})(\mr{viii})$ of Theorem \ref{TEO1}, we obtain that 
\[
\int_{\bC_\imath^+} \alpha \diff \cP'=(T+T^*)\frac{1}{2}=A=\int_{\bC_\imath^+} \alpha \diff \cP
\]
and
\[
\int_{\bC_\imath^+} \beta \diff \cP'=|T-T^*|\frac{1}{2}=B=\int_{\bC_\imath^+} \beta \diff \cP,
\]
where the integrands $\alpha$ and $\beta$ indicate the functions $\C^+_\imath \ni (\alpha,\beta) \mapsto \alpha$ and $\C^+_\imath \ni (\alpha,\beta) \mapsto \beta$, respectively. By points $(\mr{a})(\mr{i})$ and $(\mr{a})(\mr{ii})$ of the same theorem, we also infer that $\int_{\bC_\imath^+}f \diff\cP'=\int_{\bC_\imath^+} f \diff\cP$ for every real-valued polynomial function $\C^+_\imath \ni (\alpha,\beta) \mapsto f(\alpha,\beta)$. Denote by $K$ the compact subset $\supp(P) \cup \supp(P')$ of $\bC^+_\imath$. Making use of the Stone-Weierstrass theorem, one gets immediately the equality $\int_{\C^+_\imath} \varphi \diff\cP'=\int_K \varphi \diff\cP'=\int_K \varphi \diff\cP=\int_{\C^+_\imath} \varphi \diff\cP$ for every real-valued continuous function $\varphi$ on $\bC_\imath^+$. Thus, if $\varphi$ is such a function and $x \in \sH$, equality \eqref{intemediatestimatef} ensures that 
\[
\int_{\C^+_\imath} \varphi \diff\mu_x^{\sss(P)}= 
\left\langle x\left| \int_{\C^+_\imath} \varphi \diff\cP x\right.\right\rangle=\left \langle x\left| \int_{\C^+_\imath} \varphi \diff\cP' x\right. \right\rangle =\int_{\C^+_\imath} \varphi \diff\mu_x^{\sss(P')}.
\]
In view of Riesz theorem for positive regular Borel measures, we have that $\mu_x^{\sss(P)}=\mu_x^{\sss(P')}$ and thus the equality $\int_K \varphi \diff\mu_x^{\sss(P)}=\int_K\varphi \diff\mu_x^{\sss(P')}$ continues to hold for every real-valued function $\varphi$ in $M_b(K,\bH)$. In other words, if $\varphi \in M_b(K,\bH)$ is real-valued, then $\big\langle x \big| (\int_{\C^+_\imath} \varphi \diff\cP-\int_{\C^+_\imath} \varphi \diff\cP')x \big\rangle=0$ for every $x \in \sH$. The latter condition implies that $\int_{\bC_\imath^+}\varphi \diff\cP=\int_{\bC_\imath^+}\varphi \diff\cP'$ if $\varphi \in M_b(K,\bH)$ is real-valued (see Proposition 2.17$(\mr{a})$ of \cite{GhMoPe} for a proof of this implication). In particular, this is true if $\varphi=\Chi_E$ for every $E \in \mscr{B}(\bC^+_\imath)$ and hence
\[
P(E)=\int_{\bC^+_\imath}\Chi_E \diff\cP=\int_{\bC^+_\imath}\Chi_E \diff\cP'=P'(E),
\]
as desired. Moreover, we have
\[
L_\imath B=\int_{\C^+_\imath} \imath\beta\diff\cP =(T-T^*)\frac{1}{2}=\int_{\C^+_\imath}\imath\beta\diff\cP'=L'_\imath B
\]
and hence $L_\imath=L'_\imath$ on $\overline{\Ran(B)}=\overline{\Ran(|T-T^*|)}$. On the other hand, by (2.13) and (2.15) of \cite{GhMoPe}, we know that $\overline{\Ran(|T-T^*|)}=\Ker(T-T^*)^\perp$.

$(\mr{b})$ By construction of $P$ in Step I, we have that $\supp(P)=\supp(p)$. However, the classical complex spectral theorem for $T_+$ ensures that $\supp(p)$ is equal to $\sigma(T_+)$ and hence it is compact. Point $(\mr{d})$ of Theorem \ref{thm:first-5.9} ensures that $\OO_{\supp(p)}=\ssp(T)$.

$(\mr{c})$  Since $T-T^*=\int_{\C^+_\imath}2\beta\imath\diff\cP$, point $(\mr{a})(\mr{vii})$ of Theorem \ref{TEO1} ensures that $\Ker(T-T^*)=P(\R)(\sH)$, which is equivalent to $\mr{Ker}(T-T^*)^\perp=P(\bC^+_\imath \setminus \R)(\sH)$.

$(\mr{d})$ In the proof of this point, we will adopt the notations used in Step I.

Let us prove $(\mr{i})$. Since $\sigma_{\mi{pS}}(T)$ is a circular set, it suffices to show that $q \in \sigma_{\mi{pS}}(T) \cap \bC^+_\imath$ if and only if $P(\{q\}) \neq 0$ or, equivalently, if $p(\{q\}) \neq 0$. Let $q \in \bC^+_\imath$. Recall that $q \in \sigma_{pS}(T)$ if and only if $Tu=uq$ for some $u \in \sH \setminus \{0\}$ (see Proposition 4.5 of \cite{GhMoPe} for a proof). Fix $\jmath \in \bS$ with $\imath\jmath=-\jmath\,\imath$. Define $u_\pm:=\frac{1}{2}(u \mp Ju\imath) \in \sH^{J\imath}_\pm$. Due to Lemma 3.10 in \cite{GhMoPe}, the identity $Tu=uq$ decomposes in $T_+u_+=u_+q$ and $-T_+(u_-\jmath)\jmath=u_-q$. Since $q \in \bC_\imath$, the latter equality can be re-written as $T_+ u_-\jmath = u_-\jmath\overline{q}$. Since $u \neq 0$, either $u_+ \neq 0$ or $u_- \neq 0$ (and hence $u_-\jmath \neq 0$). It follows that $q \in \sigma_{\mi{pS}}(T)$ if and only if either $q \in \sigma_p(T_+)$ when $u_+ \neq 0$ or $q=\overline{q} \in \sigma_p(T_+)$ when $u_-\jmath \neq 0$ (because, in the latter case, $\overline{q} \in \sigma_p(T_+) \subset \bC^+_\imath$). In this way, we have found that $q \in \sigma_{\mi{pS}}(T)$ if and only if $q \in \sigma_p(T_+)$. From the standard complex spectral theory, we know that $q \in \sigma_p(T_+)$ if and only if $p(\{q\}) \neq 0$. In particular, this is true if $q$ is an isolated point of $\sigma_p(T_+)$ (see Theorem 8.54(b)(iii) of \cite{Mo-book}).

Point $(\mr{ii})$ is contained in Theorem 4.8$(\mr{a})(\mr{ii})$ of \cite{GhMoPe}.

The first part of $(\mr{iii})$ follows immediately from $(\mr{i})$, $(\mr{ii})$ and from the fact that $\sigma_{\mi{cS}}(T) \cap \bC^+_\imath=(\ssp(T) \cap \bC^+_\imath) \setminus \sigma_{\mi{pS}}(T)=\supp(P) \setminus \sigma_{\mi{pS}}(T)$. Let us complete the proof of $(\mr{iii})$. Let $q \in \sigma_{\mi{cS}}(T)$ and let $\vep$ be a positive real number. By definition of $\sigma_{\mi{cS}}(T)$, we know that $\Delta_q(T)$ is injective, but $\Delta_q(T)^{-1}:\mr{Ran}(\Delta_q(T)) \to \sH$ is not bounded. In this way, there exists $x_\vep \in \sH \setminus \{0\}$ such that $\|\Delta_q(T)^{-1}(x_\vep)\|/\|x_\vep\|>\vep^{-2}$. If we set $u_\vep:=\Delta_q(T)^{-1}(x_\vep)/\|\Delta_q(T)^{-1}(x_\vep)\|$, we can re-write the latter inequality as follows, $\|\Delta_q(T)u_\vep\|<\vep^2$ or, equivalently,
\beq \label{eq:u}
\|T(Tu_\vep-u_\vep q)-(Tu_\vep-u_\vep q)\overline{q}\|<\vep^2.
\eeq
If $\|Tu_\vep-u_\vep q\|<\vep$, we are done. Suppose $\|Tu_\vep-u_\vep q\|\geq\vep$. Let $\kappa \in \bS$ such that $\overline{q}=\kappa q \kappa^{-1}$. Define $v_\vep:=(Tu_\vep-u_\vep q)/\|Tu_\vep-u_\vep q\|$ and $u'_\vep:=v_\vep\kappa$. Observe that $\|u'_\vep\|=1$, because $|\kappa|=1$. Thanks to \eqref{eq:u}, we have:
\begin{align*}
\|Tu'_\vep-u'_\vep q\|&=\|Tv_\vep-v_\vep \overline{q}\| \cdot |\kappa|=\frac{\|\Delta_q(T)u_\vep\|}{\|Tu_\vep-u_\vep q\|}<\frac{\vep^2}{\vep}=\vep.
\end{align*}
The vector $u'_\vep \in \sH$ has the desired property.
\end{proof}

%%%

\subsection{iqPVMs associated with bounded normal operators}

In the preceding spectral theorem, the uniqueness of $P$ has been established, however no uniqueness property holds for $\LL$. As a matter of fact, it is possible to 
prove that every left multiplication $\LL$ with a nice interplay with $T$ can be used to construct an iqPVM verifying the spectral formula \eqref{decspet}. We state a corresponding result on this point.

\begin{theorem}\label{propL}
Let $\sH$ be a quaternionic Hilbert space, let $\imath \in \bS$, let $T \in \gB(\sH)$ be a normal operator and let $P:\mscr{B}(\bC^+_\imath) \to \gB(\sH)$ be the qPVM over $\bC^+_\imath$ in $\sH$ associated with $T$ by means of Spectral Theorem \ref{spectraltheorem}(a). Given a left scalar multiplication $\bH \ni q \stackrel{\LL}{\longmapsto}\!L_q$ of $\sH$, we have that the pair $\cP=(P,\LL)$ is an iqPVM over $\bC^+_\imath$ in $\sH$ verifying (\ref{decspet}) if and only if the following three conditions hold:
\begin{itemize}
 \item[$(\mr{i})$] $L_\imath T = T L_\imath$,
 \item[$(\mr{ii})$] $L_\jmath T = T^* L_\jmath$ for some $\jmath \in \bS$ with $\imath\jmath=-\jmath \, \imath$,
 \item[$(\mr{iii})$] $-L_\imath (T-T^*)\geq 0$.
\end{itemize}
\end{theorem}

\begin{proof} By point $(\mr{a})$ of Theorem \ref{spectraltheorem}, we just know that, if $\cP=(P,\LL)$ is an iqPVM verifying \eqref{decspet}, then properties $(\mr{i})$, $(\mr{ii})$ and $(\mr{iii})$ hold.

Let us prove the converse. Suppose that $\LL$ verifies $(\mr{i})$, $(\mr{ii})$ and $(\mr{iii})$.

First, we show that $L_q$ commutes with the operators $A:=(T+T^*)\frac{1}{2}$ and $B:=|T-T^*|\frac{1}{2}$ for every $q \in \bH$. In order to do this, it suffices to verify that $L_\imath$ and $L_\jmath$ commute with $A$ and $B$, because the left multiplication for reals coincides with the right one and does not depend on the choice of $\LL$. By $(\mr{i})$, we have that
\[
T^*L_\imath=-(L_\imath T)^*=-(T L_\imath)^*=L_\imath T^*.
\]
It follows that $L_\imath$ commutes with $A$ and with $T-T^*$. In particular, we have that
\[
\big(-L_\imath(T-T^*)\big)^2=-(T-T^*)^2=(T-T^*)^*(T-T^*)=(2B)^2.
\]
By the uniqueness of the square root of a positive operator, we infer that
\[
-L_\imath(T-T^*)=2B.
\]

By $(\mr{ii})$, $L_\jmath$ commutes with $A$ as well, and it holds $L_\jmath(T-T^*)=-(T-T^*)L_\jmath$. In this way, it holds:
\begin{align*}
2(L_\jmath B) &=-L_\jmath L_\imath(T-T^*)=-L_{\jmath\,\imath}(T-T^*)=L_{\imath\jmath}(T-T^*)\\
&=L_\imath L_\jmath(T-T^*)=-L_\imath(T-T^*)L_\jmath=2(BL_\jmath);
\end{align*}
that is, $L_\jmath$ commutes with $B$. This proves that $L_q$ commutes with $A$ and $B$ for every $q \in \bH$.

We are now in a position to show that the pair $\cP:=(P,\LL)$ is an iqPVM over $\bC^+_\imath$ in $\sH$. This is equivalent to prove that $L_\kappa$ commutes with $P(E)$ for every $\kappa \in \{\imath,\jmath\}$ and for every $E \in \mscr{B}(\bC^+_\imath)$. Let $\kappa \in \{\imath,\jmath\}$ and let $E \in \mscr{B}(\C^+_\imath)$.

Thanks to Theorem \ref{spectraltheorem}, there exists a left scalar multiplication $\LL^0$ of $\sH$ such that $\cP^0:=(P,\LL^0)$ is an iqPVM over $\bC^+_\imath$ in $\sH$ for which $T=\int_\Gamma\mi{id}\diff\cP^0$,
\beq \label{eq:A-B}
\text{$A=\int_\Gamma\alpha\diff\cP^0\;\;$ and $\;\;B=\int_\Gamma\beta\diff\cP^0$,}
\eeq
and
\beq \label{eq:P(E)}
P(E)=\int_\Gamma \Chi_E\diff\cP^0,
\eeq
where $\Gamma$ indicates the compact subset $\supp(P)$ of $\C^+_\imath$. Moreover, if $L^0_\imath$ denotes $\LL^0(\imath)$, then  $T-T^*=L^0_\imath\big(-L^0_\imath(T-T^*)\big)$, the operators $L^0_\imath$ and $-L^0_\imath(T-T^*)$ commute and $-L^0_\imath(T-T^*)=|T-T^*| \geq 0$. By hypothesis, the latter properties are true also for $L_\imath$; namely, $T-T^*=L_\imath(-L_\imath(T-T^*))$, the operators $L_\imath$ and $-L_\imath(T-T^*)$ commute and $-L_\imath(T-T^*)=|T-T^*| \geq 0$. In this way, Theorem \ref{thm:first-5.9} applied to $T-T^*$ and Theorem \ref{spectraltheorem}$(\mr{c})$ imply that
\beq\label{eq:L}
\text{$L_\imath=L^0_\imath\;$ on $\mr{Ker}(T-T^*)^\perp=P(\C^+_\imath \setminus \R)(\sH)$.}
\eeq

 We just know that $L_\kappa$ commutes with $A$ and $B$. Bearing in mind \eqref{eq:A-B} and Theorem \ref{TEO1}, we infer at once that $\int_\Gamma f\diff\cP^0$ for every real-valued polynomial function $\Gamma \ni (\alpha,\beta) \to f(\alpha,\beta)$. Thanks to the Stone-Weierstrass approximation theorem, it follows at once that the same commutative property holds for every real-valued continuous function on $\Gamma$. Let $\varphi$ be such a function. Making use of the latter commutative property of $L_\kappa$, of \eqref{intemediatestimatef} and of the equality $L_\kappa^*=-L_\kappa$, we obtain that
\begin{align}
\int_\Gamma \varphi\diff\mu^{\sss(P)}_{L_\kappa u} &=  \left\langle L_\kappa u \left|\int_\Gamma \varphi
\diff\cP^0 L_\kappa u \right. \right\rangle = \left\langle u \left|-L_\kappa\int_\Gamma \varphi
\diff\cP^0 L_\kappa u \right. \right\rangle \nonumber\\
& \label{aggid}
=-\left\langle u \left|\int_\Gamma \varphi
\diff\cP^0 L_\kappa L_\kappa u \right. \right\rangle=\left\langle u \left|\int_\Gamma \varphi
\diff\cP^0 u \right. \right\rangle = \int_\Gamma\varphi \diff\mu^{\sss(P)}_u
\end{align}
and hence
\[ %\label{aggid}
\int_\Gamma \varphi\diff\mu^{\sss(P)}_{L_\kappa u} =\int_\Gamma\varphi \diff\mu^{\sss(P)}_u
\]
for every $u \in \sH$. Thanks to Riesz' theorem for positive regular Borel measures, we have that $\mu^{\sss(P)}_{L_\kappa u}=\mu^{\sss(P)}_u$ and hence the chain of equalities \eqref{aggid} extends to all real-valued measurable functions $\varphi$ on $\Gamma$. In particular, we have:
\[
\left\langle u \left| -L_\kappa \int_\Gamma \Chi_E
\diff\cP^0 L_\kappa u \right. \right\rangle = \left\langle u \left|\int_\Gamma \Chi_E
\diff\cP^0 u \right. \right\rangle
\]
and hence, by \eqref{eq:P(E)},
\[
\big\b u \big| (-L_\kappa P(E) L_\kappa - P(E)) u \big\k=0
\]
for every $u \in \sH$. It follows that $-L_\kappa P(E) L_\kappa - P(E)=0$ (see Proposition 2.17$(\mr{a})$ of \cite{GhMoPe}) or, equivalently, $P(E)L_\kappa=L_\kappa P(E)$. This proves that $\cP=(P,\LL)$ is an iqPVM over $\bC^+_\imath$ in $\sH$.

It remains to show that $T=\int_\Gamma\mi{id}\diff\cP$. Thanks to \eqref{eq:L}, if a simple function $s:\C^+_\imath \to \bH$ is real-valued or it is $\C_\imath$-valued and $s|_\R=0$, then one has:
\[
\int_\Gamma s \diff\cP=\int_\Gamma s \diff\cP^0.
\]
Now, it suffices to consider a sequence of simple functions of the form $\{s_n+ s'_n\imath\}_{n \in \N}$ with $s_n$ and $s'_n$ real-valued and $s'_n|_\R=0$, which converges uniformly to $\mi{id}$ on $\Gamma$. Indeed, exploiting the very definition of integral of a function with respect to an iqPVM, we immediately obtain:
\[
\int_\Gamma \mi{id} \diff\cP=\int_\Gamma \mi{id} \diff\cP^0=T,
\]
as desired. The proof is complete.
\end{proof}

%%%%%%%

\subsection{Algebraic characterization of left scalar multiplications and spectral $\boldsymbol{L^2}$-representation theorem for bounded normal operators}

%\subsection{The spectral representation theorem}

We are going to prove a version of the spectral $L^2$-representation theorem for bounded quaternionic normal operators. Such a proof will use the following result, announced in Remark 3.2(2) of \cite{GhMoPe}), which is of independent interest. Indeed, it completely characterizes the class of left scalar multiplications of a quaternionic Hilbert space.

\begin{theorem}\label{propLprod}
Let $\sH$ be a quaternionic Hilbert space and let $\bH \ni q \stackrel{\LL}{\longmapsto} L_q \in \gB(\sH)$ be a map verifying the following conditions:
\begin{itemize}
 \item[$(\mr{a})$] $L_{p+q}=L_p+L_q$, $L_pL_q=L_{pq}$ and $L_ru=ur$,
 \item[$(\mr{b})$] $(L_q)^*=L_{\overline{q}}$
\end{itemize}
for every $p,q \in \bH$, for every $r \in \R$ and for every $u \in \sH$. Then $\LL$ is a left scalar multiplication of $\sH$.
\end{theorem}

\begin{proof}
We must show that there exists a Hilbert basis $N$ of $\sH$ such that $L_qz=zq$ for every $q \in \bH$ and for every $z \in N$.

\textit{Step I.} Let us prove the existence of $z \in \sH \setminus \{0\}$ such that $L_qz=zq$ for every $q \in \bH$. Let $\imath,\jmath$ be elements of $\bS$ with $\imath \jmath = -\jmath \, \imath$ and let $\kappa:=\imath\jmath$. It turns out that $\{1,\imath,\jmath,\kappa\}$ is a real vector basis of $\bH$. In this way, if $q \in \bH$, then $q=q_0+q_1\imath+q_2\jmath+q_3\kappa$ for some unique $q_0,q_1,q_2,q_3 \in \R$. By property $(\mr{a})$, we have that
\[
L_q=\1 q_0+L_\imath q_1+L_\jmath q_2+L_\imath L_\jmath q_3.
\]
Therefore, if $z \in \sH \setminus \{0\}$ satisfies the conditions $L_\imath z=z\imath$ and $L_\jmath z=z\jmath$, then it also satisfies the required condition $L_q z = zq$ for every $q \in \bH$. Consider the $\C_\kappa$-complex Hilbert space $\sH^{L_\kappa \kappa}_+:=\{x \in \sH \, | \, L_\kappa x=x\kappa\}$. It is known that $\sH^{L_\kappa \kappa}_+ \neq \{0\}$ (see Proposition 3.8$(\mr{d})$ of \cite{GhMoPe} for a proof). We assert that there exists $x \in \sH^{L_\kappa \kappa}_+ \setminus \{0\}$ such that $L_\imath x+x \imath \neq 0$. On the contrary, suppose that $L_\imath y=-y\imath$ for every $y \in \sH^{L_\kappa \kappa}_+$. Since $yk \in \sH^{L_\kappa \kappa}_+$, we would have that $L_\imath yk=-yk\imath$. Hence $L_\imath yk=y \imath k$ and $L_\imath y=y\imath$. Taking into account  that $L_\imath y=-y\imath$, it would follow that $\sH^{L_\kappa \kappa}_+=\{0\}$, which is a contradiction. This proves the existence of $x \in \sH \setminus \{0\}$ such that $L_\kappa x=x\kappa$ (namely, $x \in \sH^{L_\kappa \kappa}_+ \setminus \{0\}$) and $L_\imath x + x \imath \neq 0$, as desired.

Define $z:=x-L_\imath x\imath=-(L_\imath x+x\imath) \imath \neq 0$. It holds %We have as wanted both $L_\imath z= z\imath$ and $L_\jmath z= z\jmath$ because
\[
L_\imath z = L_\imath x + x\imath=(x-L_\imath x\imath) \imath=z\imath
\]
and
\begin{align*}
L_\jmath z &= L_\jmath x - L_\jmath L_\imath x\imath =   L_\jmath x +  L_k x\imath  =
 L_\jmath x+ x k\imath =  L_\jmath x + x \jmath\\
 &=-L_i L_k x + x \jmath =-L_i xk  + x \jmath= -L_i x i\jmath  +x \jmath =  z\jmath.
\end{align*}
Replacing $z$ with $z/\|z\|$, we can assume that $\|z\|=1$ as well. We have just proved that there exists $z \in \sH$  such that $\|z\|=1$ and $L_q z=zq$ for every $q \in \bH$.

\textit{Step II.} Given a non-empty subset $S$ of $\sH$, we say that $S$ is good if, for every $s,s' \in S$ with $s \neq s'$, $\b s|s' \k=0$, $\|s\|=1$ and $L_qs=sq$ for every $q \in \bH$. Let $\mc{N}$ be the set whose elements are all the good subsets of $\sH$, partially ordered by the inclusion. We have just shown that $\mc{N} \neq \emptyset$. Since $\mc{N}$ has the property that every totally ordered subset of $\mc{N}$ has an upper bound (given by the union of its elements), by Zorn's lemma, it admits a maximal element $N$. It remains to prove that such a set $N$ is a Hilbert basis of $\sH$; that is, $N^\perp=\{0\}$. On the contrary, suppose that $N^\perp \neq \{0\}$. By property $(\mr{b})$, we infer that $L_q(N^\perp) \subset N^\perp$ for every $q \in \bH$. Indeed, if $w \in N^\perp$ and $q \in \bH$, it holds:
\[
\b L_qw|z \k=\b w|L_{\overline{q}}z \k=\b w|z\overline{q}\k=\b w|z \k\overline{q}=0
\;\;
\text{ for every $z \in N$.}
\]
In this way, considering each $L_q$ as an operator in $\gB(N^\perp)$ by restriction, we can apply Step I again, obtaining $z_0 \in N^{\perp}$ such that $\|z_0\|=1$ and $L_qz_0=z_0q$ for every $q \in \bH$. This contradicts the maximality of $N$, because $N \cup \{z_0\} \in \mc{N}$. It follows that $N^\perp=\{0\}$, as desired. The proof is complete.
\end{proof}

Let $X$ be a non-empty set and let $\mu$ be a positive $\sigma$-additive measure over a $\sigma$-algebra of $X$. Consider a function $\varphi:X \to \bH$. Write $\varphi$ as follows $\varphi=\sum_{a=0}^3 \varphi_ae_a$, where $e_0:=1$, $e_1:=i$, $e_2:=j$, $e_3:=k$ and each $\varphi_a$ is a real-valued function on $X$. We say that $\varphi$ is $\mu$-summable if each $\varphi_a$ is. In this case, we define the integral $\int_X\varphi\diff\mu$ of $\varphi$ with respect to $\mu$ by setting
\[
\int_X\varphi\diff\mu:=\sum_{a=0}^3\left(\int_X\varphi_a\diff\mu\right)e_a.
\]

Indicate by $L^2(X,\bH;\mu)$ the set of all functions $\varphi:X \to \bH$ whose square norm $|\varphi|^2:X \to \R$ is $\mu$-summable. Define the function $L^2(X,\bH;\mu) \times L^2(X,\bH;\mu) \ni (\varphi,\psi) \mapsto \b \varphi|\psi\k \in \bH$ as follows:
\[
\b\varphi|\psi\k:=\int_X \overline{\varphi}\,\psi\diff\mu,
\]
where $\overline{\varphi}\,\psi$ indicates the pointwise product between $\overline{\varphi}$ and $\psi$.

The set $L^2(X,\bH;\mu)$, equipped with the pointwise defined operations of sum and of right quaternionic scalar multiplication, and with the preceding function $(\varphi,\varphi') \mapsto \b \varphi|\varphi'\k$ as a Hermitian quaternionic scalar product, becomes a quaternionic Hilbert space. In what follows, the symbol $L^2(X,\bH;\mu)$ will denote such a quaternionic Hilbert space. Recall that, if $\sL$ is another quaternionic Hilbert space, then an operator $U:L^2(X,\bH;\mu) \to \sL$ is called \textit{isometry} if it is bijective and it preserves the norm of vectors.

Our spectral $L^2$-representation theorem reads as follows.

\begin{theorem} \label{teorapp}
Let $\sH$ be a quaternionic Hilbert space, let $\imath \in \bS$, let $T \in \gB(\sH)$ be a normal operator and let $\cP=(P,\LL)$ be an iqPVM over $\bC^+_\imath$ in $\sH$ verifying (\ref{decspet}). Define $\Gamma:=\ssp(T) \cap \C^+_\imath$ and $\varphi(T):=\int_{\Gamma}\varphi\diff\cP$ for every $\varphi \in M_b(\Gamma,\bH)$.
 
Then there exists an orthogonal decomposition of $\sH$ into closed (right $\bH$-linear) subspaces $\sH=\bigoplus_{\alpha \in \mc{A}} \sH_\alpha$ having the following properties:
\begin{itemize}
 \item[$(\mr{a})$] $\varphi(T)(\sH_\alpha) \subset \sH_\alpha$ for every $\varphi \in M_b(\Gamma,\bH)$ and for every $\alpha \in \mc{A}$.
 \item[$(\mr{b})$] For every $\alpha \in \mc{A}$, there exist a positive $\sigma$-additive finite 
Borel measure $\nu_\alpha$ over $\Gamma$ and an isometry
$U_\alpha :L^2(\Gamma,\bH;\nu_\alpha) \to \sH_\alpha$ such that, for every $\psi_\alpha \in L^2(\Gamma,\bH;\nu_\alpha)$, it holds:
\[
U_\alpha^{-1} \varphi(T)\rr_{\sH_\alpha} U_\alpha\psi_\alpha=\varphi\psi_\alpha
\quad
\text{$\nu_a$-a.e. in $\Gamma$}
\]
and, in particular,
%for every $\psi_\alpha \in L^2(\ssp(T) \cap \C^+_\imath,\bH;\nu_\alpha)$. In particular, it holds: 
\[
U_\alpha^{-1} L_q \rr_{\sH_\alpha} U_\alpha \psi_\alpha=q \psi_\alpha
\quad
\text{$\nu_a$-a.e. in $\Gamma$}
\]
for every $q \in \bH$.
\end{itemize}
\end{theorem}

\begin{proof}
By the definition of left scalar multiplication, there exists $z \in \sH$ such that $\|z\|=1$ and $L_qz=zq$ for every $q \in \bH$.
Define $\sH_z$ as the closure in $\sH$ of the subspace spanned by the vectors of the form $\sum_{k=1}^N P(E_k) zq_k$, where $N \in \bN$, the $q_k$'s are arbitrary quaternions and $\{E_k\}_{k=1}^N$ is a family of pairwise disjoint Borel sets in $\mscr{B}(\C^+_\imath)$. Direct inspection shows that, if $s_u:=\sum_{k=1}^N q_k\Chi_{E_k}$ is the simple function on $\C^+_\imath$ associated with the vector $u=\sum_{k=1}^N P(E_k)zq_k \in \sH_z$, we get:
\[
\|u\|^2=\sum_{k=1}^N \left\|P(E_k) zq_k\right\|^2 = \sum_{k=1}^N |q_k|^2 \mu^{\sss(P)}_{z}(E_k) = \int_{\sigma_S(T) \cap \bC_\imath^+} |s_u|^2 d\mu_z^{\sss(P)}.
\]
This suggests to focus on the operator $U_z :  L^2(\Gamma,\bH;\mu_z^{\sss(P)}) \to \sH_z$ uniquely defined by fixing it on the dense subset of simple functions:
\[
\sum_{k=1}^N \Chi_{E_k} q_k \, \stackrel{U_z}{\longmapsto} \, \sum_{k=1}^N P(E_k) zq_k.
\]
Since the range of $U_z$ is dense by definition, the operator $U_z$ is an isometry.

Observe that, if $D,E \in \mscr{B}(\C^+_\imath)$ and $p,q \in \bH$, then, bearing in mind that the pair $(P,\LL)$ is an iqPVM and $L_qz=zq$, we have:
\begin{align*}
L_qP(E)\left(U_z(p\Chi_D)\right)&=L_qP(E)\big(P(D)zp\big)=P(E \cap D)L_qzp\\
&=P(E \cap D)z(qp)=U_z\left((q\Chi_E)(p\Chi_D)\right).
\end{align*}
This immediately implies that
\[
\varphi(T) (\sH_z) \subset \sH_z
\;\; \text{ and } \;\;
U_z^{-1} \varphi(T)\rr_{\sH_z} U_z\psi=\varphi\psi \;\; \text{ on $\Gamma$}
\]
for every pair of simple functions $\varphi,\psi$ on $\Gamma$. By density, the latter inclusion and the latter equality remain valid for every $\varphi \in M_b(\Gamma,\bH)$ and for every $\psi \in L^2(\Gamma,\bH;\mu_z^{\sss(P)})$. 

In the fortunate case in which $\sH_z=\sH$, the proof stops here. Otherwise, one passes to consider $\sH_z^\perp \neq \{0\}$. Such a subspace of $\sH$ is invariant under the action of every operator $\varphi(T)$ with $\varphi \in M_b(\Gamma,\bH)$; that is, $\varphi(T)(\sH_z^\perp) \subset \sH_z^\perp$. Indeed, if $\varphi$ is such a function, then $\varphi(T)(\sH_z) \subset \sH_z$, $\varphi(T)^*(\sH_z)=\overline{\varphi}(T)(\sH_z) \subset \sH_z$ and hence, for every $w \in \sH_z^\perp$ and for every $v \in \sH_z$, it holds:
\[
\b\varphi(T)w|v\k=\b w|\varphi(T)^*v\k=0,
\]
as desired. In particular, we have that $L_q(\sH_z^\perp) \subset \sH_z^\perp$ for every $q \in \bH$ and hence, by Theorem \ref{propLprod}, the restrictions of the $L_q$'s to $\sH_z^\perp$ define again a left scalar multiplication of $\sH_z^\perp$. Then one looks for another $z' \in \sH_z^\perp$ such that $L_qz'=z'q$ for every $q \in \bH$ and re-implements the procedure. More precisely, one can complete the proof by using Zorn's lemma as in Step II of the proof of Theorem \ref{propLprod}.
\end{proof}

\begin{remark}\label{remarknonuniqL}
The preceding theo\-rem permits to exhibit a huge class of left scalar multiplications of $\sH$ satisfying the hypotheses of Theorem \ref{propL} for any given normal operator $T \in \gB(\sH)$. Let $T$ be such an operator, let $\cP=(P,L)$ be an iqPVM over $\C^+_\imath$ in $\sH$ satisfying \eqref{decspet} for $T$, and let $\{\sH_\alpha\}_{\alpha \in A}$, $\{\nu_\alpha\}_{\alpha \in A}$ and $\{U_\alpha\}_{\alpha \in A}$ be as in the statement of Theorem \ref{teorapp}. Given any family $\{\gamma_\alpha:\C^+_\imath \to \C_\imath\}_{\alpha \in A}$ of $\nu_\alpha$-measurable functions $\gamma_\alpha$ with $|\gamma_\alpha(x)|=1$ $\nu_a$-a.e. in  $x \in \bC^+_\imath$, we can define the map $\bH \ni q \stackrel{\LL'}{\longmapsto} L'_q \in \gB(\sH)$ by setting
\[
\left(U_\alpha^{-1} L'_q \rr_{\sH_\alpha} U_\alpha \psi_\alpha\right)(x)=\overline{\gamma_\alpha(x)} \, q \, \gamma_\alpha(x) \, \psi_\alpha(x)
\]
$\nu_a$-a.e. in  $x \in \bC^+_\imath$. Bearing in mind \eqref{decspet}, Theorem \ref{TEO1} and Theorem \ref{teorapp}$(\mr{ii})$, it is immediate to verify that $\LL'$ is a left scalar multiplication of $\sH$ satisfying the hypotheses of Theorem \ref{propL}. Indeed, the inclusion map $\C^+_\imath \ni \alpha+\imath\beta \stackrel{\mi{{id}}}{\hookrightarrow} \alpha+\imath\beta \in \sH$ has the following properties $\imath\,\mi{id}=\mi{id}\,\imath$, $\jmath\,\mi{id}=\overline{\mi{id}}\,\jmath$ for every $\jmath \in \bS$ with $\imath\jmath=-\jmath \, \imath$ and $-\imath(\mi{id}-\overline{\mi{id}})=2\beta \geq 0$ on $\C^+_\imath$. Observe that each left scalar multiplication $\bH \ni q \mapsto L'_q$ of $\sH$ constructed in this way, via a choice of functions $\{\gamma_\alpha\}_{\alpha \in A}$ as above, coincides with $\bH \ni q \mapsto L_q$ on $\C^+_\imath$; that is, $L'_q=L_q$ for every $q \in \C^+_\imath$. \bs
\end{remark}

%%%

\subsection{Some simple examples} We present here some elementary examples concerning Theorems \ref{spectraltheorem} and \ref{teorapp}.

%%%

\subsubsection{Normal compact operators} We now consider the case of  normal compact operators. We recall that an operator $T \in \gB(\sH)$ is said to be \textit{compact} if it sends bounded subsets of $\sH$ into relatively compact subsets of $\sH$. The following result is contained in \cite{GhMoPe2}.

\begin{theorem}\label{MT2}
Let $\sH$ be a quaternionic Hilbert space and let $T \in \gB(\sH)$ be a compact normal operator. Then there exists  a Hilbert basis $N$ of $\sH$ made of eigenvectors of $T$ such that 
\beq \label{expcompact}
Tu=\sum_{z \in N}  z\lambda_z \b z|u \k \;\; \text{ for every $u \in \sH$,}
\eeq
where $\lambda_z \in \bH$ is an eigenvalue relative to the eigenvector $z$ for every $z \in N$ and the series converges absolutely. Moreover, if $\Lambda$ denotes the set $\{\lambda_z \in \bH \, | \, z \in N\}$, then it holds:
\begin{itemize}
 \item[$(\mr{a})$] For every $n \in \N \setminus \{0\}$, the set $\Lambda \cap \{q \in \bH \, | \, |q| \geq 1/n\}$ is finite. In particular, $\Lambda$ is at most countable and $0$ is the only possible accumulation point of $\Lambda$ in $\bH$.
 \item[$(\mr{b})$] For every $\lambda \in \Lambda \setminus \{0\}$, the set $\{z \in N \, | \, \lambda_z=\lambda\}$ is finite.
 \item[$(\mr{c})$] $\ssp(T) \setminus \{0\} = \Omega_\Lambda \setminus \{0\}$.
 \item[$(\mr{d})$] $0$ is the only possible element of $\sigma_{\mi{cS}}(T)$; that is, $\sigma_{\mi{cS}}(T) \subset \{0\}$.
\end{itemize} 
\end{theorem}

Fix $\imath \in \bS$ and, for every $z \in N$, choose $q_z \in \bS$ in such a way that $q_z^{-1} \lambda_z q_z \in \C_\imath^+$. Observe that the quaternion $q_z^{-1} \lambda_z q_z$ is an eigenvalue of $T$ with respect to the eigenvector $zq_z$. Indeed, it holds:
\[
T(zq_z)=T(z)q_z=z\lambda_zq_z=(zq_z)(q_z^{-1} \lambda_z q_z).
\]
Rename each vector $zq_z$ by $z$ and each quaternion $q_z^{-1} \lambda_z q_z$ by $\lambda_z$. In this way, in the statement of the preceding theorem, we can also assume that
\begin{center}
$\lambda_z \in \C^+_\imath$ for every $z \in N$.
\end{center}
Denote by $\bH \ni q \stackrel{\LL}{\lra} L_q$ the left scalar multiplication of $\sH$ associated with $N$ and define the qPVM $P:\mscr{B}(\C^+_\imath) \to \gB(\sH)$ by setting
\[
P(E)u:=\sum_{z \in N, \lambda_z \in E} z \b z|u \k \;\; \text{ for every $E \in \mscr{B}(\bC^+_\imath)$}.
\]

The reader observes that the pair $\cP:=(P,\LL)$ is an iqPVM over $\C^+_\imath$ in $\sH$; that is, the operators $P(E)$ and $L_q$ commute for every $E \in \mscr{B}(\C^+_\imath)$ and for every $q \in \bH$. Indeed, it holds:
\begin{center}
$P(E)L_qz=zq=L_qP(E)z\;$ for every $z \in N$ with $\lambda_z \in E$
\end{center}
and
\begin{center}
$P(E)L_qz=0=L_qP(E)z\;$ for every $z \in N$ with $\lambda_z \not\in E$.
\end{center}

Let us verify that $T=\int_{\C^+_\imath}\mi{id}\diff\cP$. For every $n \in \N \setminus \{0\}$, indicate by $s_n:\C^+_\imath \to \bH$ the simple function defined by setting
\[
s_n:=\sum_{z \in N, |\lambda_z| \geq 1/n} \lambda_z\Chi_{\{\lambda_z\}}.
\]
Since the sequence $\{s_n\}_{n \in \N \setminus \{0\}}$ converges uniformly to $\mi{id}$ on $\supp(P)=\Lambda$, by the very definition of integral with respect to an iqPVM and by \eqref{expcompact}, it turns out that
\begin{align*}
\int_{\bC_\imath^+} \mi{id}\diff\cP u &=\lim_{n \to +\infty} \int_{\bC_\imath^+} s_n\diff\cP u=\lim_{n \to +\infty} \sum_{z \in N, |\lambda_z| \geq 1/n} z \lambda_z \b z|u \k\\
&=\sum_{z \in N}z \lambda_z \b z|u \k=Tu,
\end{align*}
as desired. We underline that $P$ is the unique qPVM over $\C^+_\imath$ in $\sH$ associated with $T$ by the Spectral Theorem \ref{spectraltheorem}.

%%%

\subsubsection{Two matricial examples}\label{MEx}
Consider the finite dimensional quaternionic Hilbert space $\sH=\bH \oplus \bH$, equipped with the standard right scalar multiplication $(r,s)q:=(rq,sq)$ and with the standard Hermitian scalar product:
\[
\b (r,s)|(u,v) \k:=\overline{r}u+\overline{s}v \quad \text{if $r,s,u,v \in \bH$.}
\]

(1) Let $T:\sH \to \sH$ be the (right $\bH$-)linear operator represented by the matrix
\[
T=
\begin{bmatrix}
0 & i \\
j & 0
\end{bmatrix}
\in M_{2,2}(\bH),
\]
where matrix multiplication is that on the left; namely, $T(r,s)=(is,jr)$ for every $(r,s) \in \sH$. The operator $T$ is unitary, with spherical spectrum 
\[
\ssp(T)=\sigma_{pS}(T)=\left\{q=q_0+q_1i+q_2j+q_3k\in\bH \, \left| \, q_0^2=q_1^2+q_2^2+q_3^2=\frac12 \right.\right\}.
\]
The decomposition $T=A+JB$ of Theorem~\ref{thm:first-5.9} is obtained by setting
\[A=(T+T^*)\frac{1}{2}=
\begin{bmatrix}
0 & \frac{i-j}2 \\
\frac{j-i}2 & 0
\end{bmatrix},\quad
B=|T-T^*|\frac{1}{2}=
\begin{bmatrix}
\frac{\sqrt2}2 & 0 \\
0 & \frac{\sqrt2}2
\end{bmatrix},\quad 
J=
\begin{bmatrix}
0 & \frac{i+j}{\sqrt2} \\
\frac{i+j}{\sqrt2} & 0
\end{bmatrix}.
\]
Note that, in this case, $J=(T-A)B^{-1}$ is uniquely determined.

Fix $\imath=i$. %Define $\lambda_1:=\frac1{\sqrt2}+\frac1{\sqrt2}i$ and $\lambda_2:=-\frac1{\sqrt2}+\frac1{\sqrt2}i$.
Since $\ssp(T) \cap \bC_i^+=\{\lambda_1:=\frac1{\sqrt2}+\frac1{\sqrt2}i,\lambda_2:=-\frac1{\sqrt2}+\frac1{\sqrt2}i\}$, an easy computation shows that a Hilbert basis $N$ of $\sH^{Ji}_+$ (and then of $\sH$) constituted by eigenvectors of $T$ is given by the vectors
\[u_1=\left(\frac{i+k}2,\frac{1+i+j-k}{2\sqrt2}\right),
\quad u_2=\left(\frac{i-k}2,\frac{-1+i+j+k}{2\sqrt2}\right).\]
Therefore the left scalar multiplication induced by $N$ is the map defined by 
\[
L_q=u_1 q\b u_1|\ \cdot\ \k+u_2 q\b u_2|\ \cdot\ \k=
\begin{bmatrix}
q_0-q_2j & \frac{-q_3+q_1i+q_1j-q_3k}{\sqrt2} \\
\frac{q_3+q_1i+q_1j-q_3k}{\sqrt2} & q_0+q_2i 
\end{bmatrix}
\]
for every $q=q_0+q_1i+q_2j+q_3k \in \bH$. Moreover, the qPVM $P:\mscr{B}(\C^+_i) \to \gB(\sH)=M_{2,2}(\bH)$ associated with $T$ is given by 
\[
P(E):=\sum_{\mu=1}^2 \epsilon_\mu P_\mu \;\; \text{ for every $E \in \mscr{B}(\bC^+_\imath)$}, 
\]
where $\epsilon_\mu=1$ if $\lambda_\mu\in E$ and $\epsilon_\mu=0$ otherwise, and 
\[
P_1=u_1\b u_1|\ \cdot\ \k=
\frac12
\begin{bmatrix}
1 & \frac{i-j}{\sqrt2} \\
\frac{j-i}{\sqrt2} & 1
\end{bmatrix},\quad
P_2=u_2\b u_2|\ \cdot\ \k=
\frac12
\begin{bmatrix}
1 & \frac{j-i}{\sqrt2} \\
\frac{i-j}{\sqrt2} & 1
\end{bmatrix}.
\]
The Spectral Theorem~\ref{spectraltheorem} takes then the following expression
\[T=P_1L_{\lambda_1}+P_2L_{\lambda_2}=L_{\lambda_1}P_1+L_{\lambda_2}P_2.
\]

(2)
The right $\bH$-linear operator on $\sH$ represented by the matrix
\[
S=
\begin{bmatrix}
0 & i \\
-i & 0
\end{bmatrix},
\]
is self--adjoint, with spherical spectrum $\ssp(S)=\sigma_{pS}(S)=\{\pm1\}$. The decomposition $S=A+JB$ of Theorem~\ref{thm:first-5.9} is obtained by setting $A=S$ and $B=0$. Since $\Ker(B) \neq \{0\}$, the operator $J$ is not uniquely determined. We can take as $J$ the genuine \emph{left} multiplication by $i$ in $\sH$:
\[J=\begin{bmatrix}
i & 0 \\
0 & i
\end{bmatrix}.
\]

Fix $\imath=i$. Then $\sH^{Ji}_+=\bC_i\oplus\bC_i$. Since $\ssp(S)\cap\bC_i^+=\{\lambda_1:=1,\lambda_2:=-1\}$, a Hilbert basis $N$ of $\sH^{Ji}_+$ (and then of $\sH$) constituted by eigenvectors of $S$ is given by the vectors
\[
u_1=\left(\frac1{\sqrt2},-\frac{i}{\sqrt2}\right),
\quad
u_2=\left(\frac1{\sqrt2},\frac{i}{\sqrt2}\right).
\]
The left scalar multiplication induced by $N$ is the map 
\[
L_q=u_1 q\b u_1|\ \cdot\ \k+u_2 q\b u_2|\ \cdot\ \k=
\begin{bmatrix}
q & 0\\
0 & -iqi
\end{bmatrix}
\]
for every $q \in \bH$. The qPVM $P:\mscr{B}(\C^+_i) \to \gB(\sH)=M_{2,2}(\bH)$ associated with $S$ is given by 
\[
P(E):=\sum_{\mu=1}^2 \epsilon_\mu P_\mu \;\; \text{ for every $E \in \mscr{B}(\bC^+_\imath)$}, 
\]
where $\epsilon_\mu=1$ if $\lambda_\mu\in E$ and $\epsilon_\mu=0$ otherwise, and 
\[
P_1=u_1\b u_1|\ \cdot\ \k=
\frac12
\begin{bmatrix}
1 & i \\
-i & 1
\end{bmatrix},\quad
P_2=u_2\b u_2|\ \cdot\ \k=
\frac12
\begin{bmatrix}
1 & -i \\
i & 1
\end{bmatrix}.
\]
The Spectral Theorem~\ref{spectraltheorem} takes then the classical form
\[S=P_1L_{\lambda_1}+P_2L_{\lambda_2}=P_1-P_2.
\]
Observe that, even if $S$ is $\bC_i$-valued, as are $P_1$ and $P_2$, the left scalar multiplication associated with $S$ cannot be equal for each $q\in\bH$ to the standard left multiplication by $q$, since then the map would not satisfy the conditions of Theorem~\ref{propL}.

%%%

\subsubsection{The quaternionic $L^2$-space on $[0,1]$}

%Another elementary example is as follows.
Consider the quaternionic Hilbert space $\sL:=L_2([0,1],\bH;\mk{m})$, where $\mk{m}$ is the standard Lebesgue measure on $[0,1]$. Define the bounded operator $T:\sL \to \sL$ by setting
\[
(T\psi)(t):=t\psi(t) \;\; \text{ $\mk{m}$-a.e. in $t \in [0,1]$ for every $\psi \in \sL$.}
\]

It is immediate to verify that $T$ is self-adjoint and hence normal. Thanks to points $(\mr{a})(\mr{ii})$ and $(\mr{b})$ of Theorem 4.8 in \cite{GhMoPe}, $\ssp(T)$ is a compact subset of $\R$ and it coincides with $\sigma_{\mi{cS}}(T)$. Since $\big(\Delta_q(T)^{-1}\psi\big)(t)=(t-q)^{-2}\psi(t)$ for every $q \in \R$, it follows immediately that $\ssp(T)=\sigma_{\mi{cS}}(T)=[0,1]$.

Fix $\imath \in \bS$. Define the left scalar multiplication $\bH \ni q \stackrel{\LL}{\longmapsto} L_q \in \gB(\sL)$ and the qPVM $P:\mscr{B}(\C^+_\imath) \to \gB(\sL)$ by setting
\[
(L_q\psi)(t):=q\psi(t)
\;\; 
\text{ and }
\;\;
(P(E)\psi)(t):=\Chi_{E \cap [0,1]}(t)\psi(t)
\]
for every $q \in \bH$, for every $E \in \mscr{B}(\C^+_\imath)$ and for every $t \in [0,1]$. Since each function $\Chi_{E \cap [0,1]}$ is real-valued, the pair $\cP:=(P,\LL)$ is an iqPVM over $\C^+_\imath$ in $\sL$. We have that $\cP$ satisfies the spectral theorem for $T$; that is, $T=\int_{\C^+_\imath}\mi{id}\diff\cP$. The latter result can be proved following the proof of Theorem \ref{teorapp}, where $A=\{1\}$, $\sH_1=\sL$, $\nu_1(E)=\mk{m}(E \cap [0,1])$ and $U_1$ is the identity map on $\sL$.

%%%%%%%

\section{Integral of unbounded measurable functions w.r.t. an iqPVM} \label{sec:unb-int}

Fix a quaternionic Hilbert space $\sH$ and $\imath \in \bS$. Denote by $M(\C^+_\imath,\bH)$ the \textit{set of all (possibly unbounded Borel) measurable $\bH$-valued functions on $\C^+_\imath$}. Choose a qPVM $P:\mscr{B}(\C^+_\imath) \to \gB(\sH)$ over $\C^+_\imath$ in $\sH$. Given any $f \in M(\C^+_\imath,\bH)$, we define 
\[
\|f\|_\infty^{\sss(P)}:=\inf\big\{k \in \R^+ \big|\, P(\{q \in \bC_\imath^+ \, | \, |f(q)|>k\}) = 0\big\} \in \R^+ \cup \{+\infty\},
\]
where $\inf \emptyset:=+\infty$. We say that $f$ is \textit{$P$-essentially bounded} if $\|f\|^{\sss(P)}_\infty<+\infty$. Recall that $M_b(\C^+_\imath,\bH)$ indicates the subset of $M(\C^+_\imath,\bH)$ consisting of all bounded measurable $\bH$-valued functions on $\C^+_\imath$.

In what follows, we will use the finite positive regular Borel measures over $\mscr{B}(\bC^+_\imath)$ defined in \eqref{measuremuPx}. We point out that
\beq \label{summes}
\mu^{\sss(P)}_{qu}(E)=|q|^2\mu^{\sss(P)}_u(E)
\quad \text{ and } \quad
\mu^{\sss(P)}_{u+v}(E) \leq 2\big(\mu^{\sss(P)}_u(E)+\mu^{\sss(P)}_v(E)\big)
\eeq
for every $E \in \mscr{B}(\C^+_\imath)$, for every $q \in \bH$ and for every $u,v \in \sH$. Indeed, it holds:
\[
\mu^{\sss(P)}_{qu}(E)=\|P(E)L_qu\|^2=\|L_qP(E)u\|^2=|q|^2\|P(E)u\|^2=|q|^2\mu^{\sss(P)}_u(E)
\]
and
\[
\mu^{\sss(P)}_{u+v}(E)=\|P(E)(u)+P(E)(v)\|^2 \leq 2\mu^{\sss(P)}_u(E)+2\mu^{\sss(P)}_v(E).
%\mu^{\sss(P)}_{u+v}(E)=\|P(E)(u+v)\|^2 \leq 2\|P(E)u\|^2+2\|P(E)v\|^2=2\mu^{\sss(P)}_u(E)+2\mu^{\sss(P)}_v(E).
\]

\begin{proposition}\label{defintunb}
Let $\cP=(P,\mc{L})$ be an iqPVM over $\C^+_\imath$ in $\sH$ and let $f \in M(\C^+_\imath,\bH)$. Denote by $D_f$ the set $\{u \in \sH \, | \, f \in L^2(\C_\imath^+,\bH;\mu_u^{\sss(P)})\}$. The following facts hold.
\begin{itemize}
 \item[$(\mr{a})$] $D_f$ is a dense (right $\bH$-linear) subspace of $\sH$.
 \item[$(\mr{b})$] For every $n \in \N$, define the function $f_n \in M_b(\C^+_\imath,\bH)$ by setting $f_n(q):=f(q)$ if $|f(q)| \leq n$ and $f_n(q):=0$ otherwise. Then, for every $u \in D_f$, the sequence $\{f_n\}_{n \in \N}$ converges to $f$ in $L^2(\C^+_\imath,\bH;\mu^{\sss(P)}_u)$ and the sequence $\big\{\int_{\C^+_\imath}f_n\diff\cP u\big\}_{n \in \N}$ converges in $\sH$. Define the operator $\int_{\C^+_\imath}f\diff\cP:D_f \to \sH$ as follows:
\beq \label{intill}
\int_{\C^+_\imath}f\diff\cP u:=\lim_{n \to +\infty} \int_{\C_\imath^+} f_n \diff\cP u \;\; \text{ for every $u \in D_f$}.
\eeq
 \item[$(\mr{c})$] If $\{g_n\}_{n \in \N}$ is another sequence of $P$-essentially bounded functions in $M(\C^+_\imath,\bH)$ converging to $f$ in $L^2(\C^+_\imath,\bH;\mu^{\sss(P)}_u)$ for some $u \in D_f$, then the sequence $\big\{\int_{\C^+_\imath}g_n\diff\cP u\big\}_{n \in \N}$ converges to $\int_{\C^+_\imath}f\diff\cP u$ in $\sH$.
%This fact ensures that the operator $\int_{\C^+_\imath}f\diff\cP$ does not depend on the choice of the sequence of $P$-essentially bounded functions in $M(\C^+_\imath)$ converging to $f$ in $L^2(\C^+_\imath,\bH;\mu^{\sss(P)}_u)$ used to define it.
 \item[$(\mr{d})$] For every $f \in M(\C^+_\imath,\bH)$ and for every $q \in \bH \setminus \{0\}$, we have that $D_{qf}=D_f=D_{fq}$, $L_q(D_f)=D_f$ and
\beq \label{ccq}
L_q\int_{\C^+_\imath}f\diff\cP=\int_{\C^+_\imath}qf\diff\cP
\;\; \text{ and }\;\;
\left(\int_{\C^+_\imath}f\diff\cP\right)L_q=\int_{\C^+_\imath}fq\diff\cP.
\eeq
Evidently, the second equality remains true if $q=0$.

 \item[$(\mr{e})$] For every $f \in M(\C^+_\imath,\bH)$ and for every $u \in D_f$, we have that
\beq \label{normunb}
\left\|\int_{\C_\imath^+} f \diff\cP u\right\|^2 =  \int_{\C_\imath^+} |f|^2 \diff\mu^{\sss(P)}_u \eeq
and 
\beq \label{normunb2}
\left \b u \left| \int_{\C_\imath^+} f \diff\cP u\right. \right\k=\int_{\C_\imath^+} f \diff\mu^{\sss(P)}_u \;\; \mbox{ if $f$ is real-valued}.
\eeq
\end{itemize}
\end{proposition}

\begin{proof}
$(\mr{a})$ We begin by defining the pairwise disjoint sets $\{E_n\}_{n \in \N}$ in $\mscr{B}(\C_\imath^+)$ and the closed subspaces $\{\sH_n\}_{n \in \N}$ as follows:
$E_n:=\{q \in \C^+_\imath \, | \, n \leq |f(q)|<n+1\}$ and  $\sH_n :=P(E_n)(\sH)$. Since $\sum_{n \in \N}P(E_n)=P(\bigcup_{n \in \N} E_n)=P(\C^+_\imath)=\1$, the span of $\bigcup_{n \in \N}\sH_n$ is dense in $\sH$. It is now sufficient to show that $\sH_n \subset D_f$ for every $n \in \N$. Thanks to the monotone convergence theorem, we know that
\[%\beq \label{summon}
\int_{\C_\imath^+} |f|^2 \diff\mu^{\sss(P)}_u=\sum_{n \in \N} \int_{\C_\imath^+} |\Chi_{E_n}f|^2 \diff\mu^{\sss(P)}_u \in \R \cup \{+\infty\}.
\]%\eeq
On the other hand, Theorem \ref{TEO1} ensures that
\begin{align*}
\int_{\C_\imath^+} |\Chi_{E_n}f|^2 \diff\mu^{\sss(P)}_u &=\left\langle \int_{\C_\imath^+} \Chi_{E_n} f \diff\cP u \left| \int_{\C_\imath^+} \Chi_{E_n} f \diff\cP u\right.  \right\rangle\\
&=\left\langle \int_{\C_\imath^+} \Chi_{E_n} f \diff\cP u \left| \int_{\C_\imath^+} \Chi_{E_n} f \Chi_{E_n} \diff\cP u\right.\right\rangle\\
&=\left\langle \int_{\C_\imath^+} \Chi_{E_n} f \diff\cP u \left| \int_{\C_\imath^+} \Chi_{E_n} f \diff\cP  \int_{\C_\imath^+} \Chi_{E_n} \diff\cP u\right.\right\rangle\\
&=\left\langle \int_{\C_\imath^+} |\Chi_{E_n} f|^2 \diff\cP u \left|\int_{\C_\imath^+} \Chi_{E_n} \diff\cP u \right.\right\rangle\\
&=\left\langle\left.\int_{\C_\imath^+} |\Chi_{E_n}f|^2 \diff\cP u \right| P(E_n)u \right\rangle.
\end{align*}

Therefore, if $u \in \sH_m$ for some $m \in \N$, then $P(E_n)u=0$ for every $n \in \N \setminus \{m\}$ and hence
\begin{align*}
\int_{\C_\imath^+} |f|^2 \diff\mu^{\sss(P)}_u&=\sum_{n \in \N} \int_{\C_\imath^+} |\Chi_{E_n}f|^2 \diff\mu^{\sss(P)}_u=\int_{\C_\imath^+} |\Chi_{E_m}f|^2 \diff\mu^{\sss(P)}_u\\
& \leq (m+1)^2\mu^{\sss(P)}_u(\C^+_\imath) \leq (m+1)^2\|u\|^2 < +\infty.
\end{align*}

This proves that $\bigcup_{n \in \N}\sH_n \subset D_f$ and hence $D_f$ turns out to be dense in $\sH$.

The fact that $D_f$ is a subspace of $\sH$ follows immediately from \eqref{summes}.

$(\mr{b})$ Let $u \in D_f$. Exploiting Theorem \ref{TEO1} again, we infer that
\[
\left\|\int_{\C_\imath^+} f_n \diff\cP u -   \int_{\C_\imath^+} f_m \diff\cP u\right\|^2=
\int_{\bC_\imath^+} |f_n-f_m|^2 d\mu^{\sss(P)}_u.
\]
Since the sequence $\{f_n\}_{n \in \N}$ converges to $f$ in $L^2(\C^+_\imath,\bH;\mu^{\sss(P)}_u)$, the latter equality implies that the sequence $\{\int_{\C^+_\imath}f_n \diff\cP u\}_{n \in \N}$ in $\sH$ is a Cauchy sequence. This fact assures the existence of the limit in (\ref{intill}). Bearing in mind that each restriction $\int_{\C_\imath^+} f_n \diff\cP\big|_{D_f}$ is right $\bH$-linear, it follows immediately that the limit $\int_{\C_\imath^+} f \diff\cP$ is right $\bH$-linear as well.

$(\mr{c})$ If $\{g_n\}_{n \in \N}$ is a sequence of $P$-essentially bounded functions in $M(\C^+_\imath,\bH)$ converging to $f$ in $L^2(\C^+_\imath,\bH;\mu^{\sss(P)}_u)$ for some $u \in D_f$, then $\{f_n-g_n\}_{n \in \N}$ converges to $0$ in $L^2(\C^+_\imath,\bH;\mu^{\sss(P)}_u)$ and hence
\[
\left\|\int_{\C_\imath^+} f_n \diff\cP u-\int_{\C_\imath^+} g_n \diff\cP u\right\|=\left(\int_{\bC_\imath^+} |f_n-g_n|^2 \diff\mu^{\sss(P)}_u\right)^{\frac{1}{2}} \to 0 \;\; \text{ in $\sH$.}
\]

$(\mr{d})$ It is immediate to verify that $D_{qf}=D_f=D_{fq}$. By Theorem \ref{TEO1}$(\mr{a})(\mr{i})$, we have that
\beq \label{ccq'}
L_q\int_{\C^+_\imath}f_n\diff\cP u=\int_{\C^+_\imath}qf_n\diff\cP u
\; \text{ and } \; \left(\int_{\C^+_\imath}f_n\diff\cP\right)L_q u=\int_{\C^+_\imath}f_nq\diff\cP u
\eeq
for every $n \in \N$ and for every $u \in D_f$. Since $\mu^{\sss(P)}_{qu}(E)=|q|^2\mu^{\sss(P)}_u(E)$ for every $E \in \mscr{B}(\C^+_\imath)$ and for every $u \in \sH$, we have that $L_q(D_f) \subset D_f$ (and hence $L_q(D_f)=D_f$). In this way, we can perform the limits in \eqref{ccq'} as $n \to +\infty$, obtaining \eqref{ccq}. 

$(\mr{e})$ Thanks to \eqref{secondestimatef} and \eqref{intemediatestimatef}, equalities \eqref{normunb} and \eqref{normunb2} hold if $f$ is replaced by $f_n$ for an arbitrary $n \in \N$ and if $u \in \sH$. Now, taking the limits as $n \to +\infty$, we obtain \eqref{normunb} and \eqref{normunb2} for arbitrary $f \in M(\C^+_\imath,\bH)$ and $u \in D_f$. 
$(\mr{f})$ Define the constant sequence $\{g_n\}_{n \in \N}$ in $M_b(\C^+_\imath,\bH)$ by setting $g_n:=\Chi_Ef$ for every $n \in \N$. Since such a sequence satisfies the hypothesis of preceding point $(\mr{c})$ for every $u \in D_f$, we have that the integral $\int_{\C^+_\imath}f\diff\cP$ defined in \eqref{intill} coincides with $\int_Ef\diff\cP=\int_{\C^+_\imath}\Chi_Ef\diff\cP$ defined in \eqref{defintfpsi}.
\end{proof}

\begin{remark}$\null$\label{remarksupp2}
$(1)$ Let $\cP=(P,\mc{L})$ be an iqPVM over $\C^+_\imath$ in $\sH$, let $E \in \mscr{B}(\C^+_\imath)$ and let $f \in M(\C^+_\imath,\bH)$. We define
\beq \label{defintfFunb}
\int_E f \diff\cP:=\int_{\C_\imath^+} \Chi_E f \diff\cP.
\eeq
It holds:
\beq \label{threeintunb}
\int_E f \diff\cP=\int_{E \cap \, \supp(P)} f \diff\cP.
\eeq
If $g$ is a function in $M(\C^+_\imath,\bH)$ with  $\|f-g\|_\infty^{\sss(P)}=0$, then 
\[
\int_E f \diff\cP=\int_E g \diff\cP.
\]
All the preceding equalities can be proved by observing that they are evident in the bounded case and survive the limit process \eqref{intill}.

If $h: E \to \bH$ is a possibly unbounded measurable function on $E$, then we define
\beq 	\label{eq:int-restr}
\int_E h \diff\cP:=\int_{\C_\imath^+} \widetilde{h} \diff\cP,
\eeq
where $\widetilde{h}:\C_\imath^+  \to \bH$ extends $h$ to the null function outside $E$. We denote the integral $\int_E h \diff\cP$ also by the symbol $\int_E h(z) \diff\cP(z)$, which is useful in the case we desire to write the integrand $h$ by an explicit expression.

$(2)$ Let $E \in \mscr{B}(\C^+_\imath)$, let $h: E \to \bH$ be a possibly unbounded measurable function on $E$ and let $\cP=(P,\mc{L})$ and $\cP'=(P,\mc{L}')$ be iqPVMs over $\C^+_\imath$ in $\sH$. If $h$ is real-valued, then
\[
\int_{\C^+_\imath}h\diff\cP=\int_{\C^+_\imath}h\diff\cP'.
\]
This equality follows immediately from definitions \eqref{eq:int-restr} and \eqref{intill}, and from Remark \ref{remarksupp}$(4)$. \bs
\end{remark}

We can now extend Theorem \ref{TEO1}, stating and proving  the general properties of the operators  $\int_{\C_\imath^+}f \diff\cP$ for $f \in M(\C^+_\imath,\bH)$.

\begin{theorem}\label{TEO1unb}
Let $\cP=(P,\mc{L})$ be an iqPVM over $\C^+_\imath$ in $\sH$ and let $f$ and $g$ be arbitrary functions in $M(\C^+_\imath,\bH)$. The following hold.
\begin{itemize}
 \item[$(\mr{a})$] $D\big(\int_{\C_\imath^+} f \diff\cP  \int_{\C_\imath^+} g \diff\cP \big)=D_g \cap D_{fg}$ and 
\beq \label{produnb}
\int_{\C_\imath^+} f \diff\cP \int_{\C_\imath^+} g \diff\cP \subset \int_{\C_\imath^+} fg \diff\cP.
\eeq
We can replace the latter inclusion with an equality if and only if $D_{g} \supset D_{fg}$.
  \item[$(\mr{b})$] $D\big(\int_{\C_\imath^+} f \diff\cP + \int_{\C_\imath^+} g \diff\cP \big)=D_f \cap D_g$ and
\beq \label{sumunb}
\int_{\C_\imath^+} f \diff\cP + \int_{\C_\imath^+} g \diff\cP \subset \int_{\C_\imath^+} (f+g) \diff\cP.
\eeq
We can replace the latter inclusion with an equality if and only if $D_f \cap D_g \supset D_{f+g}$.
 \item[$(\mr{c})$] $D_f=D_{\overline{f}}$ and
\beq \label{aggcomp}
\left(\int_{\C_\imath^+} f \diff\cP \right)^*= \int_{\C_\imath^+} \overline{f} \diff\cP.
\eeq
In particular, we have that $\int_{\C_\imath^+} f \diff\cP$ is always a closed operator.
 \item[$(\mr{d})$] The operator $\int_{\C_\imath^+} f \diff\cP$ is normal and
\begin{align}
\left(\int_{\C_\imath^+} f \diff\cP \right)^*
\left(\int_{\C_\imath^+} f \diff\cP \right)&=
\int_{\C_\imath^+} |f|^2 \diff\cP \nonumber\\
&=\left(\int_{\C_\imath^+} f \diff\cP \right)\left(\int_{\C_\imath^+} f \diff\cP \right)^*. \label{aggcomp2}
\end{align}
 \item[$(\mr{e})$] $D_f = \sH$ if and only if $f$ is $P$-essentially bounded. In particular, \eqref{produnb} holds with the equality if $g$ is $P$-essentially bounded. Similarly, \eqref{sumunb} holds with the equality if at least one of $f$ and $g$ is $P$-essentially bounded.
 \item[$(\mr{f})$]
 Let $\jmath \in \bS$, let $\Psi :\C_\imath^+ \to \C^+_\jmath$ be a measurable map and let $Q:\C^+_\jmath \to \gB(\sH)$ be the map defined by setting $Q(E):=P(\Psi^{-1}(E))$ for every $E \in \mscr{B}(\C^+_\jmath)$. Then $\mc{Q}:=(Q,\mc{L})$ is an iqPVM over $\C^+_\jmath$ in $\sH$ such that
\beq \label{PQ2}
\int_{\bC_\imath^+} \xi \circ \Psi \diff\cP= \int_{\bC_\jmath^+} \xi \diff\mc{Q}
\eeq
for every $\xi \in M(\C^+_\jmath,\bH)$.
\end{itemize}
\end{theorem}

\begin{proof}
$(\mr{a})$ First, suppose that $\|f\|^{\sss(P)}_\infty < +\infty$. Define the sequence $\{g_n\}_{n \in \N}$ in $M_b(\C^+_\imath,\bH)$ by setting $g_n(q):=g(q)$ if $|g(q)| \leq n$ and $g_n(q):=0$ otherwise. Let $u \in D_g$. Bearing in mind points $(\mr{b})$ and $(\mr{c})$ of Proposition \ref{defintunb}, point $(\mr{a})(\mr{ii})$ of Theorem \ref{TEO1} and the fact that $\int_{\C_\imath^+} f \diff\cP \in \gB(\sH)$, we infer at once that
\begin{align*}
\int_{\C_\imath^+} f \diff\cP \int_{\C_\imath^+} g \diff\cP u &= \int_{\C_\imath^+} f \diff\cP \lim_{n \to +\infty}\int_{\C_\imath^+} g_n \diff\cP u\\
&=\lim_{n \to +\infty} \int_{\C_\imath^+} f \diff\cP \int_{\C_\imath^+} g_n \diff\cP u \\
&= \lim_{n \to +\infty} \int_{\C_\imath^+} f g_n \diff\cP u = \int_{\C_\imath^+} fg \diff\cP u.
\end{align*}
This proves that
\beq \label{itermprod}
\int_{\C_\imath^+} f \diff\cP \int_{\C_\imath^+} g \diff\cP u = \int_{\C_\imath^+} fg \diff\cP u \quad \mbox{if $u \in D_g$ and $\|f\|^{\sss(P)}_\infty <+\infty$}.
\eeq
Consequently, if $G:=\int_{\C_\imath^+} g \diff\cP$, \eqref{normunb} and \eqref{itermprod} imply
\beq \label{itermprod2}
\int_{\bC_\imath^+} |f|^2 d\mu^{\sss(P)}_{Gu} = \int_{\bC_\imath^+} |fg|^2 d\mu^{\sss(P)}_u\quad \mbox{if $u \in D_g$ and $\|f\|^{\sss(P)}_\infty<+\infty$.}
\eeq

Let $f$ be an arbitrary function in $M(\C^+_\imath,\bH)$ and let $\{f_n\}_{n \in \N}$ be the sequence in $M_b(\C^+_\imath,\bH)$ defined by setting $f_n(q):=f(q)$ if $|f(q)| \leq n$ and $f_n(q):=0$ otherwise. Since \eqref{itermprod2} applies to each $f_n$, the monotone convergence theorem assures that $Gu \in D_f$ if and only if $u \in D_{fg}$ and hence $D\big(\int_{\C^+_\imath}f \diff\cP\int_{\C^+_\imath}g \diff\cP\big)=D_g \cap D_{fg}$. 

Fix $u \in D_g \cap D_{fg}$. Thanks to the dominated convergence theorem, we infer that $\{f_n\}_{n \in \N}$ converges to $f$ in $L^2(\C^+_\imath,\bH;\mu^{\sss(P)}_{Gu})$ and $\{f_ng\}_{n \in \N}$ converges to $fg$ in $L^2(\C^+_\imath,\bH;\mu^{\sss(P)}_u)$. In particular, by \eqref{normunb}, it follows also that
\[
\lim_{n \to +\infty} \int_{\C_\imath^+} f_n \diff\cP Gu=\int_{\C_\imath^+} f \diff\cP Gu
\]
and
\[
\lim_{n \to +\infty} \int_{\C_\imath^+} f_ng \diff\cP u=  \int_{\C_\imath^+} fg \diff\cP u.
\]
Bearing in mind that \eqref{itermprod} holds if $f$ is replaced by each $f_n$, we obtain:
\begin{align*}
\int_{\C_\imath^+} f \diff\cP \int_{\C_\imath^+} g \diff\cP u&=\lim_{n \to +\infty} \int_{\C_\imath^+} f_n \diff\cP Gu\\
&=\lim_{n \to +\infty} \int_{\C_\imath^+} f_ng \diff\cP u=  \int_{\C_\imath^+} fg \diff\cP u.
\end{align*}

$(\mr{b})$ The proof of this point is  analogous to the one of $(\mr{a})$.

$(\mr{c})$ Let $f \in M(\C^+_\imath,\bH)$ and let $\{f_n\}_{n \in \N}$ be the sequence in $M_b(\C^+_\imath,\bH)$ defined as above. The equality $D_f=D_{\overline{f}}$ follows immediately from the definitions of $D_f$ and~$D_{\overline{f}}$.

Let us prove \eqref{aggcomp}. % which immediately implies that $\int^{(L)}_{\bC_\imath^+} f dP = ( \int^{(L)}_{\bC_\imath^+} \overline{ f} dP)^*$ is closed since the adjoint of an operator is always closed by construction.
Let $u \in D_f$ and let $v \in D_{\overline{f}}=D_f$. By the definitions of $\int_{\C_\imath^+} f \diff\cP$ and $\int_{\C_\imath^+} \overline{f} \diff\cP$, and by point $(\mr{a})(\mr{iii})$ of Theorem \ref{TEO1}, we infer that
\begin{align*}
\left\langle v \left| \int_{\C_\imath^+} f \diff\cP u \right\rangle\right. &=
\lim_{n \to +\infty} \left\langle v \left| \int_{\C_\imath^+} f_n \diff\cP u \right\rangle\right.
=\lim_{n \to +\infty} \left\langle \left.\int_{\C_\imath^+} \overline{f_n} \diff\cP v \right| u \right\rangle\\
&=\left\langle \left.\int_{\C_\imath^+} \overline{f} \diff\cP v \right| u \right\rangle.
\end{align*}
This implies that $v \in D\big(\big(\int_{\C_\imath^+} f \diff\cP\big)^*\big)$ and $\int_{\C_\imath^+} \overline{f} \diff\cP \subset \big(\int_{\C_\imath^+} f \diff\cP\big)^*$.

To complete the proof of \eqref{aggcomp}, we have to show that $D\big(\big(\int_{\C_\imath^+} f \diff\cP\big)^*\big) \subset D_f$. 

For every $n \in \N$, define $F_n:=\{q \in \C^+_\imath \, | \, |f(q)| \leq n\}$. Since $f_n= f\Chi_{F_n}$, point $(\mr{b})$ implies that
\beq \label{eq:Fn}
\int_{\C_\imath^+} f_n\diff\cP=\int_{\C_\imath^+} f\diff\cP \int_{\C_\imath^+} \Chi_{F_n}\diff\cP.
\eeq
Bearing in mind that the operator $\int_{\C_\imath^+} \Chi_{F_n}\diff\cP=P(F_n)$ is bounded and self-adjoint, we obtain (see point $(\mr{v})$ of Remark 2.16 of \cite{GhMoPe}):
\[
\int_{\C_\imath^+} \Chi_{F_n}\diff\cP  \left(\int_{\C_\imath^+} f \diff\cP\right)^*
\subset 
\left(\int_{\C_\imath^+} f \diff\cP  \int_{\C_\imath^+} \Chi_{F_n}\diff\cP \right)^*.
\]
Combining the latter inclusion with \eqref{eq:Fn} and with point $(\mr{a})(\mr{iii})$ of Theorem \ref{TEO1}, we infer that
\[
\int_{\C_\imath^+} \Chi_{F_n}\diff\cP \left(\int_{\C_\imath^+} f \diff\cP\right)^*
\subset \left(\int_{\C_\imath^+} f_n \diff\cP\right)^* = \int_{\C_\imath^+} \overline{f_n} \diff\cP.
\]
In particular, given any $x \in D\big(\big(\int_{\C_\imath^+} f \diff\cP\big)^*\big)$, it holds
\[
\int_{\C_\imath^+} \Chi_{F_n}\diff\cP y =   \int_{\C_\imath^+} \overline{f_n} \diff\cP x,
\]
where $y:=\big(\int_{\C_\imath^+} f \diff\cP\big)^*x$. Thanks to \eqref{normunb}, it follows that
\[
\int_{\C_\imath^+} |f_n|^2 \diff\mu^{\sss(P)}_x = \int_{\C_\imath^+} |\Chi_{F_n}|^2 \diff\mu^{\sss(P)}_y \leq \|y\|^2 \;\; \text{ for every $n \in \N$},
\]
which implies that $x \in D_f$. This proves  \eqref{aggcomp} and consequently, since the adjoint of an operator is always closed, it also proves that $\int_{\C_\imath^+} f \diff\cP $ is closed.

$(\mr{d})$ Recall that, given any $u \in \sH$, the positive Borel measure $\mu_u^{\sss(P)}$ on $\C^+_\imath$ is finite. In this way, we have that $L^4(\C_\imath^+,\bH;\mu_u^{\sss(P)}) \subset L^2(\C_\imath^+,\bH;\mu_u^{\sss(P)})$ and hence $D_{\overline{f}f}=D_{f\overline{f}} \subset D_f=D_{\overline{f}}$. Combining this fact with preceding points $(\mr{a})$ and $(\mr{c})$, we infer that
\begin{align*}
\left(\int_{\C_\imath^+} f \diff\cP \right)^*
\left(\int_{\C_\imath^+} f \diff\cP \right)&=\int_{\C_\imath^+} \overline{f} \diff\cP \int_{\C_\imath^+} f \diff\cP=\int_{\C_\imath^+} |f|^2 \diff\cP\\
&=\int_{\C_\imath^+} f \diff\cP \int_{\C_\imath^+} \overline{f} \diff\cP=\left(\int_{\C_\imath^+} f \diff\cP \right)\left(\int_{\C_\imath^+} f \diff\cP \right)^*.
\end{align*}
In particular,  the operator $\int_{\C_\imath^+} f \diff\cP$ is normal.

$(\mr{e})$ If $\|f\|^{\sss(P)}_\infty<+\infty$, then $D_f= \sH$ by point $(\mr{f})$ of Proposition \ref{defintunb}.

Assume that $D_f=\sH$. By point $(\mr{c})$, the operator  $\int_{\C_\imath^+} f \diff\cP$ is closed. In this way,  the closed graph theorem implies that $\int_{\C_\imath^+} f \diff\cP$ is bounded (see point $(\mr{f})$ of Proposition 2.11 in \cite{GhMoPe}). If $\{f_n\}_{n \in \N}$ and $\{F_n\}_{n \in \N}$ are the sequences defined above, then \eqref{firstestimatef} and point $(\mr{a})$ implies that
\begin{align*}
\|f_n\|^{\sss(P)}_\infty &= \left\|\int_{\C_\imath^+} f_n \diff\cP\right\|=\left\|\int_{\C_\imath^+} f \diff\cP \int_{\C_\imath^+} \Chi_{F_n} \diff\cP\right\|\\
&=\left\|\int_{\C_\imath^+} f \diff\cP P(F_n)\right\| \leq \left\|\int_{\C_\imath^+} f \diff\cP\right\| \, \|P(F_n)\| \leq \left\|\int_{\C_\imath^+} f \diff\cP\right\|
\end{align*}
for every $n \in \N$. Therefore, taking the limit as $n \to + \infty$, we obtain that $\|f\|_\infty^{\sss(P)} \leq  \big\|\int_{\C_\imath^+} f \diff\cP\big\|<+\infty$, as desired.

$(\mr{f})$ It is immediate to verify that $\mc{Q}=(Q,\mc{L})$ is an iqPVM over $\C^+_\jmath$ in $\sH$, and $\mu_u^{\sss(Q)}=\mu_u^{\sss(P)} \circ \psi^{-1}$ for every $u \in \sH$. Let $\xi \in M(\C_\jmath^+,\bH)$. From standard results concerning positive measures, we have that, for every $u \in \sH$, the integrals $\int_{\C_\imath^+} |\xi \circ \psi|^2 \diff\mu_u^{\sss(P)}$ and $\int_{\C_\jmath^+} |\xi|^2 \diff\mu_u^{\sss(Q)}$ coincides, also when they diverges. This fact implies that $D_{\xi \circ \psi}=D_\xi$.

Let $\{\xi_n\}_{n \in \N}$ be the sequence in $M_b(\C^+_\jmath,\bH)$ defined by setting $\xi_n(q):=\xi(q)$ if $|\xi(q)| \leq n$ and $\xi_n(q):=0$ otherwise. Thanks to point $(\mr{c})$ of Theorem \ref{TEO1} and to point $(\mr{f})$ of Proposition \ref{defintunb}, we know that \eqref{PQ2} holds for every $\xi_n$ and for every $u \in \sH$. Points $(\mr{b})$ and $(\mr{c})$ of Proposition \ref{defintunb} ensure that \eqref{PQ2} holds for $\xi$ if $u \in D_{\xi \circ \psi}=D_\xi$.
\end{proof}

As a first corollary, we improve point $(\mr{f})$ of Proposition \ref{defintunb}.

\begin{corollary} \label{cor:f}
Let $\cP=(P,\mc{L})$ be an iqPVM over $\C^+_\imath$ in $\sH$ and let $f \in M(\C^+_\imath,\bH)$. Then the following three conditions are equivalent:
\begin{itemize}
 \item[$(\mr{a})$] $D_f=\sH$.
 \item[$(\mr{b})$] $\int_{\C^+_\imath}f\diff\cP$ is bounded.
 \item[$(\mr{c})$] $f$ is $P$-essentially bounded.
\end{itemize}

Moreover, if the preceding conditions are satisfied, then $\int_{\C^+_\imath}f\diff\cP \in \gB(\sH)$ and $\|f\|^{\sss(P)}_\infty \leq \big\|\int_{\C^+_\imath}f\diff\cP\big\|$.
\end{corollary}

\begin{proof}
Implication $(\mr{a}) \Longrightarrow (\mr{b})$ follows immediately from the closed graph theorem (see Proposition 2.11$(\mr{f})$ of \cite{GhMoPe}). Implication $(\mr{c}) \Longrightarrow (\mr{a})$ is evident.

Let us prove $(\mr{b}) \Longrightarrow (\mr{c})$. For every $n \in \N$, define $E_n:=\{q \in \C^+_\imath \, | \, |f(q)| \leq n\} \in \mscr{B}(\C^+_\imath)$. By \eqref{firstestimatef}, Proposition \ref{defintunb}$(\mr{f})$ and Theorem \ref{TEO1unb}$(\mr{a})$, we have that
\begin{align*}
\|f\Chi_{E_n}\|^{\sss(P)}_\infty &=\left\|\int_{\C^+_\imath}f\Chi_{E_n}\diff\cP\right\|=\left\|\int_{\C^+_\imath}f\diff\cP\int_{\C^+_\imath}\Chi_{E_n}\diff\cP\right\|=\left\|\int_{\C^+_\imath}f\diff\cP \, P(E_n)\right\|\\
&\leq \left\|\int_{\C^+_\imath}f\diff\cP\right\|  \left\|P(E_n)\right\| \leq \left\|\int_{\C^+_\imath}f\diff\cP\right\|
\end{align*}
for every $n \in \N$. Since $\|f\|^{\sss(P)}_\infty=\sup_{n \in \N}\|f\Chi_{E_n}\|^{\sss(P)}_\infty$ (see Remark \ref{rempropP}$(2)$), the proof is complete.
\end{proof}

Another important consequence of Theorem \ref{TEO1unb} is the following.

\begin{corollary}\label{corollario1}
Let $\cP=(P,\mc{L})$ be an iqPVM over $\C^+_\imath$ in $\sH$ and let $f$ be an arbitrary function in $M(\C^+_\imath,\bH)$. Then $\Ker\big(\int_{\C^+_\imath}f\diff\cP\big)=P(f^{-1}(0))(\sH)$ and hence the operator $\int_{\C_\imath^+} f \diff\cP:D_f \to \Ran\big(\int_{\C_\imath^+} f \diff\cP\big)$ is bijective if and only if $P(f^{-1}(0))=0$. In this case 
\[
\left(\int_{\C_\imath^+} f \diff\cP\right)^{-1}= \int_{\C_\imath^+} \frac{1}{f} \, \diff\cP,
\]
where $\frac{1}{f}:\C^+_\imath \to \bH$ is the function defined by setting $\frac{1}{f}(q):=(f(q))^{-1}$ if $f(q) \neq 0$ and $\frac{1}{f}(q):= 0$ otherwise.
\end{corollary}

\begin{proof}
The proof of the equality $\Ker\big(\int_{\C^+_\imath}f\diff\cP\big)=P(f^{-1}(0))(\sH)$ is similar to the one of point $(\mr{a})(\mr{vii})$ of Theorem \ref{TEO1}. Let $E:=f^{-1}(0)$. By point $(\mr{a})$ of Theorem \ref{TEO1unb}, we know that $\int_{\C^+_\imath}f \diff\cP \, P(E)=\int_{\C^+_\imath}f \diff\cP\int_{\C^+_\imath}\Chi_E \diff\cP=\int_{\C^+_\imath}f\Chi_E \diff\cP=0$ and hence $P(E)(\sH) \subset \Ker\big(\int_{\C^+_\imath}f \diff\cP\big)$. Let $u \in \Ker\big(\int_{\C^+_\imath}f \diff\cP\big)$. By \eqref{normunb}, we have that $\int_{\C^+_\imath}|f|^2\diff\mu_u^{\sss(P)}=0$. Since the measure $\mu_u^{\sss(P)}$ is positive, it follows that $\mu_u^{\sss(P)}(\C^+_\imath \setminus E)=0$ and hence $P(\C^+_\imath \setminus E)u=0$. In this way, we infer that $u=P(E)u \in P(E)(\sH)$. This proves that $\Ker(\int_{\C^+_\imath}f\diff\cP)=P(E)(\sH)$.

Assume $\int_{\C^+_\imath}f\diff\cP$ is bijective or, equivalently, $P(E)=0$. Since $f \frac{1}{f}$ is $P$-almost everywhere equal to the function on $\C^+_\imath$ constantly equal to $1$, point $(\mr{a})$ of Theorem \ref{TEO1unb} and point $(\mr{a})(\mr{iv})$ of Theorem \ref{TEO1} ensure that $D\big(\int_{\C^+_\imath}\frac{1}{f}\diff\cP \int_{\C^+_\imath}f \diff\cP\big)=D\big(\int_{\C^+_\imath} f \diff\cP \big)$ and $\int_{\C^+_\imath}\frac{1}{f} \diff\cP\int_{\C^+_\imath} f \diff\cP \subset \1$. In this way, $\int_{\C^+_\imath} f \diff\cP$ is invertible, $\Ran\big(\int_{\C^+_\imath} f \diff\cP\big) \subset D_{\frac{1}{f}}$ and $\big(\int_{\C^+_\imath} f \diff\cP\big)^{-1} \subset \int_{\C^+_\imath} \frac{1}{f} \diff\cP$. By applying the same argument to $\frac{1}{f}$, we obtain that $\int_{\C^+_\imath} \frac{1}{f} \diff\cP$ is invertible, $\Ran\big(\int_{\C^+_\imath} \frac{1}{f} \diff\cP\big) \subset D_f$ and $\big(\int_{\C^+_\imath} \frac{1}{f} \diff\cP\big)^{-1} \subset \int_{\C^+_\imath} f \diff\cP$. In particular, we have that
\[
\textstyle
D_{\frac{1}{f}}=\big(\int_{\C^+_\imath} \frac{1}{f} \diff\cP\big)^{-1}\left(\Ran\big(\int_{\C^+_\imath} \frac{1}{f} \diff\cP\big)\right) \subset \big(\int_{\C^+_\imath} f \diff\cP\big)(D_f)=\Ran\big(\int_{\C^+_\imath} f \diff\cP\big),
\]
and hence $D_{\frac{1}{f}}=\Ran\big(\int_{\C^+_\imath} f \diff\cP\big)$ and $\int_{\C^+_\imath} \frac{1}{f} \diff\cP=\big(\int_{\C^+_\imath} f \diff\cP\big)^{-1}$, as desired.
\end{proof}

We conclude this section by improving points $(\mr{a})$ and $(\mr{b})$ of Theorem \ref{TEO1unb}. 

First, we need some preparations. As for complex Hilbert spaces, we have the following definition. 

\begin{definition}\label{defcore}
Let $T: D(T) \to \sH$ be a closed linear operator.
A linear subspace $\sS$ of $D(T)$ is said to be a \emph{core for $T$} if $\overline{T|_\sS}=T$. \bs
\end{definition}

The following is a quaternionic version of Theorem 4.3$(\mr{iii})$ of \cite{Schmudgen}.

\begin{lemma} \label{lem:core2}
Let $f \in M(\C^+_\imath,\bH)$ and let $\{E_n\}_{n \in \N}$ be a sequence of sets in $\mscr{B}(\C^+_\imath)$ such that $E_n \subset E_{n+1}$ for every $n \in \N$, $\bigcup_{n \in \N}E_n=\C^+_\imath$ and $f\Chi_{E_n} \in M_b(\C^+_\imath,\bH)$. Such a sequence of Borel sets is called \emph{bounding sequence for $f$}. Consider an iqPVM $\cP=(P,\mc{L})$ over $\C^+_\imath$ in $\sH$. Then $\bigcup_{n \in \N}P(E_n)(\sH)$ is a core for $\int_{\C^+_\imath}f\diff\cP$.
\end{lemma}

\begin{proof}
Let $n \in \N$. Thanks to points $(\mr{a})$ and $(\mr{e})$ of Theorem \ref{TEO1unb}, and to the fact that the measurable function $f\Chi_{E_n}$ is bounded, we infer at once that
\begin{align*}
%\textstyle
D\left(P(E_n)\int_{\C^+_\imath}f\diff\cP\right)
&=%\textstyle
D\left(\int_{\C^+_\imath}\Chi_{E_n}\diff\cP\int_{\C^+_\imath}f\diff\cP\right)\\
&=D_f \cap D_{f\Chi_{E_n}}=D_f \cap \sH=D_f
\end{align*}
and
\[
P(E_n)\int_{\C^+_\imath}f\diff\cP \subset \int_{\C^+_\imath}\Chi_{E_n}f\diff\cP = \int_{\C^+_\imath}f\Chi_{E_n}\diff\cP=\left(\int_{\C^+_\imath}f\diff\cP \right) P(E_n).
\]
It follows that $P(E_n)(\sH) \subset D_f$ and $\int_{\C^+_\imath}f\diff\cP(P(E_n)u)=P(E_n)(\int_{\C^+_\imath}f\diff\cP \, u)$ for every $u \in D_f$. Bearing in mind points $(\mr{a})$ and $(\mr{c})$ of Definition \ref{defSPOVM}, it follows also that, if $u \in D_f$, then $P(E_n)u \to u$ and $\int_{\C^+_\imath}f\diff\cP(P(E_n)u)=P(E_n)(\int_{\C^+_\imath}f\diff\cP \, u) \to \int_{\C^+_\imath}f\diff\cP \, u$ as $n \to +\infty$. This proves that $\bigcup_{n \in \N}P(E_n)(\sH)$ is a core for $\int_{\C^+_\imath}f\diff\cP$. 
\end{proof}

\begin{proposition} \label{TEO1unb-improved}
Given an iqPVM $\cP=(P,\mc{L})$ over some $\C^+_\imath$ in $\sH$ and arbitrary functions $f$ and $g$ in $M(\C^+_\imath,\bH)$, it holds:
\beq \label{eq:closure-composition}
\overline{\int_{\C_\imath^+} f \diff\cP \int_{\C_\imath^+} g \diff\cP} = \int_{\C_\imath^+} fg \diff\cP.
\eeq
and
\beq \label{eq:closure-sum}
\overline{\int_{\C_\imath^+} f \diff\cP + \int_{\C_\imath^+} g \diff\cP} = \int_{\C_\imath^+} (f+g) \diff\cP.\eeq
\end{proposition}

\begin{proof}
The proof is similar to the complex one (see points $(\mr{ii})$ and $(\mr{iii})$ of Theorem 4.16 of \cite{Schmudgen}). For every $n \in \N$, denote by $E_n$ the Borel subset of $\C^+_\imath$ formed by all $q$ such that $\max\{|f(q)|,|g(q)|\} \leq n$. Observe that $\{E_n\}_{n \in \N}$ is a bounding sequence for $f$, for $g$ and hence for $fg$. By Lemma \ref{lem:core2}, the subspace $\mc{D}:=\bigcup_{n \in \N}P(E_n)(\sH)$ of $\sH$ is a core for $\int_{\C^+_\imath}fg\diff\cP$. Since each measurable function $g\Chi_{E_n}$ is bounded, points $(\mr{a})$ and $(\mr{e})$ of Theorem \ref{TEO1unb} imply that
\begin{align*}
\int_{\C^+_\imath}fg\diff\cP \, P(E_n)&=\int_{\C^+_\imath}fg\Chi_{E_n}\diff\cP=\int_{\C^+_\imath}f\diff\cP \int_{\C^+_\imath}g\Chi_{E_n}\diff\cP\\
&=\int_{\C^+_\imath}f\diff\cP \int_{\C^+_\imath}g\diff\cP \, P(E_n)
\end{align*}
for every $n \in \N$. In other words, $\mc{D}$ is contained in $D\big(\int_{\C^+_\imath}f\diff\cP \int_{\C^+_\imath}g\diff\cP\big)$ and the operator $\int_{\C^+_\imath}f\diff\cP \int_{\C^+_\imath}g\diff\cP$ coincides with  $\int_{\C^+_\imath}fg\diff\cP$ on $\mc{D}$. Since $\int_{\C^+_\imath}f\diff\cP \int_{\C^+_\imath}g\diff\cP \subset \int_{\C^+_\imath}fg\diff\cP$ and $\mc{D}$ is a core for $\int_{\C^+_\imath}fg\diff\cP$, we obtain  \eqref{eq:closure-composition}.

The proof of \eqref{eq:closure-sum} is analogous.
\end{proof}

%%%%%%%

\section{Spectral theorem for unbounded normal operators and iqPVMs} \label{sec:spectral-unb}

%%%

\subsection{Preliminary lemmata}
%
%\begin{proposition}\label{lemmaAP}
%Let $A\in \gB(\sH)$ a normal opertor in the quaternionic Hilbert space $\sH$.  $x\in \sH$ is an eigenvector of $A$ with eigenvalue $q \bH$ is and only if $x$ is eigenvector of $T^*$ with eigenvalue $\overline{q}$. 
%\end{proposition}
%
%\begin{proof}  $||Ax||^2 = \langle Ax|Ax \rangle = \langle A^*Ax |x \rangle = \langle AA^*x|x \rangle =
%||A^*x||^2$. This proves the thesis for $q=0$. Next suppose $q\neq 0$. 
%$$||Ax-xq||^2 = \langle Ax|Ax\rangle - \overline{q} \langle x|Ax\rangle - \langle Ax| x \rangle q + ||x||^2 ||q||^2\:.$$
%As a consequence, with some trivial computations, using in particular $A^*A=AA^*$ and the proerties of the quaternionic scalar product,
%$$|q|^2||Ax-xq||^2 = q ||Ax-xq||^2 \overline{q} = ||A^*x - x\overline{q}||^2\:.$$
%Since $|q|^2 \neq 0$ this proves the remaing case.
%\end{proof}

Let $\sH$ be a quaternionic Hilbert space. We need some preliminary technical lemmata whose proofs are quite similar to those in complex Hilbert spaces. They mostly rely upon results of \cite{GhMoPe}, which extend known results of complex Hilbert space theory.

\begin{lemma}\label{lemA}
Let $A : D(A) \to \sH$ be a positive self-adjoint operator with dense domain and let $\lambda$ be a positive real number. Then $\Ran(\1\lambda+A)=\sH$ and the operator $\1\lambda+A$ is invertible. Furthermore, its inverse $(\1\lambda+A)^{-1}:\sH \to D(\1\lambda+A)=D(A)$ has the following properties:
\begin{itemize}
 \item[$(\mr{a})$] $(\1\lambda+A)^{-1} \in \gB(\sH)$ with $\|(\1\lambda+A)^{-1}\| \leq \lambda^{-1}$.
 \item[$(\mr{b})$] $(\1\lambda+A)^{-1} \geq 0$ and hence $(\1\lambda+A)^{-1}$ is self-adjoint as well.
 \end{itemize}
\end{lemma} 

\begin{proof}
%First, we observe that $\1\lambda : \sH \ni x \mapsto x\lambda \in \sH$ is a well-defined (right $\bH$-linear) operator because $\lambda \in \R$.

%$A+I\lambda$ is self-adjoint on its standard domain $D(I\lambda+A) = D(A)$ as can be proved by direct inspection.

Let $u \in \Ker(\1\lambda+A)$. Since $A$ is positive, we infer that $0=\b u|u\lambda+Au \k=\|u\|^2 \lambda+\b u|Au \k \geq \|u\|^2 \lambda$ and hence $u=0$. This proves that $\1\lambda+A$ is injective. The operator $\1\lambda+A$ is self-adjoint. In this way, thanks to Proposition 2.14 of \cite{GhMoPe}, we obtain that $\Ran(\1\lambda+A)$ is dense in $\sH$, because $\Ran(\1\lambda+A)^\perp=\Ker(\1\lambda+A)=\{0\}$.

Let $L:=(\1\lambda+A)^{-1}:\Ran(\1\lambda+A) \to D(A)$ be  the inverse of $\1\lambda+A$, let $y \in \Ran(\1\lambda+A)$ and let $x:=Ly$. We just know that $\b Ly | y\k=\b x | x\lambda+Ax\k \geq \|x\|^2\lambda=\|Ly\|^2\lambda$. In this way, by the Cauchy-Schwartz inequality, it follows that $\|Ly\| \|y\| \geq \lambda \|Ly\|^2$ and hence $\|y\| \geq \|Ly\|\lambda$. This fact proves that $\|L\| \leq \lambda^{-1}$. Furthermore, it holds $\b y|Ly \k=\b (\1\lambda+A)x|x \k=\b x|(\1\lambda+A)x \k \geq 0$.

In order to complete the proof, it is now sufficient to show that $\Ran(\1\lambda+A)=\sH$ or, equivalently, that $\Ran(\1\lambda+A)$ is closed in $\sH$. Indeed, if this would be true, then point $(\mr{a})$ would be evident and point $(\mr{b})$ would immediately follow from point $(\mr{b})$ of Proposition 2.17 of \cite{GhMoPe}. Consider a sequence $\{y_n\}_{n \in \N}$ in $\Ran(\1\lambda+A)$ converging to some $y \in \sH$. Let $x_n:=Ly_n \in D(A)$ for every $n \in \N$. Since $L$ is bounded, the sequence $\{x_n\}_{n \in \N}$ is a Cauchy sequence in $\sH$, which converges to some $x \in \sH$. Since $A$ is closed (because it is self-adjoint), the operator $\1\lambda+A$ is also closed and thus $y=(\1\lambda+A)x \in \Ran(\1\lambda+A)$, as desired.
\end{proof}

\begin{lemma}\label{lemA2}
Let $A:D(A) \to \sH$ be a positive self-adjoint operator with dense domain. The following facts hold.
\begin{itemize}
 \item[$(\mr{a})$] $\ssp(A) \subset [0,+\infty)$.
 \item[$(\mr{b})$] Suppose $A \in \gB(\sH)$. Let $\cP=(P,\LL)$ be an iqPVM over $\C^+_\imath$ in $\sH$ such that $A=\int_{[0,+\infty)}\mi{id}\diff\cP$ (see Theorem \ref{spectraltheorem} for the existence of such a $\cP$). Then
\beq \label{sqra}
\sqrt{A}=\int_{[0,+\infty)} \sqrt{r} \diff\cP(r),
\eeq
where $\int_{[0,+\infty)} \sqrt{r} \diff\cP(r)$ denotes the integral of the real-valued continuous function $[0,+\infty) \ni r \mapsto \sqrt{r}$ with respect to $\cP$ (see Remark \ref{remarksupp}(3)).
\end{itemize} 
\end{lemma} 

\begin{proof}
$(\mr{a})$ Thanks to point $(\mr{b})$ of Theorem 4.8 in \cite{GhMoPe}, we know that $\ssp(A) \subset \R$. We must only prove that $(-\infty,0)$ does not intersect $\ssp(A)$. Choose $\lambda \in (0,+\infty)$. By Lemma \ref{lemA}$(\mr{a})$, we have that $\1\lambda+A$ is invertible and $(\1\lambda+A)^{-1} \in \gB(\sH)$. It follows that $\Delta_{-\lambda}(A)=(\1\lambda+A)^2$ is also invertible and $\Delta_{-\lambda}(A)^{-1}=(\1\lambda+A)^{-1}(\1\lambda+A)^{-1} \in \gB(\sH)$. Therefore $-\lambda \not \in \ssp(A)$, as desired.

$(\mr{b})$ Combining point $(\mr{b})$ of Theorem 4.3 of \cite{GhMoPe} with preceding point $(\mr{a})$ and with point $(\mr{b})$ of Theorem \ref{spectraltheorem}, we obtain at once that $\ssp(A)$ is a non-empty compact subset of $\R$, which is equal to $\supp(P)$. It follows that $\int_{[0,+\infty)} \sqrt{r} \diff\cP(r) \in \gB(\sH)$. Moreover, bearing in mind point $(\mr{a})(\mr{ii})$ of Theorem \ref{TEO1}, Theorem \ref{spectraltheorem} and \eqref{intemediatestimatef}, we infer that
\[
\left(\int_{[0,+\infty)} \sqrt{r} \diff\cP(r)\right)^2= \int_{\supp(P)} \mi{id} \diff\cP=A
\]
and
\[
\left\langle u \left| \, \int_{[0,+\infty)} \sqrt{r} \diff\cP(r) \right. \right \rangle = \int_{[0,+\infty)} \sqrt{r} \diff\mu^{\sss(P)}_u(r) \geq 0.
\]

Summing up, $\int_{[0,+\infty)} \sqrt{r} \diff\cP(r)$ is a positive operator in $\gB(\sH)$ whose square is $A$. By the uniqueness of the square root of a positive operator in $\gB(\sH)$ (see Theorem 2.18 of \cite{GhMoPe}), we have that $\int_{[0,+\infty)} \sqrt{r} \diff\cP(r)=\sqrt{A}$, as desired.
\end{proof}

\begin{lemma}\label{lemB}
Let $A:D(A) \to \sH$ be a closed linear operator and let $B \in \gB(\sH)$. Suppose that there exists a dense linear subspace $\sS$ of $\sH$ such that $\sS \subset D(AB)$ and the restriction of $AB$ to $\sS$ is bounded. Then $D(AB)=\sH$ and $AB \in \gB(\sH)$.
\end{lemma} 

\begin{proof}
Consider any fixed $x \in \sH$ and a sequence $\sS \supset \{x_n\}_n \to x$. Since $B \in \gB(\sH)$ and $AB|_\sS$ is bounded, the sequence $D(A) \supset \{Bx_n\}_n \to Bx$ and the sequence $\{ABx_n\}_n$ converges to some $y \in \sH$. Bearing in mind that $A$ is closed, we infer that $Bx \in D(A)$ and $ABx=y$. In particular, we have that $D(AB)=\sH$ and $AB$ is closed. Thanks to the closed graph theorem (see Proposition 2.11$(\mr{f})$ of \cite{GhMoPe}), it follows that $AB \in \gB(\sH)$.
\end{proof}

A well-known result for complex Hilbert spaces generalizes to quaternionic Hilbert spaces.

\begin{lemma} \label{lem:core}
Let $T:D(T) \to \sH$ be a closed operator with dense domain. Then $D(T^*T)$ is a core for $T$.
\end{lemma}

\begin{proof}
By point $(\mr{d})$ of Theorem 2.15 in \cite{GhMoPe}, we know that $D(T^*T)$ is dense in $\sH$ and the operator $T^*T$ is self-adjoint. Since $T^*T$ is also positive, Lemma \ref{lemA} ensures that $\Ran(\1+T^*T)=\sH$. As a general result, one immediately sees that the closedness of $T$ is equivalent to the fact that the quaternionic (right) vector space $D(T)$, equipped with the quaternionic scalar product $\langle x|y \rangle_T := \langle x| y\rangle + \langle Tx |Ty \rangle$ is complete; that is, it is a quaternionic Hilbert space, we denote by $\sH_T$. Consequently, the thesis is true if $D(T^*T)$ is dense in $\sH_T$. Let us show it. If $x \in D(T)$ is orthogonal to $D(T^*T)$ in $\sH_T$, then we have that $0= \langle x|y\rangle_T = \langle x|(\1+T^*T)y \rangle$ for $y \in D(T^*T)$. Since $\Ran(\1+T^*T)=\sH$, we conclude that $x=0$ and thus $D(T^*T)$ is dense in $\sH_T$, as desired.
\end{proof}

The next result is a slightly improved quaternionic version of Lemma 5.7 of \cite{Schmudgen} for complex operators.

\begin{lemma}\label{lemC}
Let $T:D(T) \to \sH$ be a closed operator with dense domain.  The following hold.
\begin{itemize}
 \item[$(\mr{a})$] The operator $C_T:=(\1+T^*T)^{-1}$ is a well-defined element of $\gB(\sH)$, it is positive and $\Ran(C_T)=D(T^*T)$.
 \item[$(\mr{b})$] The operator  $Z_T:=T\sqrt{C_T}$ has the following properties:
 \begin{itemize}
  \item[$(\mr{i})$] $D(Z_T)=\sH$ and $Z_T \in \gB(\sH)$ with $\|Z_T\| \leq 1$.
  \item[$(\mr{ii})$] $\Ran(\sqrt{C_T}\,)=D\big((\sqrt{\1- Z^*_TZ_T}\,)^{-1}\big)$ is a core for $T$ and it holds:
  \beq \label{inversT}
 C_T =\1-Z^*_TZ_T
 \quad \mbox{ and } \quad
 T= \overline{Z_T \big(\sqrt{\1-Z^*_TZ_T}\,\big)^{-1}}.
  \eeq
  \item[$(\mr{iii})$] $Z_T^* = Z_{T^*}$.
  \item[$(\mr{iv})$] $Z_T$ is normal if $T$ is normal.
 \end{itemize} 
\end{itemize} 
\end{lemma}

\begin{proof}
$(\mr{a})$ As already noticed at the beginning of the proof of Lemma \ref{lem:core}, $T^*T$ is a positive self-adjoint operator with dense domain. By Lemma \ref{lemA}, the operator  $C_T:=(\1+T^*T)^{-1}:\sH \to D(T^*T)$ is well-defined, bijective, bounded and positive.

$(\mr{b})$ Since $\Ran(C_T)=D(T^*T) \subset D(T)$, we have that $D(TC_T)=\sH$ and hence $\Ran(\sqrt{C_T}\,) \subset D(Z_T)$. Let $y \in \Ran(\sqrt{C_T}\,)$ and let $x \in \sH$ such that $y=\sqrt{C_T} \, x$. Bearing in mind that $C_Tx \in D(T)$ and $TC_Tx \in D(T^*)$, it holds:
\begin{align*} 
\left\|Z_Ty \right\|^2&=\left\|TC_Tx\right\|^2=\langle C_Tx \, | \, T^*TC_Tx \rangle \leq \langle C_Tx \,|\, (\1+T^*T)C_Tx \rangle\\
&=\langle C_Tx|x \rangle = \left\|\sqrt{C_T}x\right\|^2=\|y\|^2.
\end{align*}
Observe that the operator $\sqrt{C_T}$ is injective, because $\Ker(\sqrt{C_T}\,) \subset \Ker(C_T)=\{0\}$. Since the operator $\sqrt{C_T} \in \gB(\sH)$ is self-adjoint, Proposition 2.14 of \cite{GhMoPe} ensures that  $\Ran(\sqrt{C_T}\,)^\perp=\Ker(\sqrt{C_T}\,)=\{0\}$. In other words, $\Ran(\sqrt{C_T}\,)$ is dense in $\sH$. Now we can apply Lemma \ref{lemB} for $A:=T$, $B:=\sqrt{C_T}$ and $\sS:=\Ran(\sqrt{C_T}\,)$, obtaining that $Z_T:= T\sqrt{C_T}$ is well-defined on the whole $\sH$, it belongs to $\gB(\sH)$ and $\|Z_T\| \leq 1$. This proves $(\mr{i})$.

Let us show $(\mr{ii})$. By $(\mr{i})$, we know that $D(T) \supset \Ran(\sqrt{C_T}\,)$. On the other hand, it is evident that $\Ran(\sqrt{C_T}\,) \supset \Ran(C_T)=D(T^*T)$. It follows that $D(T) \supset \Ran(\sqrt{C_T}\,) \supset D(T^*T)$. Since $D(T^*T)$ is a core for $T$ by Lemma \ref{lemC}, $\Ran(\sqrt{C_T})$ is a core for $T$ as well. Next observe that, from the relations $\Ran(C_T)=D(T^*T)$ and $Z^*_T \supset (\sqrt{C_T}\,)^*T^*=\sqrt{C_T}\,T^*$, we have:
\begin{align*}
Z_T^*Z_T \sqrt{C_T} &\supset \sqrt{C_T} \, T^*T \sqrt{C_T}\sqrt{C_T}\\
&=\sqrt{C_T}\,\big(\1+T^*T\big)C_T - \sqrt{C_T}\,C_T = (\1-C_T)\sqrt{C_T}.
\end{align*} 
Hence $Z_T^*Z_T\sqrt{C_T}=(\1-C_T)\sqrt{C_T}$. Since $\Ran(\sqrt{C_T}\,)$ is dense in $\sH$ and the opera\-tors $Z_T^*Z_T $ and $\1-C_T$ are continuous, we have also that $Z_T^*Z_T=\1-C_T$. It follows that $T|_{\Ran(\sqrt{C_T}\,)} = Z_T(\sqrt{C_T}\,)^{-1}=Z_T(\sqrt{\1-Z_T^*Z_T}\,)^{-1}$. On the other hand, $\Ran(\sqrt{C_T}\,)$ is a core for $T$ and hence $T=\overline{Z_T(\sqrt{\1-Z_T^*Z_T}\,)^{-1}}$.

Now we prove $(\mr{iii})$. Thanks to point $(\mr{b})$ of Theorem 2.15 of \cite{GhMoPe} and to the fact that $T$ is closed, we see that $D(T^*)$ is dense in $\sH$ and $(T^*)^*=\overline{T}=T$. Define $C_{T^*}:=(\1+(T^*)^*T^*)^{-1}=(\1+TT^*)^{-1}$. Since $T^*$ is closed, by applying preceding point $(\mr{a})$ to $C_{T^*}$, we obtain that $C_{T^*} \in \gB(\sH)$ and $\Ran(C_{T^*})=D(TT^*) \subset D(T^*)$. Take $x \in D(T^*)$ and set $y:= C_{T^*}x \in D(T^*)$. By definition of $C_{T^*}$, we have that $x=(\1+TT^*)y$ and hence $TT^*y=x-y \in D(T^*)$. In particular, it holds $T^*T(T^*y)=T^*(TT^*y)=T^*(x-y)=T^*x-T^*y$ or, equivalently, $T^*x=(\1+T^*T)T^*y$. This implies that  $C_TT^*x=T^*y=T^*C_{T^*}x$. We have just proved that
\[
C_TT^*x=T^*C_{T^*}x \;\; \text{ for every $x \in D(T^*)$}.
\]
By induction, the latter equality immediately implies that, for every $x \in D(T^*)$ and for every $n \in \N$, $(C_{T^*})^nx \in D(T^*)$ and $(C_T)^nT^*x=T^*(C_{T^*})^nx$. In this way, given any $x \in D(T^*)$ and any real polynomial $p(X)=\sum_{k=0}^dX^ka_k \in \R[X]$, we have that
\beq \label{eq:pct}
p(C_{T^*})x \in D(T^*)
\quad \mbox{and} \quad
p(C_T)T^*x = T^*p(C_{T^*})x,
\eeq
where $p_n(S)$ is the operator in $\gB(\sH)$ defined by  $p_n(S)y:=\sum_{k=0}^dS^kya_k$ if $S=C_T$ or $S=C_{T^*}$. Since $C_T$ and $C_{T^*}$ are positive (and hence self-adjoint) elements of $\gB(H)$, points $(\mr{b})$ of Theorems 4.3 and 4.8 of \cite{GhMoPe} and point $(\mr{a})$ of Lemma \ref{lemA2} ensure that $\ssp(C_T)$ and $\ssp(C_{T^*})$ are (nonempty) compact subsets $[0,+\infty)$. Denote by $K$ the compact subset $\ssp(C_T) \cup \ssp(C_{T^*})$ of $[0,+\infty)$. By the Stone-Weierstrass theorem, we can choose a sequence $\{p_n\}_{n \in \N}$ of polynomials in $\R[X]$ in such a way that the corresponding sequence $\{K \ni r \mapsto p_n(r)\}_{n \in \N}$ of functions converges uniformly to $K \ni r \mapsto \sqrt{r}$. Fix $\imath \in \bS$ and denote by $\cP$ and $\cP'$ two iqPVMs over $\bC^+_\imath$ in $\sH$ associated with $C_T$ and $C_{T^*}$ by means of Spectral Theorem \ref{spectraltheorem}, respectively. By point $(\mr{a})$ of Theorem \ref{TEO1} and point $(\mr{b})$ of Lemma \ref{lemA2}, we have that
\[
\left\{p_n(C_T)=\int_Kp_n(r) \diff\cP(r) \right\}_{n \in \N}
\, \lra \,
\int_K\sqrt{r}\diff\cP(r)=\sqrt{C_T} \quad \text{in $\sH$}  
\]
and
\[
\left\{p_n(C_{T^*})=\int_Kp_n(r) \diff\cP'(r) \right\}_{n \in \N}
\, \lra \,
\int_K\sqrt{r}\diff\cP'(r)=\sqrt{C_{T^*}} \quad \text{in $\sH$}.
\]

Since $T^*$ is closed, taking the limit as $n \to +\infty$ in equality \eqref{eq:pct} with $p=p_n$, we obtain that $\sqrt{C_{T^*}}\,x \in D(T^*)$ and $\sqrt{C_T}\,T^*x = T^* \sqrt{C_{T^*}}\,x$ for every $x \in D(T^*)$. Bearing in mind that $Z_{T}^*=(T\sqrt{C_T}\,)^* \supset \sqrt{C_T}\,T^*$, we infer that $Z_T^*x=\sqrt{C_T}\,T^*x=T^*\sqrt{C_{T^*}}\,x=Z_{T^*}x$ for every $x \in D(T^*)$. On the other hand, $D(T^*)$ is dense in $\sH$ and the operators $Z_T$ and $Z_{T^*}$ belong to $\gB(\sH)$. It follows that $Z_T^*= Z_{T^*}$.
 
It remains to prove $(\mr{iv})$. Suppose $T$ is normal. Bearing in mind the first equality of \eqref{inversT}, preceding point $(\mr{iii})$ and the fact that $(T^*)^*=T$, we obtain at once that
\begin{align*}
\1-Z_T^*Z_T &=C_T=(\1+T^*T)^{-1}=(\1+TT^*)^{-1}\\
&=\big(\1+(T^*)^*T^*\big)^{-1}=\1-(Z_{T^*})^*Z_{T^*}=\1-Z_TZ_T^*.
\end{align*} 
It follows that $Z_T$ is normal, as desired.
\end{proof}

%%%

\subsection{The spectral theorem for unbounded normal operators and some consequences}

We are in a position to state and prove the spectral theorem for unbounded normal operators (see also Proposition \ref{prop:unb-sp-teo} below).

\begin{theorem}\label{spectraltheoremU}
Let $\sH$ be a quaternionic Hilbert space and let $\imath \in \bS$. Then, given a closed normal operator $T:D(T) \to \sH$ with dense domain, there exists an iqPVM $\mc{Q}=(Q,\LL)$ over $\C^+_\imath$ in $\sH$ such that
\beq \label{decspetU}
T=\int_{\C^+_\imath} \mi{id}\diff\mc{Q},
\eeq
where $\mi{id}:\C^+_\imath \hookrightarrow \bH$ indicates the inclusion map.

The following additional facts hold.
\begin{itemize}
\item[$(\mr{a})$] $T$ determines uniquely $Q$. More precisely, if $\mc{Q}'=(Q',\LL')$ is any iqPVM over $\C^+_\imath$ in $\sH$ such that $T=\int_{\bC^+_\imath} \mi{id} \diff\mc{Q}'$, then $Q'=Q$. Furthermore, if $L'_q$ denotes $\LL'(q)$ for every $q \in \bH$, then we have
%$T$ determines uniquely $Q$. More precisely, if $\mc{Q}'=(Q',\LL')$ is any iqPVM over $\C^+_\imath$ in $\sH$ such that $T=\int_{\bC^+_\imath} \mi{id} \diff\mc{Q}'$, then $Q'=Q$. Furthermore, if $L'_q$ denotes $\LL'(q)$ for every $q \in \bH$, then we have:
 \begin{itemize}
  \item[$(\mr{i})$] $L'_\imath T=TL'_\imath$ and $L'_\imath T^*=T^*L'_\imath$,
  \item[$(\mr{ii})$] $L'_\jmath T=T^*L'_\jmath$ if $\jmath \in \cS$ with $\imath\jmath=-\jmath\,\imath$,
  \item[$(\mr{iii})$] $-L'_\imath(T-T^*) \geq 0$.
 \end{itemize}
 \item[$(\mr{b})$] $\supp(Q)=\ssp(T) \cap \C^+_\imath$, and $\supp(Q)$ (or equivalently $\ssp(T)$) is compact if and only if $T \in \gB(\sH)$.
% \item[$(\mr{c})$] It holds 
%\beq \label{aggdecspetU}
%T^*=\int_{\bC^+_\imath} \overline{\mi{id}}\diff\mc{Q}.
%\eeq
 \item[$(\mr{c})$] $P(\R)(\sH)=\Ker(\, \overline{T-T^*}\,)$ and hence $P(\bC^+_\imath \setminus \R)(\sH)=\Ker(\, \overline{T-T^*}\,)^\perp$.
 \item[$(\mr{d})$] Concerning the spherical spectrum $\ssp(T)$ of $T$, we have:
 \begin{itemize}
  \item[$(\mr{i})$] $q \in \sigma_{\mi{pS}}(T)$ if and only if $Q(\bS_q \cap \bC^+_\imath) \neq 0$. In particular, every isolated point of $\ssp(T) \cap \bC^+_\imath$ belongs to $\sigma_{\mi{pS}}(T)$.
  \item[$(\mr{ii})$] $\sigma_{\mi{rS}}(T)=\emptyset$.
  \item[$(\mr{iii})$] $q \in \sigma_{\mi{cS}}(T)$ if and only if $Q(\bS_q \cap \bC^+_\imath)=0$ and $Q(A) \neq 0$ for every open subset $A$ of $\bC^+_\imath$ containing $\bS_q \cap \bC^+_\imath$. Furthermore, if $q \in \sigma_{\mi{cS}}(T)$, then, for every positive real number $\epsilon$, there exists a vector $u_\epsilon \in \sH$ such that $\|u_\epsilon\|=1$ and $\|Tu_\epsilon - u_\epsilon q\|<\epsilon$.
 \end{itemize}
\end{itemize}
\end{theorem}

\begin{proof}
We divide the proof into five steps.

\textit{Step I.} Let us prove the existence of an iqPVM $\mc{Q}=(Q,\LL)$ satisfying \eqref{decspetU}. 

Let $C_T$ and $Z_T$ be the operators in $\gB(\sH)$ defined as in the statement of Lemma \ref{lemC}. By point $(\mr{b})(\mr{iv})$ of the same lemma, we know that $Z_T$ is normal. In this way, we can apply Theorem \ref{spectraltheorem} to $Z_T$, obtaining a bounded iqPVM $\cP=(P,\LL)$ over $\C^+_\imath$ in $\sH$ such that $\supp(P)=\ssp(Z_T) \cap \bC_\imath^+$ and
\beq \label{eq:Z_T}
Z_T=\int_{\ssp(Z_T) \cap \bC_\imath^+} \mi{id}\diff\cP.
\eeq
By Theorem 4.3$(\mr{a})$ of \cite{GhMoPe} and above Lemma \ref{lemC}$(\mr{b})(\mr{i})$, we have that $\ssp(Z_T) \cap \bC_\imath^+ \subset B$, where $B:=\{z \in \bC_\imath \,|\, |z| \leq 1\}$.

Let us show that $P(\ssp(Z_T) \cap \bC_\imath^+ \cap \partial B)=0$. We will use this equality below. Define the measurable function $\psi:\C^+_\imath \to \C^+_\imath$ and the qPVM $R:\mscr{B}(\C^+_\imath) \to \gB(\sH)$ by setting $\psi(z):=|z|^2$ and $R(E):=P(\psi^{-1}(E))$. Points $(\mr{a})$ and $(\mr{c})$ of Theorem~\ref{TEO1} implies that the iqPVM $\mc{R}:=(R,\LL)$ satisfies the Spectral Theorem \ref{spectraltheorem} for $Z_T^*Z_T$; that is, $Z_T^*Z_T =\int_{\C^+_\imath} \mi{id}\diff\mc{R}$. Since $Z_T^*Z_T$ is self-adjoint, Theorem 4.8$(\mr{b})$ of \cite{GhMoPe} implies that $\ssp(T) \subset \R$. On the other hand, by combining the fact that $Z_T^*Z_T$ is positive and $\|Z^*_T Z_T\| \leq \|Z_T\|^2 \leq 1$ with Lemma \ref{lemA2}$(\mr{a})$, we obtain at once that $\supp(R)=\ssp(Z_T^*Z_T) \subset [0,1]$. Suppose that $P(\ssp(Z_T) \cap \bC_\imath^+ \cap \partial B) \neq 0$. We would have that $R(\{1\})=P(\partial B)=P(\supp(P) \cap \partial B)=P(\ssp(Z_T) \cap \bC_\imath^+ \cap \partial B) \neq 0$. By above Theorem \ref{spectraltheorem}$(\mr{d})(\mr{i})$ and Proposition 4.5 of \cite{GhMoPe}, it would imply that $1$ is an eigenvalue of $Z_T^*Z_T$ or, equivalently, that $\Ker(\1-Z_T^*Z_T) \neq \{0\}$. This is impossible, because the first equality of \eqref{inversT} asserts that $\1-Z_T^*Z_T$ is equal to $C_T$, which is bijective. Therefore $P(\ssp(Z_T) \cap \bC_\imath^+ \cap \partial B)=0$.

To go on, looking at the second equality of (\ref{inversT}), one immediately realizes that a promising idea for constructing the desired iqPVM $\mc{Q}$ is to compare $T$ with the closed normal operator $\int_{\C_\imath^+} f\diff\cP$, where $f:\C^+_\imath \to \bH$ is the measurable function in $M(\C^+_\imath,\bH)$ given by
\beq \label{f}
\text{$f(z):=z\big(\sqrt{1-\overline{z}z}\,\big)^{-1}$ if $|z|<1$ and $f(z):=0$ otherwise.}
\eeq

Let us follow such an idea. By \eqref{eq:Z_T} and Theorem \ref{TEO1}$(\mr{a})$, we have that
\[
\1-Z_T^*Z_T=\int_{\C_\imath^+} (1-\overline{z}z)\diff\cP(z).
\]
Define the measurable function $\psi':\C^+_\imath \to \C^+_\imath$, the qPVM $R':\mscr{B}(\C^+_\imath) \to \gB(\sH)$ and the iqPVM $\mc{R}'$ over $\C^+_\imath$ in $\sH$ by setting $\psi'(z):=1-\overline{z}z$, $R'(E):=P((\psi')^{-1}(E))$ and $\mc{R}':=(R',\LL)$, respectively. Observe that $\supp(R') \subset (\psi')^{-1}([0,1])=B$. By Theorem \ref{TEO1}$(\mr{c})$, we have that  $\1-Z_T^*Z_T=\int_B\mi{id}\diff\mc{R}'$. Lemma \ref{lemC} ensures that the operator $\1-Z_T^*Z_T$ is positive. By combining points $(\mr{a})$ and $(\mr{b})$ of Theorem \ref{spectraltheorem} with point $(\mr{a})$ of Lemma \ref{lemA2}, we infer that $\supp(R') \subset [0,1]=B \cap [0,+\infty)$. In this way, if $g:\C^+_\imath \to \bH$ is the measurable function defined by $g(z):=\sqrt{1-\overline{z}z}$ if $|z|<1$ and $g(z):=0$ otherwise, then Lemma \ref{lemA2}$(\mr{b})$ and Theorem \ref{TEO1}$(\mr{c})$ imply that
\[
\sqrt{\1-Z^*_TZ_T}=\int_{[0,1]}\sqrt{r}\diff\mc{R}'(r)=\int_B g\diff\cP=\int_{\C^+_\imath} g\diff\cP.
\]
Since $\supp(P) \cap g^{-1}(0)=\ssp(A) \cap \C^+_\imath \cap \partial B$, we have that $P(g^{-1}(0))=0$. This allows to apply Corollary \ref{corollario1}, obtaining:
\[
\big(\sqrt{\1-Z^*_TZ_T}\,\big)^{-1}=   \int_{\C_\imath^+}\frac{1}{g}\diff\cP.
\]
On the other hand, thanks to Lemma \ref{lemC}$(\mr{a})(\mr{ii})$, we know that 
\[
T|_{\Ran(\sqrt{C_T}\,)} = Z_T(\sqrt{C_T}\,)^{-1}=Z_T(\sqrt{\1-Z_T^*Z_T}\,)^{-1}
\]
and hence, exploiting Theorem \ref{TEO1unb}$(\mr{a})$, we obtain:
\begin{align*}
T|_{\Ran(\sqrt{C_T}\,)}&=Z_T(\sqrt{\1-Z_T^*Z_T}\,)^{-1}= \int_{\C_\imath^+} \mi{id}\diff\cP \int_{\C_\imath^+} \frac{1}{g} \diff\cP\\
&=\int_{\C_\imath^+} \mi{id} \cdot \frac{1}{g} \diff\cP=\int_{\C_\imath^+} f\diff\cP,
\end{align*}
where $f$ is the measurable function defined in \eqref{f}. Here we used the equality $D_{\frac{1}{g}}=D_{\mi{id} \cdot \frac{1}{g}}$ (and hence $D_{\frac{1}{g}} \supset D_{\mi{id} \cdot \frac{1}{g}}$), which is immediate to verify.

Now, by Lemma \ref{lemC}$(\mr{b})(\mr{ii})$, we know that $\Ran(\sqrt{C_T}\,)$ is a core for $T$ and, by the last part of point $(\mr{c})$ of Theorem \ref{TEO1unb}, we have that the operator $\int_{\C_\imath^+} f\diff\cP$ is closed. It follows that
\[
T=\overline{T|_{\Ran(\sqrt{C_T}\,)}} =\overline{\int_{\C^+_\imath} f\diff\cP}=\int_{\C^+_\imath}  f\diff\cP,
\]

Define the measurable function $F:\C^+_\imath \to \C^+_\imath$ and the qPVM $Q$ over $\C^+_\imath$ in $\sH$ by setting $F(z):=f(z)$ for every $z \in \C^+_\imath$ and $Q(E):=P(F^{-1}(E))$ for every $E \in \mscr{B}(\C^+_\imath)$. Denote by $\mc{Q}$ the iqPVM $(Q,\LL)$ over $\C^+_\imath$ in $\sH$. Since $f=\mi{id}\circ F$, using again Theorem \ref{TEO1}$(\mr{c})$, we obtain:
\[
T=\int_{\C^+_\imath}  f\diff\cP=\int_{\C^+_\imath} (\mi{id} \circ F)\diff\cP=\int_{\C^+_\imath}\mi{id}\diff\mc{Q},
\]
which proves \eqref{decspetU}.

\textit{Step II.} Let us show point $(\mr{a})$.

$(\mr{a})$ Suppose that there exists another iqPVM $\mc{Q}'=(Q',\LL')$ over $\C^+_\imath$ in $\sH$ such that $T=\int_{\C^+_\imath} \mi{id}\diff\mc{Q}'$. Let $\phi:\C^+_\imath \to \C^+_\imath$ be the measurable function defined by setting $\phi(z):=z\sqrt{(1+\overline{z}z)^{-1}}$. 
Notice that $\phi$ is a homeomorphism onto its image $\phi(\C^+_\imath)=\{z \in \C^+_\imath \, | \, |z|<1\}$. Define two bounded iqPVMs $\mc{U}=(U,\LL)$ and $\mc{U}'=(U',\LL')$ over $\C^+_\imath$ in $\sH$, and two normal operators $S$ and $S'$ in $\gB(\sH)$ as follows:
\[
U(E):=Q(\phi^{-1}(E)), \quad U'(E):=Q'(\phi^{-1}(E))
\]
and
\[
S:=\int_{\C^+_\imath}\mi{id}\diff\mc{U}, \quad S':=\int_{\C^+_\imath}\mi{id}\diff\mc{U}'.
\]
By Theorem \ref{TEO1}$(\mr{c})$, we have that
\beq \label{eq:S}
S=\int_{\C^+_\imath}\phi(z)\diff\mc{Q}(z),
\quad
S'=\int_{\C^+_\imath}\phi(z)\diff\mc{Q}'(z).
\eeq
Combining these equalities with Theorem \ref{TEO1unb}, Corollary \ref{corollario1} and Lemmata \ref{lemA}, \ref{lemA2} and \ref{lemC}, we easily obtain:
\[
S=Z_T=S'.
\]
Thanks to Theorem \ref{spectraltheorem}$(\mr{a})$, we infer at once that $U=U'$ and hence $Q=Q'$ as well. Indeed, for every $F \in \mscr{B}(\C^+_\imath)$, the set $\phi(F)$ is a Borel subset of $\C^+_\imath$ and hence it holds:
\[
Q(F)=U(\phi(F))=U'(\phi(F))=Q'(F).
\]
This proves the uniqueness of the qPVM $Q$.

Let us show $(\mr{i})$, $(\mr{ii})$ and $(\mr{iii})$. By Theorem \ref{TEO1unb}$(\mr{c})$ and \eqref{decspetU}, we know that $D(T)=D_{\mi{id}}=D_{\overline{\mi{id}}}=D(T^*)$ and $T^*=\int_{\C^+_\imath}\overline{\mi{id}}\diff\mc{Q}$. Combining these facts with \eqref{ccq}, we obtain at once $(\mr{i})$ and $(\mr{ii})$, because $\imath \cdot \mi{id}=\mi{id} \cdot \imath$, $\imath \cdot \overline{\mi{id}}=\overline{\mi{id}} \cdot \imath$ and $\jmath \cdot \mi{id}=\overline{\mi{id}} \cdot \jmath$ if $\jmath \in \cS$ with $\imath\jmath=-\jmath\,\imath$. Moreover, thanks to \eqref{normunb}, we have that
\beq \label{agg}
\text{$\|T^*u\|=\|Tu\|\;$ if $u \in D(T)=D(T^*)$.}
\eeq
In particular, we have:
\beq \label{agg-bis}
\Ker(T)=\Ker(T^*).
\eeq
We will use \eqref{agg-bis} in Step IV below.

Now, by Proposition \ref{defintunb}$(\mr{d})$ and Theorem \ref{TEO1unb}$(\mr{b})$, we have that
\[
-L_\imath(T-T^*) \subset \int_{\C^+_\imath} -\imath (z-\overline{z})\diff\mc{Q}(z)=\int_{\C^+_\imath} 2\beta\diff\mc{Q},
\]
where the integrand $2\beta$ indicates the function $\C_\imath^+ \ni z=\alpha+\imath\beta \mapsto 2\beta \in \bH$. This proves $(\mr{iii})$. Indeed, thanks to \eqref{normunb2}, given any $u \in D(T)=D(T^*)$, it holds:
\[
\langle u| -L_\imath (T-T^*)u \rangle = \int_{\C_\imath^+} 2\,\beta \diff\mu^{\sss(Q)}_u \geq 0.
\]

\textit{Step III.} Let us prove $(\mr{b})$.

We begin by recalling that, by definition, $q \in \srho(T)=\bH \setminus \ssp(T)$ if and only if the operator $\Delta_q(T):D(T^2) \to \sH$, defined by $\Delta_q(T):=T^2-T(2\mr{Re}(q))+\1|q|^2$, has the following three properties $\Ker(\Delta_q(T))=\{0\}$, $\Ran(\Delta_q(T))$ is dense in $\sH$ and the operator  $\Delta_q(T)^{-1}:\Ran(\Delta_q(T)) \to D(T^2)$ is bounded (see Subsection \ref{subsec:ssp}).

Let us prove that
\beq \label{eq:srho}
\srho(T)=\big\{q \in \bH \, \big| \, \Delta_q(T):D(T^2) \to \sH \text{ is bijective and }\Delta_q(T)^{-1} \in \gB(\sH)\big\}.
\eeq
Evidently, if $\Delta_q(T)$ is bijective and $\Delta_q(T)^{-1} \in \gB(\sH)$, then $q \in \srho(T)$.

Consider now a point $q$ in $\srho(T)$. In order to obtain \eqref{eq:srho}, it suffices to show that $\Ran(\Delta_q(T))$ is closed in $\sH$. First, we show that
\beq \label{eq:int-Delta-q}
\Delta_q(T)=\int_{\C^+_\imath}\big(z^2-z(2\mr{Re}(q))+|q|^2\big)\diff\mc{Q}(z)
\eeq
Since the positive Borel measure $\mu_u^{\sss(Q)}:\mscr{B}(\C^+_\imath) \to \R^+$ is finite (see \eqref{measuremuPx} for its definition), we know that $L^4(\C^+_\imath,\bH;\mu_u^{\sss(Q)}) \subset L^2(\C^+_\imath,\bH;\mu_u^{\sss(Q)})$ for every $u \in \sH$. It follows immediately that $D_{\mi{id}} \supset D_{\mi{id}^2}$. Thanks to \eqref{decspetU} and Theorem \ref{TEO1unb}$(\mr{a})$, we infer that
\[
T^2=\int_{\C^+_\imath}\mi{id}^2\diff\mc{Q}.
\]
By using \eqref{decspetU} and \eqref{ccq}, we have also that $T(2\mr{Re}(q))=\int_{\C^+_\imath}\mi{id}(2\mr{Re}(q))\diff\mc{Q}$ and $\1|q|^2=\int_{\C^+_\imath}c_{\,|q|^2}\diff\mc{Q}$, where $c_{\,|q|^2}:\C^+_\imath \to \bH$ is the function constantly equal to~$|q|^2$. Observe that, if $z \in \C^+_\imath$ and $|z| \geq 2(2\mr{Re}(q))$, then $|z| \leq |z-2\mr{Re}(q)|+|2\mr{Re}(q)| \leq 2|z-2\mr{Re}(q)|$ and hence $|z^2| \leq 2|z^2-z(2\mr{Re}(z))|$. It follows that $D_{\mi{id}^2-\mi{id}(2\mr{Re}(z))}=D_{\mi{id}^2}=D_{\mi{id}^2} \cap D_{-\mi{id}(2\mr{Re}(z))}$. Thanks to this fact and to the equality $D(\1|q|^2)=\sH$, we can apply points $(\mr{b})$ and $(\mr{e})$ of Theorem \ref{TEO1unb} obtaining \eqref{eq:int-Delta-q}:
\[ \Delta_q(T)=\big(T^2-T(2\mr{Re}(q))\big)+\1|q|^2=\int_{\C^+_\imath}\big(\mi{id}^2-\mi{id}(2\mr{Re}(q))+c_{\,|q|^2}\big)\diff\mc{Q}.
\]

By point $(\mr{c})$ of the same theorem, we have also that:
\beq \label{eq:closedness-Delta-q}
\text{$\Delta_q(T):D(T^2) \to \sH$ is a closed operator}.
\eeq

We are in a position to prove that $\Ran(\Delta_q(T))$ is closed in $\sH$. We adapt to the present situation the argument used in Remark 4.2(1) of \cite{GhMoPe}. Let $\Ran(\Delta_q(T)) \supset \{y_n\}_{n \in \N} \to y$ in $\sH$ and, for every $n \in \N$, let $x_n:=\Delta_q(T)^{-1}y_n \in D(\Delta_q(T))$. Since we are assuming that $q$ belongs to $\srho(T)$, $\Delta_q(T)^{-1}$ is bounded and hence $\{x_n\}_{n \in \N}$ is a Cauchy sequence in $\sH$; indeed, $\|x_n-x_m\| \leq \|\Delta_q(T)^{-1}\| \, \|y_n-y_m\| \to 0$ as $n,m \to +\infty$. We infer that $\{x_n\}_{n \in \N} \to x$ for some $x \in \sH$. Since the operator $\Delta_q(T)$ is closed, it follows that $x \in D(T^2)$ and $y \in \Ran(\Delta_q(T))$. In particular, $\Ran(\Delta_q(T))$ is closed in $\sH$, and equality \eqref{eq:srho} is valid. Such an equality permits to prove $(\mr{b})$ as follows.

Let $q \in \C^+_\imath \setminus \supp(Q)$. Define the function $h_q$ in $M_b(\C^+_\imath,\bH)$ and the operator $H_q$ in $\gB(\sH)$ by setting $h_q(z):=\big(z^2-z(2\mr{Re}(q))+|q|^2\big)^{-1}$ if $z \in \supp(Q)$, $h_q(z):=0$ if $z \in \C^+_\imath \setminus \supp(Q)$, and $H_q:=\int_{\C^+_\imath}h_q\diff\mc{Q}$. By combining points $(\mr{a})$ and $(\mr{e})$ of Theorem \ref{TEO1unb} with \eqref{threeintunb} and \eqref{eq:int-Delta-q}, it follows at once that $H_q\Delta_q(T) \subset \1$ and $\Delta_q(T)H_q=\1$. In other words, $\Delta_q(T)$ is bijective and $\Delta_q(T)^{-1} \in \gB(\sH)$, and hence \eqref{eq:srho} implies that $q \in \srho(T)$. 

Suppose now that $q \in \supp(Q)$. We must show that $q \not\in \srho(T)$. For every $n \in \N$, define
\[
B_n:=\big\{z \in \C^+_\imath \, \big| \, |z-q| \leq (n+1)^{-1}\big\}
\]
and
\[
m_n:=\sup_{z \in B_n}\big|z^2-z(2\mr{Re}(q))+|q|^2\big|^2>0.
\] 
Evidently, the sequence $\{m_n\}_{n \in \N}$ converges to $0$. Since $q \in \supp(Q)$, for every $n \in \N$, $Q(B_n) \neq 0$ and hence we can choose $u_n \in Q(B_n)(\sH) \setminus \{0\}$. Recall that $Q(B_n)=\int_{\C^+_\imath}\Chi_{B_n}\diff\mc{Q}$ and observe that $Q(B_n)u_n=u_n$. In this way, using \eqref{eq:int-Delta-q} and points $(\mr{a})$ and $(\mr{e})$ of Theorem \ref{TEO1unb}, we obtain that $D(\Delta_q(T)Q(B_n))=\sH$, $u_n \in D(T^2)$ and
\begin{align*}
\Delta_q(T)u_n&=\int_{\C^+_\imath}\big(z^2-z(2\mr{Re}(q))+|q|^2\big)\diff\mc{Q}(z)\int_{\C^+_\imath}\Chi_{B_n}\diff\mc{Q}\,u_n\\
&=\int_{\C^+_\imath}\big(z^2-z(2\mr{Re}(q))+|q|^2\big)\Chi_{B_n}(z)\diff\mc{Q}(z)\,u_n\\
&=\int_{B_n}\big(z^2-z(2\mr{Re}(q))+|q|^2\big)(z)\diff\mc{Q}(z)\,u_n.
\end{align*}
Thanks to \eqref{normunb}, we obtain also that
\begin{align*}
\|\Delta_q(T)u_n\|^2&=\int_{\C^+_\imath}\big|z^2-z(2\mr{Re}(q))+|q|^2\big|^2\diff\mu_{u_n}^{\sss(Q)}(z)\\
&\leq m_n \int_{\C^+_\imath}\diff\mu_{u_n}^{\sss(Q)}=m_n\|u_n\|^2.
\end{align*}
Thus, if $\Delta_q(T)$ would be bijective, then each vector $v_n:=\Delta_q(T)^{-1}(u_n)$ would be different from $0$, it would hold that
\[
\|\Delta_q(T)^{-1}(v_n)\|^2 \geq \frac{1}{m_n}\|v_n\|^2
\]
and hence $\Delta_q(T)^{-1} \not\in \gB(\sH)$. This shows that $q \not\in \srho(T)$. 

We have just proved that $\supp(Q)=\ssp(T) \cap \C^+_\imath$. Let us complete the proof of~$(\mr{b})$. If $T \in \gB(\sH)$, the compactness of $\supp(T)$ is a direct consequence of above Theorem \ref{spectraltheorem}$(\mr{b})$ and of Theorem 4.3$(\mr{b})$ of \cite{GhMoPe}. Conversely, if 
$\supp(Q)$ is compact, then $\mi{id}$ is bounded on $\supp(Q)$ and hence, by \eqref{decspetU}, \eqref{threeintunb} and Proposition \ref{defintunb}$(\mr{f})$, $T=\int_{\C^+_\imath}\mi{id}\diff\mc{Q}=\int_{\supp(Q)}\mi{id}\diff\mc{Q}$ belongs to $\gB(\sH)$.

\textit{Step IV.} Let us show point $(\mr{c})$. By  Proposition \ref{TEO1unb-improved}, we know that $\overline{T-T^*}=\int_{\C^+_\imath}2\beta\imath\diff\cP$. In this way, Corollary \ref{corollario1} ensures that $\Ker(\,\overline{T-T^*}\,)=P(\R)(\sH)$, which is equivalent to say that $\Ker(\,\overline{T-T^*}\,)^\perp=P(\bC^+_\imath \setminus \R)(\sH)$.

\textit{Step V}. Let us complete the proof by proving $(\mr{d})$.

We begin with $(\mr{i})$. Let $q \in \C^+_\imath$ such that $Q(\{q\}) \neq 0$. Write $q$ as follows $q=\alpha+\imath\beta$, where $\alpha,\beta \in \R$ and $\beta \geq 0$. Choose $w \in Q(\{q\})(\sH) \setminus \{0\}$. Using \eqref{eq:int-Delta-q}, points $(\mr{a})$ and $(\mr{e})$ of Theorem \ref{TEO1unb}, and \eqref{ccq}, we obtain that $D(TQ(\{q\}))=\sH$, $w \in D(T)$ and
\begin{align*}
Tw&=TQ(\{q\})w=\int_{\C^+_\imath}\mi{id}\diff\mc{Q}\int_{\C^+_\imath}\Chi_{\{q\}}\diff\mc{Q} \, w=\int_{\C^+_\imath}\mi{id} \cdot \Chi_{\{q\}}\diff\mc{Q} \, w\\
&=\int_{\C^+_\imath}q\Chi_{\{q\}}\diff\mc{Q} \, w=L_q\int_{\C^+_\imath}\Chi_{\{q\}}\diff\mc{Q} \, w=L_qQ(\{q\})w\\
&=L_qw=w\alpha+L_\imath w \beta.
\end{align*}
Consider the $\C_\imath$-complex subspaces $\sH_\pm^{L_\imath\imath}=\{u \in \sH \, | \, L_\imath u=\pm u\imath\}$ of $\sH$ and the decomposition $\sH=\sH_+^{L_\imath\imath} \oplus \sH_-^{L_\imath\imath}$ (see Subsection \ref{subsec:complex-subspaces}). Define $w_\pm:=\frac{1}{2}(w \mp L_\imath w \imath) \in \sH^{L_\imath\imath}_\pm$ in such a way that $w=w_++w_-$. Since $TL_\imath=L_\imath T$, we have that $L_\imath (D(T)) \subset D(T)$, $w_\pm \in D(T)$ and $Tw_\pm \in \sH^{L_\imath\imath}_\pm$. %$T (D(T) \cap \sH^{L_\imath\imath}_\pm) \subset \sH^{L_\imath\imath}_\pm$.
Thus it holds:
\begin{align*}
Tw_++Tw_-&=Tw=w\alpha+L_\imath w\beta =w_+\alpha+w_-\alpha+L_\imath w_+\beta+L_\imath w_-\beta\\
&=w_+\alpha+w_-\alpha+w_+\imath\beta-w_-\imath\beta=w_+q+w_-\overline{q}
\end{align*}
or, equivalently, $Tw_+=w_+q$ and $Tw_-=w_-\overline{q}$. Since $w \neq 0$, at least one of $w_\pm$ is different from $0$ and thus $q$ is an eigenvalue of $T$. Thanks to Proposition 4.5 of \cite{GhMoPe}, this is equivalent to say that $q \in \sigma_{pS}(T)$. We have just proven that, if $Q(\bS_q \cap \bC^+_\imath) \neq 0$, then $q \in \sigma_{pS}(T)$.

Let us prove the converse implication. Suppose that 
$q \in \C^+_\imath \cap \sigma_{pS}(T)$. If $q=0$, then $\Ker(T) \neq \{0\}$. Moreover, $Q(\{0\})=Q(\mi{id}^{-1}(0))$ must be different from~$0$; otherwise, by Corollary \ref{corollario1}, $T$ would be bijective and, in particular, $\Ker(T)=\{0\}$. Suppose that $q \neq 0$. Using Proposition 4.5 of \cite{GhMoPe} again, we know that there exists $y \in D(T) \setminus \{0\}$ such that $Ty=yq$. Point $(\mr{a})$ of Theorem \ref{TEO1unb} and \eqref{decspetU} imply that
\[
Q(\{q\})yq=Q(\{q\})Ty=\int_{\C_\imath^+} \Chi_{\{q\}} \mi{id}\diff\mc{Q}\,y=L_qy
\]
and thus
\[
Q(\{q\})y=L_qyq^{-1}=y\alpha q^{-1}+L_\imath y \beta q^{-1}.
\]
Decompose $y$ as follows $y=y_++y_-$ with $y_\pm \in \sH^{L_\imath\imath}_\pm$. Since $L_\imath$ commutes with $Q(\{q\})$, we can repeat the preceding argument, obtaining that $Q(\{q\})y_+=y_+$ and $Q(\{q\})y_-=y_-\overline{q}q^{-1}$. At least one of $y_\pm$ is different from $0$ and hence $Q(\{q\}) \neq 0$.

Let us show $(\mr{ii})$. Recall that, by definition, a quaternion $q$ belongs to $\sigma_{rS}(T)$ if $\Ker(\Delta_q(T))=\{0\}$ and $\Ran(\Delta_q(T))$ is not dense in $\sH$. Such a quaternion cannot exist. Indeed, by \eqref{eq:closedness-Delta-q} and Theorem \ref{TEO1unb}$(\mr{d})$, we have that the operator $\Delta_q(T)$ is closed and normal. In this way, we can apply \eqref{agg-bis} and Proposition 2.14 of \cite{GhMoPe} to $\Delta_q(T)$, obtaining that
\[
\{0\}=\Ker(\Delta_q(T))=\Ker(\Delta_q(T)^*)=\Ran(\Delta_q(T))^{\perp}.
\]
Thus, if $\Ker(\Delta_q(T))=\{0\}$, then $\Ran(\Delta_q(T))$ is dense in $\sH$.

It remains to show $(\mr{iii})$. Since $\ssp(T)$ is equal to the disjoint union of $\sigma_{pS}(T)$ and $\sigma_{cS}(T)$, above points $(\mr{b})$ and $(\mr{d})(\mr{i})$ imply immediately the first part of $(\mr{iii})$. In order to prove the second part of $(\mr{iii})$, one can repeat exactly the proof of the second part of point $(\mr{d})(\mr{iii})$ of Theorem \ref{spectraltheorem}.
\end{proof}

The proof of Theorem \ref{spectraltheoremU} singles out some results which are interesting in their own right. Let us collect those results in the next proposition.

\begin{proposition} \label{propLu}
Let $\sH$ be a quaternionic Hilbert space and let $T:D(T) \to \sH$ be a closed normal operator with dense domain. The following facts hold.
\begin{itemize}
 \item[$(\mr{a})$] There is an unitary operator $U:\sH \to \sH$ such that $T^*=U^*TU$.
 \item[$(\mr{b})$] $D(T)=D(T^*)$, $\|T^*u\|=\|Tu\|$ if $u \in D(T)$ and $\Ker(T)=\Ker(T^*)=\Ran(T)^\perp$. Moreover, we have that $D(T^2)=D(T^*T)=D(TT^*)$.
 \item[$(\mr{c})$] $\Delta_q(T)=\int_{\C^+_\imath}\big(z^2-z(2\mr{Re}(q))+|q|^2\big)\diff\mc{Q}(z)$ for every $q \in \bH$.
 \item[$(\mr{d})$] $\srho(T)=\big\{q \in \bH \, \big| \, \Delta_q(T):D(T^2) \to \sH \text{ is bijective, }\Delta_q(T)^{-1} \in \gB(\sH)\big\}$.
\end{itemize}
\end{proposition}
\begin{proof}
Point $(\mr{a})$ follows from Theorem \ref{spectraltheoremU}$(\mr{a})(\mr{ii})$, the first part of $(\mr{b})$ from \eqref{agg} and Proposition 2.14 of \cite{GhMoPe}, $(\mr{c})$ from \eqref{eq:int-Delta-q} and $(\mr{d})$ from \eqref{eq:srho}: see Steps II and III of the preceding proof. The second part of $(\mr{b})$ is an immediate consequence of \eqref{decspetU} and of points $(\mr{a})$ and $(\mr{d})$ of Theorem \ref{TEO1unb}.
\end{proof}

\begin{remark}
Preceding point $(\mr{a})$ is false in the complex setting. Preceding point $(\mr{d})$ was proven in Remark 4.2 of \cite{GhMoPe} in the special case of normal operators in $\gB(\sH)$.~\bs
\end{remark}

Theorem \ref{spectraltheoremU} allows to prove the existence, and the uniqueness, of the square root of a positive self-adjoint operator and of the absolute value of a normal operator in the unbounded case.

\begin{lemma} \label{lem:J}
The following facts hold.
\begin{itemize}
 \item[$(\mr{a})$] Given a positive self-adjoint operator $A:D(A) \to \sH$ with dense domain, there exists a unique positive self-adjoint operator $\sqrt{A}:D(\sqrt{A}\,) \to \sH$ with dense domain, we call \emph{square root of $A$}, such that
\[
\big(\sqrt{A}\,\big)^2=A.
\]
Moreover, if $\mc{K}=(K,\mc{N})$ is an iqPVM over some $\C^+_\imath$ in $\sH$ such that $\supp(K)=\ssp(A) \subset [0,+\infty)$ and $A=\int_{[0,+\infty)}r\diff\mc{K}(r)$ (which exists by Theorem \ref{spectraltheoremU} and Lemma \ref{lemA2}(a); see also Remark \ref{remarksupp2}(1)), then
\beq \label{eq:sqrt-A}
\sqrt{A}=\int_{[0,+\infty)} \sqrt{r} \diff\mc{K}(r).
\eeq
 \item[$(\mr{b})$] Let $T:D(T) \to \sH$ be a closed normal operator with dense domain. Then the operator $T^*T:D(T^*T) \to \sH$ is a positive self-adjoint operator with dense domain. We denote by $|T|:D(|T|) \to \sH$ the square root of $T^*T$.
 
If $\mc{Q}=(Q,\LL)$ is an iqPVM over $\C^+_\imath$ in $\sH$ such that $T=\int_{\C^+_\imath}\mi{id}\diff\mc{Q}$, then we have also that
\[
T^*T=TT^*=\int_{\C^+_\imath}|\mi{id}|^2\diff\mc{Q}=\int_{[0,+\infty)}r^2\diff\mc{Q}(r)
\]
and
\[
|T|=\int_{\C^+_\imath}|\mi{id}|\diff\mc{Q}=\int_{[0,+\infty)}r\diff\mc{Q}(r).
\]
In particular, if $T$ is a positive self-adjoint operator with dense domain, then $|T|=T$. 
% \item[$(\mr{c})$] If $A:D(A) \to \sH$ is a positive self-adjoint operator with dense domain, then $|A|=A$.
\end{itemize} 
\end{lemma} 

\begin{proof}
$(\mr{a})$ If we define $\sqrt{A}:=\int_{[0,+\infty)} \sqrt{r} \diff\mc{K}(r)$, then points $(\mr{a})$ and $(\mr{c})$ of Theorem \ref{TEO1unb} and equality \eqref{normunb2} imply that $\sqrt{A}$ is positive, self-adjoint and $(\sqrt{A}\,)^2=A$.

Let $B:D(B) \to \sH$ be another positive self-adjoint operator with dense domain, whose square is $A$. Since $B$ is closed and normal, thanks to Theorem \ref{spectraltheoremU}, Lemma \ref{lemA2}(a) and Remark \ref{remarksupp2}$(1)$, there exists an iqPVM $\cP'=(P',\mc{N}')$ over $\C^+_\imath$ in $\sH$ such that $\supp(P')=\ssp(B) \subset [0,+\infty)$ and $B=\int_{[0,+\infty)}\mi{id}\diff\cP'$. Denote by $\mr{sgn}:\R \to \{-1,0,1\}$ the sign function, sending $t<0$ into $-1$, $t=0$ into $0$ and $t>0$ into $1$. Define the homeomorphism $\psi:\C^+_\imath \to \C^+_\imath$ and the iqPVM $\mc{K}'=(K',\mc{N}')$ over $\C^+_\imath$ in $\sH$ by setting
\[
\psi(\alpha+\imath\beta):=\mr{sgn}(\alpha)\sqrt{|\alpha|}+\imath\beta,
\]
and
\[
K'(F):=P'(\psi(F)) \;\; \text{for every $F \in \mscr{B}(\C^+_\imath)$.}
\]
Observe that $\supp(K')=\psi^{-1}(\supp(P')) \subset [0,+\infty)$ and, for every $r \in [0,+\infty)$, $(\mi{id} \circ \psi)(r)=\sqrt{r}$ and $(\mi{id}^2 \circ \psi)(r)=\mi{id}(r)$. By combining these facts with points $(\mr{a})$ and $(\mr{f})$ of Theorem \ref{TEO1unb}, we obtain:
\beq \label{eq:C}
B=\int_{\C^+_\imath}\mi{id} \circ \psi \diff\mc{K}'=\int_{[0,+\infty)}\sqrt{r}\diff\mc{K}'(r)
\eeq
and
\begin{align*}
A&=B^2=\int_{\C^+_\imath}\mi{id}^2\diff\cP'=\int_{\C^+_\imath}\mi{id}^2 \circ \psi \diff\mc{K}'\nonumber\\
%\label{eq:C^2}
&=\int_{[0,+\infty)}\mi{id}\diff\mc{K}'=\int_{\C^+_\imath}\mi{id}\diff\mc{K}'.
\end{align*}
Since $\int_{\C^+_\imath}\mi{id}\diff\mc{K}=A=\int_{\C^+_\imath}\mi{id}\diff\mc{K}'$, Theorem \ref{spectraltheoremU}$(\mr{a})$ ensures that $K'=K$. Since the function $[0,+\infty) \ni r \mapsto \sqrt{r} \in \bH$ is real-valued, we have that $\int_{[0,+\infty)}\sqrt{r}\diff\mc{K}(r)=\int_{[0,+\infty)}\sqrt{r}\diff\mc{K}'(r)$ (see Remark \ref{remarksupp2}$(2)$). Thanks to \eqref{eq:sqrt-A} and \eqref{eq:C}, we infer that $B=\sqrt{A}$, as desired.

$(\mr{b})$ Theorem 2.15$(\mr{d})$ of \cite{GhMoPe} ensures that $T^*T$ is a self-adjoint operator with dense domain. Evidently, $T^*T$ is also positive and hence point $(\mr{b})$ follows immediately from Theorem \ref{TEO1unb}$(\mr{d})$ and from above point $(\mr{a})$.
\end{proof} 

Thanks to the latter result, we can improve point $(\mr{a})$ of Spectral Theorem \ref{spectraltheoremU}.

\begin{proposition} \label{prop:unb-sp-teo}
Let $\sH$ be a quaternionic Hilbert space, let $\imath \in \bS$ and let $T:D(T) \to \gB(\sH)$ be a closed normal operator with dense domain. If $\mc{Q}=(Q,\LL)$ and $\mc{Q}'=(Q,\LL')$ are iqPVMs over $\C^+_\imath$ in $\sH$ satisfying (\ref{decspetU}); that is, 
\[
\int_{\C^+_\imath}\mi{id}\diff\mc{Q}=T=\int_{\C^+_\imath}\mi{id}\diff\mc{Q}',
\]
then $L_q(u)=L_q'(u)$ for every $u \in \Ker(\, \overline{T-T^*} \,)^\perp$ and for every $q \in \C_\imath$, where $L_q:=\LL(q)$ and $L'_q:=\LL'(q)$.
\end{proposition}

\begin{proof}
%Let us proceed as in Step II of the proof of Theorem \ref{spectraltheorem}.

Define $B:=\big|(\,\overline{T-T^*}\,)\frac{1}{2}\big|$. Thanks to Proposition \ref{TEO1unb-improved} and Lemma \ref{lem:J}$(\mr{b})$, we have that $\int_{\C^+_\imath}\beta\diff\mc{Q}=B=\int_{\C^+_\imath}\beta\diff\mc{Q}'$ and also
\[
L_\imath B=\int_{\C^+_\imath} \imath\beta\diff\mc{Q} =(\,\overline{T-T^*}\,)\frac{1}{2}=\int_{\C^+_\imath}\imath\beta\diff\mc{Q}'=L'_\imath B.
\]
It follows that $L_\imath=L'_\imath$ on $\Ran(B)$ and hence, by continuity, on $\overline{\Ran(B)}$ too. On the other hand, by Proposition \ref{propLu}$(\mr{b})$ and Corollary \ref{corollario1}, we infer that $\overline{\Ran(B)}=\overline{\Ran\big(\big|\,\overline{T-T^*}\,\big|\big)}=\Ker\big(\big|\,\overline{T-T^*}\,\big|\big)^\perp$ and $\Ker\big(\big|\,\overline{T-T^*}\,\big|\big)=Q(\{2\beta=0\})(\sH)=Q(\{2\beta\imath=0\})(\sH)=\Ker\big(\,\overline{T-T^*}\,\big)$. The proof is complete.
\end{proof}

%%%

\subsection{iqPVMs associated with unbounded normal operators}

Our next theorem gives a complete characterization of the iqPVMs associated with an unbounded closed quaternionic normal operator via Spectral Theorem \ref{spectraltheoremU}. It is the version of Theorem \ref{propL} for the unbounded case. 

\begin{theorem}\label{thmLu}
Let $\sH$ be a quaternionic Hilbert space, let $\imath \in \bS$, let $T:D(T) \to \sH$ be a closed normal operator with dense domain and let $Q:\mscr{B}(\bC^+_\imath) \to \gB(\sH)$ be the qPVM over $\bC^+_\imath$ in $\sH$ associated with $T$ by means of Spectral Theorem \ref{spectraltheoremU}(a). Given a left scalar multiplication $\bH \ni q \stackrel{\LL}{\longmapsto}\!L_q$ of $\sH$, we have that the pair $\mc{Q}=(Q,\LL)$ is an iqPVM over $\bC^+_\imath$ in $\sH$ verifying (\ref{decspetU}) if and only if the following three conditions hold:
\begin{itemize}
 \item[$(\mr{i})$] $L_\imath T = T L_\imath$,
 \item[$(\mr{ii})$] $L_\jmath T = T^* L_\jmath$ for some $\jmath \in \bS$ with $\imath\jmath=-\jmath \, \imath$,
 \item[$(\mr{iii})$] $-L_\imath (T-T^*)\geq 0$.
\end{itemize}
\end{theorem}

\begin{proof} 
Point $(\mr{a})$ of Theorem \ref{spectraltheoremU} ensures that, if the pair $\mc{Q}=(Q,\LL)$ is an iqPVM verifying \eqref{decspetU}, then properties $(\mr{i})$, $(\mr{ii})$ and $(\mr{iii})$ hold.

Suppose $\LL$ satisfies $(\mr{i})$, $(\mr{ii})$ and $(\mr{iii})$. Thanks to point $(\mr{i})$, Remark 2.16(v) of~\cite{GhMoPe} and equality $(L_\imath)^*= -L_\imath$, it follows immediately that $-L_\imath T^* \subset (TL_\imath)^*=(L_\imath T)^*=-T^*L_\imath$, so that $L_\imath T^* \subset T^*L_\imath$. In order to prove that $T^*L_\imath=L_\imath T^*$, it remains to verify that $T^*L_\imath \subset L_\imath T^*$; that is, $L_\imath(D(T^*))=D(T^*L_\imath) \subset D(L_\imath T^*)=D(T^*)$. Point $(\mr{i})$ implies that $L_\imath(D(T))=D(TL_\imath)=D(L_\imath T)=D(T)$. On the other hand, by Proposition \ref{propLu}$(\mr{b})$, we have that $D(T^*)=D(T)$ and hence $L_\imath(D(T^*))=D(T^*)$. This proves that $T^*L_\imath=L_\imath T^*$. From $(\mr{ii})$, it follows that $T=-L_\jmath L_\jmath T=-L_\jmath T^*L_\jmath$ and $L_\jmath(D(T))=D(T^*L_\jmath)=D(L_\jmath T)=D(T)$. In particular, $D(TL_\jmath)=L_\jmath(D(T))=D(T)=D(L_\jmath T^*)$ and hence $TL_\jmath=-L_\jmath T^*L_\jmath L_\jmath=L_\jmath T^*$. Summing up, properties $(\mr{i})$ and $(\mr{ii})$ are equivalent to the following:
\begin{itemize}
 \item[$(\mr{i}')$] $L_\imath T = T L_\imath$ and $L_\imath T^* = T^* L_\imath$, 
 \item[$(\mr{ii}')$] $L_\jmath T = T^* L_\jmath$ and $L_\jmath T^* = T L_\jmath$ for some $\jmath \in \bS$ with $\imath\jmath=-\jmath \, \imath$.
% \item[$(\mr{iii}')$] $-L_\imath (T-T^*)\geq 0$.
\end{itemize}

Let $C_T \in \gB(\sH)$ be the positive operator and let $Z_T \in \gB(\sH)$ be the normal operator defined as in Lemma \ref{lemC}; that is, $C_T:=(\1+T^*T)^{-1}$ and $Z_T:=T\sqrt{C_T}$. Thanks to $(\mr{i}')$, we have that $L_\imath(\1+T^*T)=(\1+T^*T)L_\imath$ and hence $C_TL_\imath=C_TL_\imath(\1+T^*T)C_T=C_T(\1+T^*T)L_\imath C_T=L_\imath C_T$. Similarly, combining $(\mr{ii}')$ with the equality $T^*T=TT^*$, we obtain that $L_\jmath(\1+T^*T)=(\1+T^*T)L_\jmath$ and hence $C_TL_\jmath=L_\jmath C_T$. By Theorem 2.18 of \cite{GhMoPe}, it follows that $\sqrt{C_T}\,L_\imath=L_\imath\sqrt{C_T}$ and $\sqrt{C_T}\,L_\jmath=L_\jmath\sqrt{C_T}$ as well. In this way, we infer that $L_\imath Z_T=Z_TL_\imath$ and $L_\jmath Z_T=Z_{T^*}L_\jmath$. Thanks to point $(\mr{b})(\mr{iii})$ of Lemma \ref{lemC}, we know that $Z_T^*=Z_{T^*}$, thus $L_\jmath Z_T=Z_T^*L_\jmath$. We have just proved that
\begin{itemize}
 \item[$(\mr{i}'')$] $L_\imath Z_T=Z_TL_\imath$,
 \item[$(\mr{ii}'')$] $L_\jmath Z_T=Z_T^*L_\jmath$.
\end{itemize}

Let us show that
\begin{itemize}
 \item[$(\mr{iii}'')$] $-L_\imath(Z_T-Z_T^*) \geq 0$.
\end{itemize}
Denote by $\sqrt[4]{C_T}$ the square root of $\sqrt{C_T}$. Using again Theorem 2.18 of \cite{GhMoPe}, we see that $\sqrt[4]{C_T}\,L_\imath=L_\imath\sqrt[4]{C_T}$. Let $\mc{Q}'=(Q,\LL')$ be an iqPVM over $\C^+_\imath$ in $\sH$ such that $T=\int_{\C^+_\imath} z\diff\mc{Q}'(z)$. Define the measurable function $f:\C^+_\imath \to \bH$ by setting $f(z):=z(1+\overline{z}z)^{-\frac{1}{4}}$. Since $\sqrt{C_T}=\int_{\C^+_\imath}(1+\overline{z}z)^{-\frac{1}{2}}\diff\mc{Q}'(z)$ (see Step II of the proof of Theorem \ref{spectraltheoremU}), we have that  $\sqrt[4]{C_T}=\int_{\C^+_\imath}(1+\overline{z}z)^{-\frac{1}{4}}\diff\mc{Q}'(z)$. Moreover, thanks to points $(\mr{a})$, $(\mr{c})$ and $(\mr{e})$ of Theorem \ref{TEO1unb}, it holds:
\begin{itemize}
 \item $T\sqrt[4]{C_T}=\int_{\C^+_\imath} f \diff\mc{Q}'$ and $T^*\sqrt[4]{C_T}=\int_{\C^+_\imath} \overline{f} \diff\mc{Q}'$,
 \item $D\big(\sqrt[4]{C_T}(T\sqrt[4]{C_T}\,)\big)=D_f=D\big(\sqrt[4]{C_T}(T^*\sqrt[4]{C_T}\,)\big)$
\end{itemize}
and
\begin{itemize}
 \item $\sqrt[4]{C_T}(T\sqrt[4]{C_T}\,) \subset \int_{\C^+_\imath} (1+\overline{z}z)^{-\frac{1}{4}}f(z) \diff\mc{Q}'(z)=Z_T$,
 \item $\sqrt[4]{C_T}(T^*\sqrt[4]{C_T}\,) \subset \int_{\C^+_\imath} (1+\overline{z}z)^{-\frac{1}{4}}\overline{f}(z) \diff\mc{Q}'(z)=Z_T^*$.
\end{itemize}
In particular, for every $u \in D_f$, we have:
\begin{align*}
-L_\imath(Z_T-Z_T^*)u &=-L_\imath\sqrt[4]{C_T}(T\sqrt[4]{C_T}\,)u+L_\imath\sqrt[4]{C_T}(T^*\sqrt[4]{C_T}\,)u\\
&=-L_\imath\sqrt[4]{C_T}(T-T^*)\sqrt[4]{C_T}\,u=\sqrt[4]{C_T}(-L_\imath(T-T^*))\sqrt[4]{C_T}\,u
\end{align*}
and hence, by $(\mr{iii})$,
\begin{align*}
\big\langle u\big|-L_\imath(Z_T-Z_T^*)u\big\rangle &= \big\langle u\big|\sqrt[4]{C_T}(-L_\imath(T-T^*))\sqrt[4]{C_T}\,u\big\rangle\\ &=\big\langle \sqrt[4]{C_T}\,u\big|-L_\imath(T-T^*)(\sqrt[4]{C_T}\,u)\big\rangle \geq 0.
\end{align*}
By Proposition \ref{defintunb}$(\mr{a})$, $D_f$ is dense in $\sH$. In this way, $(\mr{iii}'')$ follows by continuity.
 
Let $P$ be the qPVM over $\C^+_\imath$ in $\sH$ associated with $Z_T$ by means of Spectral Theo\-rem \ref{spectraltheorem}. Thanks to points $(\mr{i}'')$, $(\mr{ii}'')$ and $(\mr{iii}'')$, Theorem \ref{propL} applies to $Z_T$. It ensures that the pair $\cP=(P,\LL)$ is an iqPVM over $\C^+_\imath$ in $\sH$ such that $Z_T=\int_{\C^+_\imath}\mi{id}\diff\cP$. Let $F:\C^+_\imath \to \C^+_\imath$ be the measurable map given by $F(z):=z(\sqrt{1-|z|^2}\,)^{-1}$ if $|z|<1$ and $F(z):=0$ if $|z| \geq 1$, and let $Q^0:\mscr{B}(\C^+_\imath) \to \gB(\sH)$ be the qPVM over $\C^+_\imath$ in $\sH$ defined by setting $Q^0(E):=P(F^{-1}(E))$. The proof of Theorem \ref{spectraltheoremU} presented above establishes that the pair $\mc{Q}^0:=(Q^0,\LL)$ is an iqPVM over $\C^+_\imath$ in $\sH$ such that $T=\int_{\C^+_\imath}\mi{id}\diff\mc{Q}^0$ (see Step I) and $Q^0=Q$ (see Step II again). It follows that $\mc{Q}=\mc{Q}^0$ satisfies \eqref{decspetU}, as desired.
\end{proof}

%%%%%%%

\section{Applications} \label{sec:applications}

%%%

%\subsection{Left-eigenvalues and left-eigenvectors}
%\subsection{Left- and right-eigenvalues}
\subsection{Spherical and left spectrum}

Let $T:D(T) \to \sH$ be an operator and let $\bH \ni q \stackrel{\LL}{\longmapsto} L_q$ be a left scalar multiplication of $\sH$. Let us introduce the notion of left spectrum of $T$ w.r.t.\ $\LL$ and some related concepts.

\begin{definition} \label{def:left-spectrum}
The \textit{left resolvent set of $T$ w.r.t.\ $\LL$} is  the set $\rho_\LL(T)$ of all quaternions $q$ such that 
\begin{itemize}
 \item $\Ker(T-L_q)=\{0\}$, 
 \item $\Ran(T-L_q)$ is dense in $\sH$,
 \item $(T-L_q)^{-1}:\Ran(T-L_q) \to D(T)$ is bounded.
\end{itemize}
The  \textit{left spectrum of $T$ w.r.t.\ $\LL$} is the set $\sigma_\LL(T):=\bH \setminus \rho_\LL(T)$, which is the disjoint union of the sets $\sigma_{p\LL}(T)$, $\sigma_{r\LL}(T)$ and $\sigma_{c\LL}(T)$ defined as follows:
\begin{itemize}
 \item $\sigma_{p\LL}(T)$ is the set of all $q \in \bH$ such that $\Ker(T-L_q) \neq \{0\}$. We call $\sigma_{p\LL}(T)$ \textit{left point spectrum of $T$ w.r.t.\ $\LL$}. 
 \item $\sigma_{r\LL}(T)$ is the set of all $q \in \bH$ such that $\Ker(T-L_q)=\{0\}$ and $\Ran(T-L_q)$ is not dense in $\sH$. We call $\sigma_{r\LL}(T)$ \textit{left residual spectrum of $T$ w.r.t.~$\LL$.}  
 \item $\sigma_{c\LL}(T)$ is the set of all $q \in \bH$ such that $\Ker(T-L_q)=\{0\}$, $\Ran(T-L_q)$ is dense in $\sH$ %$\overline{\Ran(T-L_q)}=\sH$
 and $(T-L_q)^{-1}:\Ran(T-L_q) \to D(T)$ is not bounded. We call $\sigma_{c\LL}(T)$ \textit{left continuous spectrum of $T$ w.r.t.\ $\LL$}.
\end{itemize} 
The elements of $\sigma_{p\LL}(T)$ are the \textit{left eigenvalues of $T$ w.r.t.\ $\LL$}. 
If $q \in \sigma_{p\LL}(T)$, the subspace $\Ker(T-L_q)$ of $\sH$ is the \textit{left eigenspace of $T$ w.r.t.\ $\LL$ associated with~$q$} and the non-zero vectors in $\Ker(T-L_q)$ are the \textit{left eigenvectors of $T$ w.r.t.\ $\LL$ associated with~$q$}. \bs
\end{definition}

We remind the reader that, by Proposition 4.5 of \cite{GhMoPe}, the spherical point spectrum $\sigma_{\mi{pS}}(T)=\{q \in \bH \,|\, \Ker(\Delta_q(T)) \neq \{0\} \}$ of $T$ coincides with the right point spectrum of $T$; that is, $\sigma_{\mi{pS}}(T)$ is equal to the set of all (right) eigenvalues of $T$:
\[
\sigma_{pS}(T)=\big\{q \in \bH \, \big| \, Tu=uq \text{ for some $u \in D(T) \setminus \{0\}$}\big\}.
\]

As an application of Spectral Theorem \ref{spectraltheoremU}, we establish the exact relation existing between the spherical spectrum of an unbounded normal operator $T$ and the left spectrum of $T$ w.r.t.\ any left scalar multiplication associated with $T$ itself via Theorem \ref{thmLu}. The corresponding result reads as follows.

\begin{theorem} \label{thm:LL-S}
Let $\sH$ be a quaternionic Hilbert space, let $T:D(T) \to \sH$ be a closed normal operator with dense domain, let $\imath \in \bS$ and let $\mc{Q}=(Q,\LL)$ be an iqPVM over $\C^+_\imath$ in $\sH$ such that $T=\int_{\C^+_\imath}\mi{id}\diff\mc{Q}$. Define $L_q:=\LL(q)$ for every $q \in \bH$. Then it holds:
\beq \label{eq:rho-LL}
\rho_\LL(T)=\{q \in \bH \, \big| \, \text{$T-L_q:D(T) \to \sH$ is bijective, $(T-L_q)^{-1}$ is bounded}\,\},
\eeq
\beq \label{eq:bH-resolvent}
(T-L_q)^{-1}=\int_{\C^+_\imath}(z-q)^{-1}\diff\mc{Q}(z) \;\; \text{ if $q \in \rho_\LL(T)$}
\eeq
and
\begin{align} 
\sigma_\LL(T)&=\ssp(T) \cap \C^+_\imath, \label{eq:sigma-LL}\\
\sigma_{p\LL}(T)&=\sigma_{pS}(T) \cap \C^+_\imath, \label{eq:sigma-pLL}\\
\sigma_{r\LL}(T)&=\sigma_{rS}(T)=\emptyset, \label{eq:sigma-rLL}\\
\sigma_{c\LL}(T)&=\sigma_{cS}(T) \cap \C^+_\imath. \label{eq:sigma-cLL}
\end{align}
In particular, if $\psi:\C_\imath \to \C$ is the natural complex linear isomorphism $\alpha+\imath\beta \mapsto \alpha+i\beta$, then $\ssp(T)$ coincides with the circularization of $\psi(\sigma_\LL(T))$; namely, $\ssp(T)=\OO_{\psi(\sigma_\LL(T))}$. Moreover, $\sigma_{pS}(T)=\OO_{\psi(\sigma_{p\LL}(T))}$ and $\sigma_{cS}(T)=\OO_{\psi(\sigma_{c\LL}(T))}$.

The following additional facts hold regarding the left point spectrum and the left continuous spectrum. Define $\sH_q:=Q(\{q\})(\sH)$ for every $q \in \C^+_	\imath$. Then we have:
\begin{itemize}
 \item[$(\mr{a})$] $\sH_q=\Ker(T-L_q)=\Ker(\Delta_q(T))$ if $q \in \C^+_\imath$ and $\Ker(T-L_q)=\{0\}$ if $q \in \bH \setminus \C^+_\imath$. %We can say that $\sH_q$ is equal to the left eigenspace of $T$ associated with $p$ with respect to $\LL$.
 \item[$(\mr{b})$] Consider the decomposition $\sH= \sH^{L_\imath\imath}_+ \oplus \sH^{L_\imath\imath}_-$ and write each vector $u$ in $\sH$ as follows $u=u_++u_-$ with $u_\pm \in \sH^{L_\imath\imath}_\pm$. If $q \in \C^+_\imath$, then
\[
\sH_q=\{u \in D(T) \, | \, Tu_+=u_+q, \; Tu_-=u_-\overline{q}\}.
\]
 \item[$(\mr{c})$] For every $q \in \sigma_{p\LL}(T)$, the (right) eigenspace of $T$ associated with $q$ is contained in the left eigenspace of $T$ w.r.t.\ $\LL$ associated with $q$; that is,
\beq \label{eq:incl} 
\{u \in D(T) \, | \, Tu=uq\} \subset \Ker(T-L_q).
\eeq
In particular, if $u$ is a (right) eigenvalue of $T$ associated with $q \in \C^+_\imath$, then $uq=L_qu$. The inclusion \eqref{eq:incl} is an equality if and only if $q \in \R$.
  \item[$(\mr{d})$] For every $q \in \sigma_{c\LL}(T)$ and for every positive real number $\epsilon$, there exists a vector $u_\epsilon \in \sH$ such that $\|u_\epsilon\|=1$ and $\|Tu_\epsilon-L_qu_\epsilon\|<\epsilon$. 
\end{itemize} 
\end{theorem}

\begin{proof}
We begin by proving \eqref{eq:rho-LL}. Evidently, it suffices to show that, if $q \in \rho_\LL(T)$, then $\Ran(T-L_q)$ is closed in $\sH$. To do this, we adapt to the present situation the argument used in Step III of the proof of Theorem \ref{spectraltheoremU}. By point $(\mr{a})(\mr{vi})$ of Theorem \ref{TEO1} and by points $(\mr{b})$ and $(\mr{c})$ of Theorem \ref{TEO1unb}, the operator $T-L_q$ is closed and %$T-L_q=\int_{\C^+_\imath}(z-q)\diff\mc{Q}(z)$ .
\beq \label{eq:int-LL}
T-L_q=\int_{\C^+_\imath}(z-q)\diff\mc{Q}(z) \; \; \text{ for every $q \in \bH$}.
\eeq
Suppose $q \in \rho_\LL(T)$. Let $\{y_n\}_{n \in \N}$ be a sequence in $\Ran(T-L_q)$ converging to some point $y$ of $\sH$ and let $x_n:=(T-L_q)^{-1}(y_n)$ for every $n \in \N$. Since $(T-L_q)^{-1}$ is bounded, the sequence $\{x_n\}_{n \in \N} \subset D(T)$ is a Cauchy sequence and hence it converges to some $x \in \sH$. The closedness of $T-L_q$ implies that $x \in D(T)$ and $y=(T-L_q)x \in \Ran(T-L_q)$, thus $\Ran(T-L_q)$ is closed in $\sH$ and \eqref{eq:rho-LL} is proved.

For every $q \in \bH$, define the measurable functions $f_q:\C^+_\imath \to \bH$ and $g_q:\C^+_\imath \to \bH$
by setting $f_q(z):=z-q$ and $g_q(z):=z^2-z(q+\overline{q})+|q|^2$. Observe that $(f_q)^{-1}(0)=\{q\} \cap \C^+_\imath$,  $(g_q)^{-1}(0)=\{q,\overline{q}\} \cap \C^+_\imath$ for every $q \in \bH$ and $g_q=f_qf_{\overline{q}}$ if $q \in \C_\imath$.

Combining \eqref{eq:rho-LL}, \eqref{eq:int-LL} and Corollaries \ref{cor:f} and \ref{corollario1}, we infer at once that   
\beq \label{eq:LL-LL}
\rho_\LL(T)=\{q \in \bH\, | \, Q(\{q\} \cap \C^+_\imath)=0, \|f_q\|^{\sss(Q)}_\infty<+\infty\}.
\eeq
Similarly, points $(\mr{c})$ and $(\mr{d})$ of Proposition \ref{propLu}, and Corollaries \ref{cor:f} and \ref{corollario1} imply that   
\beq \label{eq:sp-sp}
\srho(T)=\{q \in \bH\, | \, Q(\{q,\overline{q}\} \cap \C^+_\imath)=0, \|g_q\|^{\sss(Q)}_\infty<+\infty\}.
\eeq

Let us prove \eqref{eq:sigma-LL}. Thanks to \eqref{eq:LL-LL}, it is evident that $\bH \setminus \C^+_\imath \subset \rho_\LL(T)$. In other words, we have that $\sigma_\LL(T) \subset \C^+_\imath$. If $q \in \rho_\LL(T) \cap \C^+_\imath$, then $Q(\{q\})=0$ and $M:=\|f_q\|^{\sss(Q)}_\infty<+\infty$. Observe that $\{q,\overline{q}\} \cap \C^+_\imath=\{q\}$. Since $g_q=f_qf_{\overline{q}}=f_q(f_q+2\mr{Im}(q))$, we infer that $\|g_q\|^{\sss(Q)}_\infty \leq M(M+2|\mr{Im}(q)|)<+\infty$ and hence $q \in \srho(T)$ by \eqref{eq:sp-sp}.

Suppose now that $q \in \srho(T) \cap \C^+_\imath$. Using again \eqref{eq:sp-sp}, we know that $Q(\{q\})=Q(\{q,\overline{q}\} \cap \C^+_\imath)=0$ and $N:=\|g_q\|^{\sss(Q)}_\infty<+\infty$. If $q \not\in \R$, then $f_q=g_q(f_{\overline{q}})^{-1}$ and $\inf_{z \in \C^+_\imath}|f_{\overline{q}}(z)|=|\mr{Im}(q)|>0$. It follows that $\|f_q\|^{\sss(Q)}_\infty \leq N|\mr{Im}(q)|^{-1}<+\infty$. If $q \in \R$, then $g_q=f_q^2$ and hence $\|f_q\|^{\sss(Q)}_\infty=\sqrt{N}<+\infty$. In both cases, \eqref{eq:LL-LL} implies that $q \in \rho_\LL(T)$. We have just proved \eqref{eq:sigma-LL}.

Let us show \eqref{eq:sigma-rLL}. Suppose $\sigma_{r\LL}(T) \neq \emptyset$ and consider $q \in \sigma_{r\LL}(T) \subset \C^+_\imath$. Since $T$ commutes with $L_q$ and $(L_q)^*=L_{\overline{q}}$ (see Theorem \ref{spectraltheoremU}$(\mr{a})(\mr{i})$), we have that the operator $T-L_q$ is normal. Then, thanks to  Proposition \ref{propLu}, we infer that $\sH=\Ker(T-L_q)^\perp=\overline{\Ran(T-L_q)}$, which is a contradiction. This shows that $\sigma_{r\LL}(T)$ is empty. Recall that Theorem 4.8$(\mr{a})(\mr{ii})$ of \cite{GhMoPe} ensures that $\sigma_{rS}(T)$ is empty as well. Equality \eqref{eq:sigma-rLL} is proved. Additionally, if $q \in \C^+_\imath$, combining \eqref{eq:int-LL} with Corollary \ref{corollario1}, we obtain:
$\Ker(T-L_q)=Q((f_q)^{-1}(0))(\sH)=Q(\{q\})(\sH)=\sH_q$ and $\Ker(\Delta_q(T))=Q((g_q)^{-1}(0))(\sH)=Q(\{q\})(\sH)=\sH_q$. On the other hand, if $q \in \bH \setminus \C^+_\imath$, then $\Ker(T-L_q)=Q((f_q)^{-1}(0))(\sH)=Q(\emptyset)(\sH)=\{0\}$. This proves point $(\mr{a})$.

Let us now focus on point $(\mr{b})$ and on \eqref{eq:sigma-pLL}. Fix $q \in \C_\imath$. Let $\alpha,\beta \in \R$ such that $q=\alpha+\imath\beta$. Since $L_\imath T=TL_\imath$ (see Theorem \ref{spectraltheoremU}$(\mr{a})(\mr{i})$), $q\imath=\imath q$ and $\overline{q}\imath=\imath\overline{q}$, if $u \in D(T)$, then $u_\pm \in D(T)$, $Tu_+-u_+q \in \sH^{L_\imath\imath}_+$ and $Tu_--u_-\overline{q} \in \sH^{L_\imath\imath}_-$. Proceeding as in Step V of the proof of Theorem \ref{spectraltheoremU},  we obtain:
\begin{align*}
(T-L_q)u&=Tu_++Tu_--u_+\alpha-L_\imath u_+\beta-u_-\alpha-L_\imath u_-\beta\\
&=Tu_++Tu_--u_+\alpha-u_+\imath\beta-u_-\alpha+u_-\imath\beta\\
&=(Tu_+-u_+q)+(Tu_--u_-\overline{q}).
\end{align*}
It follows that
\beq \label{eq:Tq}
\Ker(T-L_q)=\{u \in D(T) \, | \, Tu_+=u_+q, \; Tu_-=u_-\overline{q}\}.
\eeq
In particular, when $\beta \geq 0$, we infer point $(\mr{b})$.

Choose $\jmath \in \bS$ in such a way that $\imath\jmath=-\jmath\,\imath$ and suppose $q \in \sigma_{p\LL}(T)$. Let $u \in D(T) \setminus \{0\}$ be such that $Tu=L_qu$. By \eqref{eq:Tq}, we have that $Tu_+=u_+q$, $Tu_-=u_-\overline{q}$ and also $Tu_-\jmath=u_-\overline{q}\jmath=u_-\jmath q$. If $u_+ \neq 0$, then $q \in \sigma_{pS}(T)$. Similarly, if $u_- \neq 0$, then $u_-\jmath \neq 0$ and hence $q \in \sigma_{pS}(T)$ again. This proves that  $\sigma_{p\LL}(T) \subset \sigma_{pS}(T) \cap \C^+_\imath$.

Let us prove the converse inclusion. Suppose $q \in \sigma_{pS}(T) \cap \C^+_\imath$. Let $v \in D(T) \setminus \{0\}$ be such that $Tv=vq$. Since $Tv_+-v_+q \in \sH^{L_\imath\imath}_+$ and $Tv_--v_-q \in \sH^{L_\imath\imath}_-$, we have that $Tv_+=v_+q$ and $Tv_-=v_-q$. If $v_+ \neq 0$ and $u:=v_+ \in \sH$, then $Tu_+=Tv_+=v_+q=u_+q$ and $Tu_-=T0=0=u_-\overline{q}$. In this way, \eqref{eq:Tq} implies that $q \in \sigma_{p\LL}(T)$. Suppose that $v_+=0$ and $v_- \neq 0$. Define $u:=v_-\jmath \in \sH^{L_\imath\imath}_+$. Since $u_+=v_-\jmath$ and $u_-=0$, it holds $Tu_+=Tv_-\jmath=v_-q\jmath=v_-\jmath\overline{q}=u_+\overline{q}$ and $Tu_-=0=u_-q$. Using again \eqref{eq:Tq}, we have that $\overline{q} \in \sigma_{p\LL}(T)$. On the other hand, in this case, it holds that $\overline{q}=q$, being $\sigma_{p\LL}(T) \subset \sigma_\LL(T) \subset \C^+_\imath$. Equality \eqref{eq:sigma-pLL} is proved.

Equality \eqref{eq:sigma-cLL} is now an immediate consequence of \eqref{eq:sigma-LL}, \eqref{eq:sigma-pLL} and \eqref{eq:sigma-rLL}, because $\sigma_{c\LL}(T)=\sigma_\LL(T) \setminus \sigma_{p\LL}(T)$ and $\sigma_{cS}(T)=\ssp(T) \setminus \sigma_{pS}(T)$.

Let us prove $(\mr{c})$. Let $q \in \sigma_{p\LL}(T)=\sigma_{pS}(T) \cap \C^+_\imath$ and let $v \in \sH \setminus \{0\}$ such that $Tv=vq$. Proceeding as in the proof of Proposition 4.5 of \cite{GhMoPe}, we obtain that $\Delta_q(T)v=T(Tv-vq)-(Tv-vq)\overline{q}=0$ and hence $v \in \Ker(\Delta_q(T))=\Ker(T-L_q)$. In particular, it holds $L_qv=Tv=vq$. Furthermore, the set $\{u \in D(T) \, | \, Tu=uq\}$ is a (right $\bH$-linear) subspace of $\sH$ only if $q$ belongs to the center of $\bH$; that is, if $q \in \R$. On the other hand, if $q \in \R$, then $L_qu=uq$ for every $u \in \sH$ and hence the equality holds in \eqref{eq:incl}. 

Finally, point $(\mr{d})$ follows immediately from preceding point $(\mr{c})$ and from Theorem \ref{spectraltheoremU}$(\mr{d})(\mr{iii})$.
\end{proof}

\begin{example}
Consider the finite dimensional quaternionic Hilbert space $\sH=\bH \oplus \bH$, the unitary operator on $\sH$ defined by the matrix
\[
T=
\begin{bmatrix}
0 & i \\
j & 0
\end{bmatrix}
\]
and define $\imath:=i$ as in Example \ref{MEx}(2). Recall that $T$ has spherical spectrum 
\[
\ssp(T)=\sigma_{pS}(T)=\left\{q=q_0+q_1i+q_2j+q_3k \in \bH \, \left| \, q_0^2=q_1^2+q_2^2+q_3^2=\frac12\right.\right\}.
\]
In particular, $\ssp(T) \cap \bC_i^+=\{\lambda_1:=\frac1{\sqrt2}+\frac1{\sqrt2}i,\lambda_2:=-\frac1{\sqrt2}+\frac1{\sqrt2}i\}$.
Moreover, we can associate to $T$ an iqPVM $(P,\LL)$ over $\C^+_i$ in $\sH$, with left scalar multiplication $\bH \ni q=q_0+q_1i+q_2j+q_3k \mapsto \LL(q)=L_q$ defined by
\[
L_q=
\begin{bmatrix}
q_0-q_2j & \frac{-q_3+q_1i+q_1j-q_3k}{\sqrt2} \\
\frac{q_3+q_1i+q_1j-q_3k}{\sqrt2} & q_0+q_2i 
\end{bmatrix},
\]
and with the qPVM $P$ given by 
\[
P(E):=\sum_{\mu=1}^2 \epsilon_\mu P_\mu \;\; \text{ for every $E \in \mscr{B}(\bC^+_\imath)$}, 
\]
where $\epsilon_\mu=1$ if $\lambda_\mu\in E$ and $\epsilon_\mu=0$ otherwise, and 
\[
P_1=
\frac12
\begin{bmatrix}
1 & \frac{i-j}{\sqrt2} \\
\frac{j-i}{\sqrt2} & 1
\end{bmatrix},
\quad
P_2=
\frac12
\begin{bmatrix}
1 & \frac{j-i}{\sqrt2} \\
\frac{i-j}{\sqrt2} & 1
\end{bmatrix}.
\]
The spectral theorem asserts that $T=L_{\lambda_1}P_1+L_{\lambda_2}P_2$.

A quaternionic matrix is invertible if and only if the determinant of one of its complex representations does not vanish (see Section 5.9 of \cite{Rodman}). The left eigenvalues of $T$ w.r.t.\ $\LL$ can then be found by computing the determinant of a representation of $T-L_q$ in $M_{4,4}(\C)$.  One finds that 
\begin{align*}
{\det}_{\C}(T-L_q)=&
\left(\left(q_0^2 - \frac12\right)+\left(q_1-\frac1{\sqrt2}\right)^2\right)^2 + \left(\sqrt2 q_1-1\right)^2 \\
&+\left(1+2q_0^2+\left(\sqrt2 q_1-1\right)^2+q_2^2+q_3^2\right) \left(q_2^2+q_3^2\right).
\end{align*}
This implies that
\[
\sigma_{p\LL}(T)=\left\{\lambda_1,\lambda_2\right\}=\sigma_{pS}(T) \cap \C^+_i,
\]
as predicted by Theorem \ref{thm:LL-S}. If $\mu \in \{1,2\}$, then $\sH_{\lambda_\mu}=P(\{\lambda_\mu\})(\sH)$ coincides with the subspace of $\sH$ generated by the columns of $P_\mu$ or, equivalently, by one of the two columns of $P_\mu$. Since $TP_1=L_{\lambda_1}P_1$ and $TP_2=L_{\lambda_2}P_2$, it follows easily that
\[
\sH_q=\Ker(T-L_q)\text{\quad for every }q \in \C^+_i. 
\]

As an illustration of points (a) and (c) of the statement of Theorem~\ref{thm:LL-S}, consider the eigenvalue $\lambda_1=\frac1{\sqrt2}+\frac1{\sqrt2}i\in\sigma_{p\LL}(T)$. An easy computation shows that the right eigenspace $\{u \in \sH \, | \, Tu=u\lambda_1\}$ has real dimension two. It is generated over $\R$ by the vectors
\[u_1=\left(\frac{i+k}2,\frac{1+i+j-k}{2\sqrt2}\right),
\quad u'_1=\left(\frac{1-j}2,\frac{1-i+j+k}{2\sqrt2}\right).\]
On the other hand, the kernel of the matrix
\[\Delta_{\lambda_1}(T)=T^2-\sqrt2 T+\1=
\begin{bmatrix}
1+k & -\sqrt2 i \\
-\sqrt2 j & 1-k
\end{bmatrix}
\]
has real dimension four, with independent generators $\{u_1,u_1',u_1j,u_1'j\}$. Since $u_1,u_1'$ belong to $\sH_{\lambda_1}=\Ker(T-L_{\lambda_1})$, it holds that $\Ker(\Delta_{\lambda_1}(T))\subset\sH_{\lambda_1}$. But it holds also the reverse inclusion: if $Tu=L_{\lambda_1}u$, then $T^2u=TL_{\lambda_1}u=L_{\lambda_1}Tu=\lambda_1^2u$ and then $\Delta_{\lambda_1}(T)u=(\lambda_1^2-\sqrt2\lambda_1+1)u=0$. Therefore $\Ker(\Delta_{\lambda_1}(T))=\sH_{\lambda_1}$.
 \;\text{ \bs}
\end{example}

%%%

\subsection{Slice-type decomposition of unbounded normal operators}

The aim of this subsection is to prove an unbounded version of Theorem \ref{thm:first-5.9}. Indeed, unbounded normal operators admit a slice-type decomposition as described in our next, and concluding, result.

\begin{theorem} \label{thm:J-unb}
Let $\sH$ be a quaternionic Hilbert space and let $T:D(T) \to \sH$ be a closed normal operator with dense domain. Then there exist three operators $A:D(A) \to \sH$, $B:D(B) \to \sH$ and $J \in \gB(\sH)$ with the following properties:
\begin{itemize}
 \item[$(\mr{i})$] $T=A+JB$; in particular, we have that $D(T)=D(A) \cap D(B)$.
 \item[$(\mr{ii})$] $A$ is self-adjoint and $B$ is positive and self-adjoint.
 \item[$(\mr{iii})$] $J$ is anti self-adjoint and unitary.
 \item[$(\mr{iv})$] $BJ=JB$; in particular, we have that $J(D(B))=D(B)$.
% \item[$(\mr{iv})$] $\overline{AB}=\overline{BA}$, $AJ=JA$ and $BJ=JB$. %In parti\-cular, we have that $D(AJ)=D(JA)=D(A)$ and $D(BJ)=D(JB)=D(B)$.
\end{itemize}
and also
\begin{itemize}
 \item[$(\mr{v})$] $\overline{AB}=\overline{BA}$ and $AJ=JA$; in particular, $J(D(A))=D(A)$.
\end{itemize}
The operators $A$ and $B$ are uniquely determined by $T$ by means of the equalities
\[
\textstyle
\text{$A=\overline{(T+T^*)\frac{1}{2}}\;\;$ and $\;\;B=\left|\,\overline{(T-T^*)\frac{1}{2}}\,\right|$,}
\]
and the operator $J$ is uniquely determined by $T$ on $\mr{Ker}(\,\overline{T-T^*}\,)^\perp$. More precisely, if $A':D(A') \to \sH$, $B':D(B') \to \sH$ and $J' \in \gB(\sH)$ are operators satisfying conditions (i)-(iv) with $A'$, $B'$ and $J'$ in place of $A$, $B$ and $J$ respectively, then $A'=A$, $B'=B$ and $J'u=Ju$ for every $u \in \mr{Ker}(\,\overline{T-T^*}\,)^\perp$.
\end{theorem}

\begin{proof}
Let $\mc{Q}=(Q,\LL)$ be an iqPVM over some $\C^+_\imath$ in $\sH$ such that $T=\int_{\C^+_\imath} \mi{id}\diff\mc{Q}$, whose existence was proven in Theorem \ref{spectraltheoremU}. Define $J:=\LL(\imath)$, $A:=\int_{\C^+_\imath}\alpha\diff\mc{Q}$ and $B:=\int_{\C^+_\imath}\beta\diff\mc{Q}$, where $\alpha$ and $\beta$ denote the functions $\C^+_\imath \ni \alpha+\imath\beta \mapsto \alpha$ and $\C^+_\imath \ni \alpha+\imath\beta \mapsto \beta$, respectively. Properties $(\mr{i})$-$(\mr{v})$ follows immediately from Theorem \ref{TEO1unb} and equalities \eqref{ccq}, \eqref{normunb2} and \eqref{eq:closure-composition}. By \eqref{eq:closure-sum} and Lemma \ref{lem:J}$(\mr{b})$, it follows also that
\[
{\textstyle\overline{(T+T^*)\frac{1}{2}}}=\int_{\C^+_\imath}\alpha\diff\mc{Q}=A
\]
and
\[
{\textstyle\left|\,\overline{(T-T^*)\frac{1}{2}}\,\right|}=\int_{\C^+_\imath}|\imath\beta|\diff\mc{Q}=\int_{\C^+_\imath}\beta\diff\mc{Q}=B.
\]

Suppose now that there exist operators $A':D(A') \to \sH$, $B':D(B') \to \sH$ and $J' \in \gB(\sH)$ satisfying properties $(\mr{i})$-$(\mr{iv})$ with $A'$, $B'$ and $J'$ in place of $A$, $B$ and $J$ respectively. By Remark 2.16$(\mr{v})$ of \cite{GhMoPe}, we know that $(J'B')^*=(B')^*(J')^*=-B'J'$. It follows that $(J'B')^*=-J'B'$, $T^*=A'-J'B'$ and hence $(T+T^*)\frac{1}{2} \subset A'$ and $(T-T^*)\frac{1}{2} \subset J'B'$. Since the operators $A'$ and $B'$ are self-adjoint, we have that they are also closed, $J'B'$ is closed as well and hence $A={\textstyle\overline{(T+T^*)\frac{1}{2}}} \subset A'$ and $\int_{\C^+_\imath}\imath\beta\diff\mc{Q}={\textstyle\overline{(T-T^*)\frac{1}{2}}} \subset J'B'$. Furthermore, bearing in mind Remark 2.16$(\mr{i})$ of \cite{GhMoPe} and above Theorem \ref{TEO1unb}$(\mr{c})$, we infer that $A=A'$ and $\int_{\C^+_\imath}\imath\beta\diff\mc{Q}=J'B'$. Indeed, it holds $A \subset A'=(A')^* \subset A^*=A$ and
\[
\int_{\C^+_\imath}\imath\beta\diff\mc{Q} \subset J'B'=-(J'B')^* \subset -\left(\int_{\C^+_\imath}\imath\beta\diff\mc{Q}\right)^*=-\int_{\C^+_\imath}\overline{\imath\beta}\diff\mc{Q}=\int_{\C^+_\imath}\imath\beta\diff\mc{Q}.
\]
Since $(J'B')^*(J'B')=-B'J'J'B'=(B')^2$ and $B'$ is assumed to be positive, Lemma \ref{lem:J} implies that
\[
B'=|J'B'|=\left|\int_{\C^+_\imath}\imath\beta\diff\mc{Q}\right|=\int_{\C^+_\imath}|\imath\beta|\diff\mc{Q}=\int_{\C^+_\imath}\beta\diff\mc{Q}=B.
\]
We can now apply the same argument used in the proof of Proposition \ref{prop:unb-sp-teo}. Observe that, being $B'=B$ and $J'B=J'B'=\int_{\C^+_\imath}\imath\beta\diff\mc{Q}=L_\imath\int_{\C^+_\imath}\beta\diff\mc{Q}=JB$, it holds that $J'=J$ on $\Ran(B)$. Applying Proposition \ref{propLu}$(\mr{b})$ to $B$, we obtain that $\Ker(B)=\Ran(B)^\perp$ and hence $\overline{\Ran(B)}=\Ker(B)^\perp$. Thanks to the continuity of $J'$ and $J$, it follows that $J'=J$ on $\overline{\Ran(B)}=\Ker(B)^\perp$. Finally, by Corollary \ref{corollario1}, we have that $\Ker(B)=\Ker\big(\big|\,\overline{T-T^*}\,\big|\big)=Q(\{2\beta=0\})(\sH)=Q(\{2\beta\imath=0\})(\sH)=\Ker\big(\,\overline{T-T^*}\,\big)$. The proof is complete.
\end{proof}

%%%%%%%

\end{document}